%% file: lambda-statistics.tex
\documentclass[11pt,a4paper]{amsart}
\usepackage[foot]{amsaddr}
\usepackage{fullpage}

\usepackage[T1]{fontenc}
\usepackage[utf8]{inputenc}
\usepackage[british]{babel}

\usepackage{enumerate}

\usepackage{url}
\usepackage{breakurl}
\usepackage[breaklinks]{hyperref}
\usepackage[capitalize,nameinlink,noabbrev]{cleveref}

\usepackage{xcolor}
\definecolor{lgreen}{rgb}{0.0, 0.48, 0.0}
\definecolor{lpurple}{rgb}{0.48, 0.0, 0.48}
\definecolor{bblue}{rgb}{0.2, 0.4, 0.8}
\hypersetup{linktocpage,
            colorlinks=true,
            linkcolor=lgreen,
            citecolor=lpurple,
            linktoc=true}

\makeatletter
\renewcommand{\tocsection}[3]{%
  \indentlabel{\@ifnotempty{#2}{\bfseries\ignorespaces#1 #2\quad}}\bfseries#3}
\renewcommand{\tocsubsection}[3]{%
  \indentlabel{\@ifnotempty{#2}{\ignorespaces#1 #2\quad}}#3}
\makeatother

\usepackage{floatrow}
\definecolor{bblue}{rgb}{0.2, 0.4, 0.8}
\definecolor{bgreen}{rgb}{0.2, 0.6, 0.4}
\definecolor{bred}{rgb}{0.8, 0.4, 0.2}
\definecolor{bviolet}{rgb}{0.7, 0.2, 0.7}
\definecolor{blackred}{rgb}{0.6, 0.3, 0.3}
\definecolor{blackblue}{rgb}{0.3, 0.3, 0.6}

\usepackage{tikz}
\usetikzlibrary{fit,arrows,trees,shapes,shapes.geometric,calc,matrix}
\tikzset{
  treenode/.style = {align=center, inner sep=0pt, text centered,
    font=\sffamily},
  arn_nn/.style = {treenode, circle, bblue, draw=bblue,
    fill=bblue!10,
    minimum width=0.5em, minimum height=0.5em
},
  arn_n/.style = {treenode, circle, bblue, draw=bblue,
    text width=1.5em, very thick,
    fill=bblue!10},
  arn_g/.style = {treenode, circle, bgreen, draw=bgreen,
    text width=1.5em, very thick,
    fill=bblue!10},
  arn_r/.style = {treenode, circle, bviolet, draw=bviolet,
    text width=1.5em, very thick,
    fill=bviolet!10},
  arn_x/.style = {treenode, triangle, draw=black,
    minimum width=0.5em, minimum height=0.5em},
  triangle/.style = {treenode, bred, draw=bred, fill=bred!20, regular polygon, regular polygon
    sides=3, very thick, text width=1.5em },
  triangle_b/.style = {treenode, bblue, draw=bblue,
    fill=bblue!20, regular polygon, regular polygon
    sides=3, very thick, text width=1.5em },
  triangle_g/.style = {treenode, bgreen, draw=bgreen,
    fill=bgreen!20, regular polygon, regular polygon
    sides=3, very thick, text width=1.5em },
  triangle_v/.style = {treenode, bviolet, draw=bviolet,
    fill=bviolet!20, regular polygon, regular polygon
    sides=3, very thick, text width=1.5em },
  triangle_h/.style = {treenode, bblue, draw=bblue,
    fill=gray!20, regular polygon, regular polygon
    sides=3, very thick, text width=1.5em },
  arn_e/.style = {treenode, blackblue, draw=blackblue,
    fill=bblue!10, circle,
    very thick, text width=1.5em },
  arn_w/.style = {treenode, black, draw=black,
    fill=white, circle,
    densely dashed, thick, text width=1.5em }
}

\usepackage{eqnarray,amstext}
\usepackage{qsymbols,xparse}
\usepackage{amsmath,aliascnt,mathtools}
\usepackage{mathrsfs}

\theoremstyle{definition}
\newtheorem{theorem}{Theorem}
\numberwithin{theorem}{section}

\newtheorem{prop}[theorem]{Proposition}

\newtheorem{lemma}[theorem]{Lemma}
\newtheorem{definition}[theorem]{Definition}
\newtheorem{example}[theorem]{Example}
\newtheorem{remark}[theorem]{Remark}

\renewcommand\autoref\Cref

\renewcommand\doteq\approx

\usepackage{multirow}

\let\geq\geqslant
\let\leq\leqslant

\usepackage{draftwatermark}
\SetWatermarkLightness{0.98}
\SetWatermarkScale{4}

\newcommand{\set}[1]{\{#1\}}
\newcommand{\seq}[1]{\left(#1\right)}

\def\vec{\boldsymbol}

\newcommand{\idx}[1]{\mbox{\underline{\sf #1}}}
\newcommand{\classL}[1][\infty]{\mathcal{L}_{ #1 }}
\newcommand{\deriv}[2]{\frac{\partial}{\partial #1} #2}

\newcommand{\at}[2][]{#1|_{#2}}

\newcommand{\lterm}{$`l$\nobreakdash-term}
\newcommand{\lterms}{$`l$\nobreakdash-terms}
\newcommand{\lcalculus}{$`l$\nobreakdash-calculus}

\newcommand{\bredex}{$`b$\nobreakdash-redex}
\newcommand{\bredexes}{$`b$\nobreakdash-redexes}
\newcommand{\breduction}{$`b$\nobreakdash-reduction}

\DeclareMathOperator*{\Geom}{\mbox{\sc Geom}}

\usepackage{caption, subcaption}

\numberwithin{equation}{section}

\begin{document}

\title{Statistical properties of lambda terms}
\author{Maciej Bendkowski$^\dagger$}
\address{
    $^\dagger$Theoretical Computer Science Department\\
  Faculty of Mathematics and Computer Science\\
  Jagiellonian University\\
  ul. Prof. {\L}ojasiewicza 6, 30-348 Krak\'ow, Poland.}
\email{maciej.bendkowski@tcs.uj.edu.pl}

\author{Olivier Bodini$^\ddagger$}
\address{
  $^\ddagger$Institut Galilée\\
  Université Paris 13\\
  99 Avenue Jean Baptiste Clément 93430\\
  Villetaneuse, France.
}
\email{\{olivier.bodini, dovgal\}@lipn.univ-paris13.fr}

\author{Sergey Dovgal$^{\ddagger \mathsection \mathparagraph}$}
\address{
  $^\mathsection$Institut de Recherche en Informatique Fondamentale\\
  Université Paris 7\\
  5 Rue Thomas Mann 75013\\
  Paris, France.
}

\address{
$^\mathparagraph$Moscow Institute of Physics and Technology\\
Institutskiy per. 9\\
Dolgoprudny, Russia 141700.
}

\date{\today}
\thanks{Maciej Bendkowski was partially supported within the Polish National
Science Center grant 2016/21/N/ST6/01032 and the French Government Scholarship
within the French-Polish POLONIUM grant number 34648/2016.  Olivier Bodini and
Sergey Dovgal were supported by the French project ANR project MetACOnc,
ANR-15-CE40-0014.}

\begin{abstract}
We present a quantitative, statistical analysis of random lambda terms in the
    de~Bruijn notation. Following an analytic approach using multivariate
    generating functions, we investigate the distribution of various
    combinatorial parameters of random open and closed lambda terms, including
    the number of redexes, head abstractions, free variables or the de~Bruijn
    index value profile. Moreover, we conduct an average-case complexity
    analysis of finding the leftmost-outermost redex in random lambda terms
    showing that it is on average constant. The main technical ingredient of our
    analysis is a novel method of dealing with combinatorial parameters inside
    certain infinite, algebraic systems of multivariate generating functions.
    Finally, we briefly discuss the random generation of lambda terms following
    a given skewed parameter distribution and provide empirical results
    regarding a series of more involved combinatorial parameters such as the
    number of open subterms and binding abstractions in closed lambda terms.
\end{abstract}

\maketitle

\tableofcontents

\input{1-introduction.tex}

\input{2-preliminaries.tex}

\input{3-basic-statistics.tex}

\input{4-empirical-results.tex}

\input{5-infinite-systems.tex}

\input{6-advanced-marking.tex}

\input{7-conclusions.tex}

\bibliographystyle{plain}
\bibliography{lambda-statistics}
\end{document}

%% file: 1-introduction.tex
\section{Introduction}\label{sec:introduction}

Lambda calculus (often abbreviated to \lcalculus) is a functional calculus
proposed by Alonzo Church in the 1930s as an alternative foundation of
mathematics. Although the initial plan failed, due to logical inconsistencies
discovered in Church's naive system, it was quickly realised that
\lcalculus~itself is able to elegantly capture the, by then still informal,
notion of computability, see~\cite{DBLP:series/hhl/CardoneH09}. Nowadays,
\lcalculus~is considered not only as an important theoretical model of
computation, but is also used in practical applications ranging from functional
programming languages~\cite{peytonJones1987}, including the evaluation and
testing of functional programming language
compilers~\cite{Claessen-2000,palka2012}, to automated theorem
provers~\cite{Bertot:2010:ITP:1965123}.

Despite the extensive use of \lterms~(i.e.~formal expressions of \lcalculus) as
computations in functional programming languages or as components of proof
artifacts in various automated theorem provers, quantitative investigations into
the combinatorial or statistical properties of \lterms~were initialised only
quite recently. Motivated by the uniformly random (conditioned on size)
generation of \lterms, Wang~\cite{Wang05generatingrandom} explored a
combinatorial model of \lcalculus~where $`a$\nobreakdash-convertible
\lterms~(i.e.~terms identical up to bound variable names) are considered
equivalent. The central problem of providing asymptotic estimates on the number
of \lterms~in this model remained, however, open.  Some time later, David et
al.~\cite{dgkrtz} investigated a similar model of \lcalculus~where variables do
not contribute to the term size and showed that asymptotically almost all
\lterms~are strongly normalising. In other words, the fraction of \lterms~for
which all evaluation strategies terminate approaches one as the term size tends
to infinity. Likewise, in this model the central problem of giving accurate
estimates on the number of \lterms~of size $n$ remained open. Enumeration problems for
restricted classes of closed affine and linear \lterms, where binders capture at
most and exactly one variable, respectively, were investigated by Bodini, Gardy,
Jacquot and
Gittenberger~\cite{BGGJ2013,BODINI2013227,doi:10.1137/1.9781611973204.3}.  The
class of \lterms~with restricted unary height was considered by Bodini, Gardy
and Gittenberger in~\cite{doi:10.1137/1.9781611973013.3}.

The canonical models of Wang and David et al.~pose considerable difficulties due
to the global, intractable structure of binders (abstractions) and their
associated variables, all considered modulo $`a$\nobreakdash-equivalence.
Explicit variable names, though elegant for manual manipulation, introduce also
problems with substitution of terms for variables, especially when computations
in \lcalculus~are meant to be automatised. For the latter purpose, de~Bruijn
proposed an alternative notation of \lterms, involving non-negative indices
instead of explicit variable names~\cite{deBruijn1972}. This notation was later
adopted by Lescanne~\cite{Lescanne2013,grygiel_lescanne_2013} who proposed a new
combinatorial representation for the enumeration of \lterms. Within this new
representation, \lterms~represent entire $`a$\nobreakdash-equivalence classes in
the former models. Consequently, it became possible to enumerate also open terms
(i.e.~containing free variables) not just closed ones. Let us also remark that
independently, a different combinatorial model was proposed by Tromp who
considered a binary encoding of \lcalculus~meant for the construction of a
compact and efficient self-interpreter with applications to Kolmogorov
complexity~\cite{tromp2006}.  Enumeration problems related to the binary
\lcalculus, as well as the effective random generation of \lterms, were later
studied by Grygiel and Lescanne~\cite{grygiel_lescanne_2015}.

Investigations into quantitative properties of \lterms~in the de~Bruijn notation
were continued by Bendkowski et al.~\cite{Bendkowski2016,BendkowskiGLZ16} who
showed that, in contrast to the canonical representation of David et al.,
asymptotically almost all \lterms~are not strongly normalising. In other words,
the proportion of terms for which all evaluation strategies terminate approaches
zero as the term size tends to infinity. Various size models based on the de
Bruijn notation, such as Tromp’s binary encoding or the so-called natural size
notion introduced by Bendkowski et al.~were later generalised in a common
framework by Gittenberger and Gołębiewski who provided tight lower and upper
asymptotic bounds on the number of closed
\lterms~\cite{gittenberger_et_al:LIPIcs:2016:5741}. Recently, the gap between
both the lower and bounds was closed by Bodini, Gittenberger and
Gołębiewski~\cite{BodiniGitGol17}. Subsequently, efficient sampling methods for
closed terms were developed and the enumeration of closed \lterms~was finally
completed.

In the current paper we propose to deepen the quantitative analysis of
\lcalculus~in the de~Bruijn notation, offering a detailed statistical analysis
of random \lterms. We investigate the distribution of  several combinatorial
parameters related to plain (i.e.~unrestricted) and closed
\lterms.~\autoref{table:basic:statistics} provides a brief overview of the
obtained results.

\begin{table}[hbt!]
    \caption{Comparison of obtained statistics for random \lterms. Listed constants are
    approximated up to the third decimal point. See~\autoref{subsec:lambda:calculus}
    for details on {\sf trivial}.}
    \label{table:basic:statistics}
    \centering
    \begin{tabular}{ccccc}
        \multirow{2}{*}{\textbf{Parameter}}
        &
        \multicolumn{2}{c}{
            \textbf{Mean, $\sim$}
        }
        &
        \multicolumn{2}{c}{
            \textbf{Distribution}
        }
        \\
        \cline{2-3}
        \cline{4-5}
        &
        \textbf{plain}
        &
        \textbf{closed}
        &
        \textbf{plain}
        &
        \textbf{closed}
        \\
        \hline
        Variables
        &
        \multicolumn{2}{c}{$0.307n$}
        &
        \multicolumn{2}{c}{Normal}
        \\
        Abstractions
        &
        \multicolumn{2}{c}{$0.258n$}
        &
        \multicolumn{2}{c}{Normal}
        \\
        Successors
        &
        \multicolumn{2}{c}{$0.129n$}
        &
        \multicolumn{2}{c}{Normal}
        \\
        Redexes
        &
        \multicolumn{2}{c}{$0.091n$}
        &
        \multicolumn{2}{c}{Normal}
        \\
        Index value
        &
        \multicolumn{2}{c}{0.420}
        &
        \multicolumn{2}{c}{Geometric}
        \\
        Redex search time
        &
        6.222 & 6.054
        &
        Discrete & Discrete
        \\
        Head abstractions
        &
        0.420 & 1.447
        &
        Geometric & Discrete
        \\
        \hline
        $m$-openness
        &
        2.019 & 0
        &
        Discrete & {\sf trivial}
        \\
        Free variables
        &
        5.722 & 0
        &
        Discrete & {\sf trivial}
        \\
        Unary height profile
        &
        \multicolumn{2}{c}{$ 0.122 \sqrt{n} $}
        &
        \multicolumn{2}{c}{Rayleigh}
        \\
        Natural height profile
        &
        \multicolumn{2}{c}{$ 0.412 \sqrt{n} $}
        &
        \multicolumn{2}{c}{Rayleigh}
    \end{tabular}
\end{table}

In the current paper, we provide limit laws and asymptotic estimates using
techniques from analytic combinatorics. While plain \lterms~in de~Bruijn size
notion can be analysed using classical methods, the respective analysis of
closed \lterms~requires solving infinite systems of algebraic equations.  Let us
mention that an earlier paper by Drmota, Gittenberger and
Morgenbesser~\cite{infinitesystems} deals with infinite algebraic systems which
are strongly connected and whose Jacobian can be represented as a sum of a
scaled identity operator and an operator whose power is compact.
Here, we develop a general tool meant for analysis of infinite
algebraic systems that resemble in structure systems for closed \lterms, however
do not fit into the framework of Drmota, Gittenberger and Morgenbesser.  In
this context, our result can be considered as a continuation
of~\cite{infinitesystems}.

The paper is structured as follows. In~\autoref{sec:preliminaries} we provide a
concise presentation of preliminary notions and techniques. In particular, we
discuss the de~Bruijn representation of
\lterms~(\autoref{subsec:lambda:calculus}) and introduce the utilised analytic
toolbox (\autoref{subsec:analytic:tools}). We then continue with a fairly
standard analysis of basic statistics corresponding to plain
\lterms~(\autoref{sec:basic:statistics}). Next, we provide an empirical
evaluation of several statistical properties corresponding to plain, closed, and
so-called $h$\nobreakdash-shallow \lterms, i.e.~terms with de~Bruijn indices
whose value does not exceed $h$ (\autoref{sec:empirical:results}).  We give
empirical histograms and relate the discovered distributions with considered
term types, exhibiting some intriguing correlations. In the next section we
develop our main technical tool for investigating combinatorial parameters in
closed \lterms~(\autoref{sec:infinite:systems}). In the subsequent section we
use our advanced marking techniques and study various parameters in closed
\lterms~(\autoref{sec:advanced:marking}).  Finally, we conclude the paper with
remarks and open questions (\autoref{sec:conclusions}).

%% file: 2-preliminaries.tex
\section{Preliminaries}\label{sec:preliminaries}

\subsection{Lambda calculus}\label{subsec:lambda:calculus}
\lcalculus~is a theoretical formalism famously equivalent in expressiveness to
Turing machines, see~\cite{barendregt1984}.  In this calculus, computations are
represented as \emph{\lterms} defined by the formal grammar $T ::= x~|~(`lx.
T)~|~(T~T)$ in which $x$ belongs to the countable, infinite alphabet of \emph{variables};
$(`lx. T)$ is an \emph{abstraction} of variable $x$ in $T$; and $(T~T)$ denotes
an \mbox{\emph{application}} of two \lterms. Given an abstraction $(`lx.T)$,
occurrences of $x$ in $T$ are said to be \emph{bound}. Unbound variable
occurrences are said to occur \emph{freely}.

Lambda terms, intended to represent anonymous functions, are executed by means
of the iterated process of \emph{$`b$-reduction}. First, an arbitrary
\emph{$`b$-redex} subterm in form of $(`lx.  N) M$ is selected (if no such
subterm exists, computations are terminated). Next, the selected $`b$-redex is
replaced with $N[x := M]$, i.e.~$N$ in which each occurrence of \( x \) is
substituted, in a \emph{capture-avoiding} manner, by \( M \). While substituting
$M$ for $x$ in $N$ we have to avoid the unintended situation in which free
variable occurrences in $M$ get bound, in other words \emph{captured}, by some
abstractions occurring in $N$.  For instance, let \( N = (\lambda y.x) \) and \(
M = y \). The term \( (\lambda x.N)M \) should not be reduced to $`ly.y$ as, by
doing so, the free variable occurrence $y$ gets bound due to a coincidental
\emph{clash} with the inner abstraction variable name. Certainly, the arbitrary choice of
the formal variable name $y$ should not influence the intended semantics of the
represented computation.  Following this motivation, \lterms~differing only in
bound variable names are considered equivalent (in other words
\emph{$`a$-convertible}). In order to avoid potential name clashes, we can
therefore \emph{rename} bound variable occurrences before proceeding with
$`b$-reduction.  Since there is an infinite supply of available variable names,
it is always possible to avoid variable captures.  Consequently, we can
equivalently $`a$-convert $(`lx.`ly.x) y$ into, say, $(`lx.`lw.x) y$ and proceed
with $(`lx.`lw.x) y \rightarrow_{`b} (`lw. x)[x := y] = `lw. y$.

Though intuitive, explicit variable names pose considerable conceptual and
implementation problems. For instance, consider the terms $`lx.x$ and $`ly.y$.
Although syntactically different, both semantically represent the same anonymous
identity function as $(`lx.x)T \to_{`b} T$ and $(`ly.y)T \to_{`b} T$ for
arbitrary $T$.  In order to facilitate
automatic computations, de~Bruijn proposed a different notation for
\lterms~eliminating in effect the troublesome variable
names~\cite{deBruijn1972}. In his notation, variable occurrences are replaced
with \emph{indices} represented as non-negative integers.  The intention is to
view \lterms~as natural tree-like structures and encode variable occurrences as
indices denoting their relative distance to respective variable binders -- each
index $\idx{n}$ denotes  a variable occurrence $x$ whose relative distance to
its binder (in terms of passed lambda symbols) is equal to $n+1$. For instance,
$\idx{0}$ corresponds to a variable occurrence bound to the nearest abstraction
on its unique path to the root in the associated tree-like representation of the
considered \lterm.  Consequently, $`a$-convertible \lterms~have the same
de~Bruijn representation. In effect, each \lterm~in the de~Bruijn notation
represents an entire $`a$-equivalence class of \lterms~in the classic variable
notation. For instance, both $`lx.x$ and $`ly.y$, being $`a$-convertible, are
represented as $`l \idx{0}$ in the de~Bruijn notation.  The use of de~Bruijn
indices significantly simplifies the automatic substitution operation.  Due to
the lack of explicit variable names, variable captures and name clashes do not
pose implementation issues.

\begin{remark}
There exists a disagreement in the literature whether to start
    de~Bruijn indices with $\idx{0}$ or $\idx{1}$. Although de~Bruijn himself
    assumed the latter~\cite{deBruijn1972}, some authors follow his footsteps,
    see e.g~\cite{grygiel_lescanne_2013,grygiel_lescanne_2015} whereas others
    do not, including $\idx{0}$ in the set of admissible indices, see
    e.g.~\cite{gittenberger_et_al:LIPIcs:2016:5741,BendkowskiGLZ16,bodinitarau2017}.
    Certainly, neither convention is better than the other. In the current
    paper, we follow the convention of starting de~Bruijn indices with $\idx{0}$
    so the keep consistent with the most recent literature.
\end{remark}

\begin{definition}\label{def:lambda:terms}
    Let $\set{\idx{0},\idx{1},\ldots}$ be an infinite, denumerable set of
    available indices. Then, the set $\classL$ of \emph{\lterms~in the de~Bruijn
    notation} is defined inductively as follows:
    \begin{enumerate}
        \item Each index $\idx{n}$ is a \lterm;
        \item If $N$ and $M$ are \lterms, then $(N M)$ is a \lterm;
        \item If $N$ is a \lterm, then $(`l N)$ is a \lterm.
    \end{enumerate}

    Following usual notational conventions, we omit outermost parentheses and
    drop parentheses from left-associated \lterms. For instance, $`l x. `l y. `l
    z. ((x y) z)$ in the classical variable notation is written as $`l`l`l
    \idx{2} \idx{1} \idx{0}$.

    An index occurrence $\idx{n}$ is said to be \emph{bound} in the term $N$ if
    there exist at least $n+1$ lambda symbols on the unique path from $\idx{n}$
    to the root of the associated tree-like representation of $N$, see
    e.g.~\autoref{fig:example:lambda:tree}. Otherwise, $\idx{n}$ is said to be
    occurring \emph{freely} in $N$ and hence corresponds to a free variable in
    the classical \lcalculus~notation. For convenience, we refer to de~Bruijn
    indices both as indices and variables. If each index occurrence in $N$ is
    bound, then $N$ is said to be \emph{closed}.  Otherwise, it is said to be
    \emph{open}. And so, $`l`l`l \idx{2} \idx{1} \idx{0}$ is closed whereas
    $`l`l\idx{2} \idx{1}$ is not as here $\idx{2}$ is not bound.  If prepending
    $N$ with $m$ lambdas turns it into a closed \lterm, then $N$ is said to be
    \emph{$m$-open}.  Certainly, if $N$ is $m$-open, then it is also
    $(m+1)$-open. Moreover, $0$-open \lterms~correspond exactly to closed
    \lterms.  Hence, though $`l`l\idx{2} \idx{1}$ is not closed, it is $1$-open
    as $`l`l`l\idx{2} \idx{1}$ is a closed \lterm.  Finally, we write that a
    \lterm~is \emph{plain} if we mean to denote that it is either open or
    closed, without specifying which case holds.
    \begin{figure}[ht!]
    \centering
\begin{subfigure}{.45\textwidth}
    \centering
\begin{tikzpicture}[>=stealth',level/.style={thick}]
\draw
node[arn_r](A){\( \lambda x \)}
child[level distance=1cm]{
node[arn_r](B){\( \lambda y \)}
child[level distance=1cm]{
node[arn_r](C){\( \lambda z \)}
child[level distance=1cm]{
    node[arn_n]{ \( @ \) }
    child[level distance=1cm]{
        node[arn_n]{ \( @ \) }
        child[sibling distance=.8cm]{
            node[arn_g](D){ $x$ }
        }
        child[sibling distance=.8cm]{
            node[arn_g](E){ $z$ }
        }
    }
    child[level distance=1cm]{
        node[arn_n]{ \( @ \) }
        child[sibling distance=.8cm]{
            node[arn_g](F){ $y$ }
        }
        child[sibling distance=.8cm]{
            node[arn_g](G){ $z$ }
        }
    }
}
}
};
\path[->] (D) edge [bred,dashed,thick,bend left=48] node {} (A);
\path[->] (E) edge [bred,dashed,thick,bend left=75] node {} (C);
\path[->] (F) edge [bred,dashed,thick,bend right=70] node {} (B);
\path[->] (G) edge [bred,dashed,thick,bend right=60] node {} (C);
\end{tikzpicture}
\end{subfigure}
\begin{subfigure}{.45\textwidth}
    \centering
\begin{tikzpicture}[>=stealth',level/.style={thick}]
\draw
node[arn_r](A){\( \lambda \)}
child[level distance=1cm]{
node[arn_r](B){\( \lambda \)}
child[level distance=1cm]{
node[arn_r](C){\( \lambda \)}
child[level distance=1cm]{
    node[arn_n]{ \( @ \) }
    child[level distance=1cm]{
        node[arn_n]{ \( @ \) }
        child[sibling distance=.8cm]{
            node[arn_g](D){ $\idx{2}$ }
        }
        child[sibling distance=.8cm]{
            node[arn_g](E){ $\idx{0}$ }
        }
    }
    child[level distance=1cm]{
        node[arn_n]{ \( @ \) }
        child[sibling distance=.8cm]{
            node[arn_g](F){ $\idx{1}$ }
        }
        child[sibling distance=.8cm]{
            node[arn_g](G){ $\idx{0}$ }
        }
    }
}
}
};
\path[->] (D) edge [bred,dashed,thick,bend left=48] node {} (A);
\path[->] (E) edge [bred,dashed,thick,bend left=75] node {} (C);
\path[->] (F) edge [bred,dashed,thick,bend right=70] node {} (B);
\path[->] (G) edge [bred,dashed,thick,bend right=60] node {} (C);
\end{tikzpicture}
\end{subfigure}
\caption{Two tree-like representations associated with the same example \lterm~--- $`lx.`ly.`lz. x z (y
z)$ and its de~Bruijn notation variant $`l`l`l \idx{2} \idx{0} (\idx{1}
\idx{0})$. Back pointers to abstractions are included for illustrative purposes
only.}\label{fig:example:lambda:tree}
\end{figure}
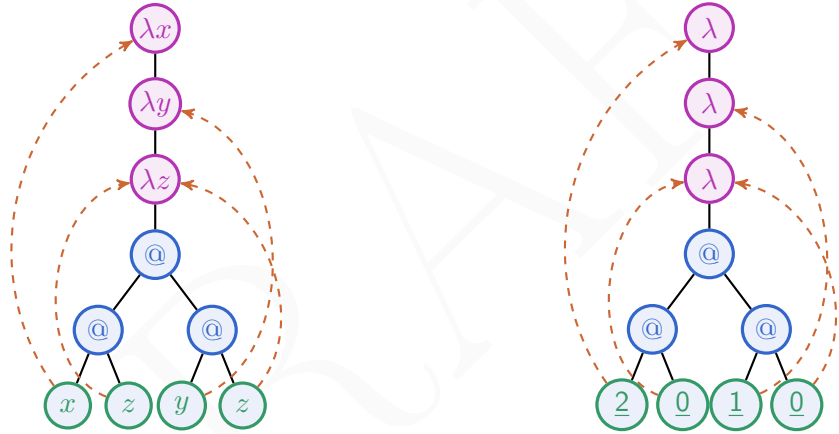
\end{definition}

\subsubsection{Enumeration}
In the current paper we
follow~\cite{Bendkowski2016,BendkowskiGLZ16,gittenberger_et_al:LIPIcs:2016:5741,BodiniGitGol17}
and investigate the statistical properties of random \lterms~in the de~Bruijn
representation. We assume a unary base encoding of indices, i.e.~an encoding in
which $\idx{n}$ is identified with an $n$-fold application of the
\emph{successor} operator ${\sf S}$ to zero. Formally, the set $\classL$ of
\lterms~is described by the following formal grammar:
\begin{align}\label{eq:plain:terms:grammar}
    \begin{split}
        \classL &::= \idx{n}~|~(`l \classL)~|~(\classL \classL)\\
        \idx{n} &::= {\sf 0}~|~{\sf S}~\idx{n} .
    \end{split}
\end{align}
In order to enumerate \lterms, we have to assign a formal notion of \emph{size}
to each term in such a manner that for each available size $n$ the number of terms
of size $n$ is finite. Though various size measures are considered in the
literature, most notably the general size model framework of Gittenberger and
Gołębiewski~\cite{gittenberger_et_al:LIPIcs:2016:5741}, we assume the simple
\emph{natural size notion}~\cite{Bendkowski2016} in which the size of $T$ is
equal to the total number of abstractions, applications, successors and zeros
occurring in $T$. Formally, we define the size of $T$ inductively as follows:
\begin{equation}
  \begin{array}{l@{\hspace*{50pt}}l}
    \begin{array}{lcl}
        | \mathsf{0} | &=& 1\\
        | \mathsf{S}~ \idx{n}| &=& |\idx{n}| + 1
    \end{array}
    &
    \begin{array}{lcl}
      |M \, N| &=& |M| + |N| + 1\\
      |`l M | &=& |M| + 1 .
    \end{array}
  \end{array}
\end{equation}

\begin{example}
    Note that, in general, $\idx{n}$ is of size $n+1$ as it consists of $n$
    successors applied to zero. Consequently, the term $`l`l`l \idx{2} \idx{1}
    \idx{0}$ is of size $11$ as it consists of three $`l$ symbols, two
    applications between \( \idx{2}, \idx{1} \) and \( \idx{0} \), by convention
    omitted in writing, and indices $\idx{0},\idx{1},\idx{2}$ of total size six.
\end{example}

\begin{remark}
It is worth noticing that, with some minor technical overhead, the analysis
    presented in the current paper extends onto the more general size model
    framework of Gittenberger and
    Gołębiewski~\cite{gittenberger_et_al:LIPIcs:2016:5741} including the assumed
    natural size notion as a special case.  We prefer to avoid technicalities
    related to the general size notion and so, for the reader's convenience,
    favour a lucid presentation of the key arguments.
\end{remark}

Let $\classL[m]$ denote the set of $m$-open \lterms,
see~\autoref{def:lambda:terms} (plain terms can be viewed as ``infinitely'' open,
hence the $\infty$ symbol in the subscript of $\classL$).  Like plain
\lterms~\eqref{eq:plain:terms:grammar}, $\classL[m]$ can be described in terms of
a formal, though now infinite, grammar as follows:
\begin{align}\label{eq:m-open:terms:grammar}
    \begin{split}
        \classL[m] &::= (`l \classL[m+1])~|~(\classL[m] \classL[m])~|~ \idx{0},
        \idx{1}, \ldots, \idx{m - 1}\\
        \classL[m+1] &::= (`l \classL[m+2])~|~(\classL[m+1] \classL[m+1])~|~ \idx{0},
        \idx{1}, \ldots, \idx{m}\\
        \ldots & \qquad \ldots
    \end{split}
\end{align}
An $m$-open \lterm~$T$ can take one of the three forms. Either $T$ is in the
form of abstraction followed by an $(m+1)$-open \lterm; or it is an application
of two $m$-open \lterms; or, finally, $T$ is one of the indices
$\idx{0},\idx{1},\ldots,\idx{m-1}$.

Due to the infinite combinatorial specification~\eqref{eq:m-open:terms:grammar}
for $\classL[m]$ standard analytic combinatorics techniques are not readily
applicable. Consequently, enumerating closed \lterms~poses a considerable
challenge. In~\cite{gittenberger_et_al:LIPIcs:2016:5741} a partial solution
bounding the asymptotic growth rate of the number of $m$-open \lterms~of size
$n$ was proposed.  Although both the lower and upper bounds were of the form $C
\rho^n n^{-3/2}$, a typical trait of various tree-like structures, the
asymptotic growth rate of $m$-open terms remained open. Remarkably, some time
later in their joint paper~\cite{BodiniGitGol17} Bodini, Gittenberger and
Gołębiewski closed the remaining gap and confirmed the conjectured $C \rho^n
n^{-3/2}$ form of the asymptotic growth of $m$-open \lterms. Furthermore, two
combinatorial problems related to random closed \lterms~were studied.
Specifically, the number of terms with an \emph{a priori} fixed number of
abstractions and the number of terms in $`b$-normal form, i.e.~without
$`b$-redexes. In this context, our contribution is a natural continuation of
their work. In addition, we offer a different proof of the asymptotic growth
rate of $m$-open \lterms.

\subsection{Analytic tools}\label{subsec:analytic:tools}
In the following subsection we present standard, analytic combinatorics tools
following the exposition of Flajolet and Sedgewick~\cite{flajolet09}.  We also
assume conventional notation corresponding to generating functions, their
coefficients and asymptotic expansions. We refer the reader
to~\cite{flajolet09,Wilf2006} for a detailed exposition.

For our purposes, combinatorial parameter analysis outlines as follows:
\begin{itemize}
    \item Let \( a_{n,k} \) denote the number of plain (closed) lambda terms of
        size \( n \) for which the investigated parameter takes value \( k \).
        Note that we do not assume that the numbers \( a_{n,k} \) are \emph{a
        priori} known.  With such a two-dimensional sequence of numbers we
        associate a bivariate generating function
        \begin{equation}
            A(z, u) := \sum_{n,k \geq 0} a_{n,k} z^n u^k;
        \end{equation}
        In order to simultaneously study several different parameters of
        interest, we introduce \emph{multivariate generating functions}
        in form of
        \begin{equation}
            A(z, \vec u) = \sum_{n, \vec k \geq 0} a_{n, \vec k} z^n \vec
            u^{\vec k}
        \end{equation}
        where \( \vec u = (u_1, \ldots, u_d) \) is a $d$-dimensional variable,
        \( \vec k = (k_1, \ldots, k_d) \) is a $d$-dimensional index satisfying
        \( k_i \geq 0 \), \( \vec u^{\vec k} := u_1^{k_1} u_2^{k_2} \cdots
        u_d^{k_d} \), and \( a_{n, \vec k} \) denotes the number of plain
        (closed) lambda terms of size $n$ for which the investigated parameter
        values equal \( k_1, k_2, \ldots, k_d \), respectively;
    \item Considered combinatorial parameters (patterns) inside plain or closed
        \lterms~are described in terms of admissible combinatorial
        specifications (sometimes infinite, as in the case of closed terms);
\item So obtained specifications are then converted into systems of equations
    involving multivariate generating functions where additional
        variables \( \vec u = \seq{u_1, u_2, \ldots, u_d} \) \emph{mark} associated
        combinatorial patterns;
\item In the case of plain lambda terms, the resulting systems of equations are
    solved, usually approximately, in terms of standard complex-valued functions
        like \( f(z) = \sqrt{1 - z} \).  The coefficients of associated
        generating functions depend on the marking variables \( \vec u =
        \seq{u_1, u_2, \ldots, u_d} \).  In the case of closed lambda terms,
        novel tools developed in~\autoref{sec:infinite:systems} are applied;
    \item Finally, an application of Flajolet and Odlyzko's transfer theorem
        provides access to probability generating functions of the limiting
        probability distributions.  In consequence, properties of investigated
        combinatorial parameters become readily available.
\end{itemize}

\subsubsection{Asymptotic expansions}
In order to access the asymptotic form of the coefficients of corresponding
generating functions, we view them as functions analytic at the origin of the
complex plane and examine their singularities, in particular so-called
\emph{dominant singularities} of smallest possible absolute value. Typically, at
this point, an analytic continuation of the formal power series outside its
circle of convergence is required. The following, usual domain in which the
function is considered, is called a \emph{delta-domain}.

\begin{prop}[Transfer
    theorem~{\cite[Section~VI.3]{flajolet09}}]\label{proposition:transfer:theorem}
    Suppose that \( f(z/\rho) \) is a function analytic in the so-called
    \emph{delta-domain} \( \Delta(R, \phi) \) for some \( R > 1 \) and
    \( 0 < \phi < \frac{\pi}{2} \), where
    \begin{equation}
        \Delta(R, \phi) = \{
            \zeta \colon |\zeta| < R, \ \zeta \neq 1, \
            \arg(\zeta-1) > \phi
        \}.
    \end{equation}
    Suppose that as \( z \to \rho \), for \( z/\rho \in \Delta(R, \phi) \), it holds
    \begin{equation}
        f(z) = h(z) - g(z)
        \left(
            1 - \dfrac{z}{\rho}
        \right)^{-\alpha}
        +
        O
        \left( \left|
            1 - \dfrac{z}{\rho}
        \right|^{-\beta} \right)
    \end{equation}
    where $\alpha, \beta \in \mathbb{C} \setminus \mathbb{Z}_{\leq 0}$,
    and \( h(z) \) and \( g(z) \) are functions analytic in \( |z| < R \).

    Then, as \( n \to \infty \), the coefficients $[z^n]f(z)$
    admit an asymptotic approximation in form of
    \begin{equation}
        [z^n]f(z)
        \sim
        g(\rho)
        \cdot
        \rho^{-n}
        \cdot
        \dfrac{n^{\alpha - 1}}{\Gamma(\alpha)}
        + O \left(
            \rho^{-n} n^{\beta - 1}
        \right)
    \end{equation}
    where $\Gamma \colon
    \mathbb{C} \setminus \mathbb{Z}_{\leq 0} \to \mathbb{C}$ is the
    Gamma function defined as
    \(
        \Gamma(z) = \int_{0}^{\infty}
        x^{z-1} e^{-x} dx
    \).
\end{prop}
In particular, if for \( z/\rho \in \Delta(R, \phi) \), as \( z \to \rho \),
we have
\begin{equation}
    f_1(z) \sim h_1(z) - g_1(z) \sqrt{1 - \dfrac{z}{\rho}}
    \quad \text{and} \quad
    f_2(z) \sim \dfrac{g_2(z)}{\sqrt{1 - \dfrac{z}{\rho}}}
    ,
\end{equation}
then we obtain, as \( n \to \infty \), the following coefficient expansions:
\begin{equation}
[z^n]f_1(z) \sim \frac{g_1(\rho)\rho^{-n}}{2 \sqrt{\pi} n^{3/2}}
    \quad \text{and} \quad
[z^n]f_2(z) \sim \frac{g_2(\rho)\rho^{-n}}{\sqrt{\pi} n^{1/2}}.
\end{equation}

\begin{prop}[Semi-large powers theorem, {\cite[Theorem IX.16]{flajolet09}}, \cite{banderier2001random}]
    \label{proposition:semilarge:power:theorem}
    Suppose that \( f(z/\rho) \) is a function analytic in delta-domain \(
    \Delta(R, \phi) \), see~\autoref{proposition:transfer:theorem}, for some \( R >
    1 \), and \( f(z) \) admits asymptotic expansion as \( z \to \rho \)
    for \( z/\rho \) staying in \( \Delta(R, \phi) \):
    \begin{equation}
        f(z) \sim 1 - a \sqrt{1 - \dfrac{z}{\rho}}.
    \end{equation}
    Then, for \( x \) in any compact subinterval of \( (0, +\infty) \) the
    coefficient standing by \( z^n \) in \( {f(z)}^k \) admits an asymptotic
    estimate
    \begin{equation}
        [z^n] {f(z)}^k \sim \dfrac{\rho^{-n}}{n}
        S(
            ax
        )
    \end{equation}
    where
    \( S(x) \) is the Rayleigh function satisfying
    \begin{equation}
        S(x) = \dfrac{x e^{-x^2 / 4}}{2 \sqrt \pi}
        \quad \text{and} \quad
        x = \dfrac{k}{\sqrt n}.
    \end{equation}
\end{prop}

\subsubsection{Algebraic systems}
The following theorem, commonly known as the Drmota--Lalley--Woods theorem, is a
fundamental result obtained independently by several
authors~\cite{drmota1997systems,woods1997coloring,lalley1993finite} in order to
establish limit laws in various families of tree structures specified by
context-free grammars. In our exposition, we reference Drmota's
book~\cite[Section 2.2.5]{drmota2009random}, and
the papers~\cite{drmota1997systems,infinitesystems,pivoteau2012algorithms,bell2010characteristic}.

\begin{definition}
    \label{definition:nested:systems}
    Consider a polynomial system of equations
    \begin{equation}
        \vec y = \vec \Phi(z, \vec y, \vec u)
    \end{equation}
    which is a vector notation for \( \seq{y_j = \Phi_j(z, y_1, \ldots, y_m,
    u_1, \ldots, u_d)} \)
    with $j$ ranging over $1, \ldots, m$. Assume that
    \( \vec \Phi(0, \vec 0, \vec 0) = \vec 0 \). Then,
\begin{itemize}
\item
    $\vec \Phi (z, \vec y, \vec u)$ is said to be \emph{non-linear} if at least one
        of its component polynomials \( \Phi_j \) is non-linear in one of
        the formal variables \( y_1, \ldots, y_m \);
\item
    $\vec \Phi (z, \vec y, \vec u)$ is said to be
        \emph{algebraic positive} if all of its component
        polynomials \( \Phi_j \) have non-negative coefficients;
\item
    $\vec \Phi (z, \vec y, \vec u)$ is said to be
    \emph{algebraic proper} if it admits a unique formal power series solution
        to which the iteration
        \begin{equation}
            \vec y_0(z, \vec u) = \vec 0
            \quad \text{and} \quad
            \vec y_{k+1}(z, \vec u) = \vec \Phi(z,
            \vec y_k(z, \vec u), \vec u)
        \end{equation}
        considered in the metric space of formal power series with valuation,
        converges as $k \to \infty$, and the Jacobian matrix
        \( \dfrac{\partial \vec \Phi}{\partial \vec y} \) is nilpotent at
        \( (z, \vec y) = (0, \vec 0) \).
\item
    $\vec \Phi (z, \vec y, \vec u)$ is said to be \emph{algebraic irreducible}
        if its dependency graph (i.e.~a graph whose vertices are the integers \(
        1, \ldots, m \) and there exists a directed edge \(k \to j \) if \( y_j
        \) figures in a monomial of \( \Phi_k \)) is strongly connected;
\item
    $\vec \Phi (z, \vec y, \vec u)$ is said to be
    \emph{algebraic aperiodic} if each of its component solutions
    \( y_j(z, \vec 1) \) for \( j = 1,\ldots,m \) is aperiodic in the sense that
   the greatest common divisor of the pairwise differences
        of the set of exponent indices of \( z \) within
        \( y_j(z, \vec 1) \) is equal to \( 1 \).
\end{itemize}
\end{definition}

\begin{remark}
    The notion of algebraic properness of systems, also referred to as
    \emph{well-foundedness}, is extensively studied in~\cite[Section
    5]{pivoteau2012algorithms}.  As discussed
    in~\cite{pivoteau2012algorithms,joyal1981theorie}, the system has
    combinatorial meaning only if the Jacobian is nilpotent, i.e.~if the
    recursive definition is well-defined and allows to inductively construct all
    the instances of combinatorial species.

    Let us note that the condition \( \vec \Phi (0, \vec 0, \vec 0) = \vec 0 \)
    is a technical assumption of the Drmota--Lalley--Woods theorem.  Pivoteau,
    Salvy and Soria consider, \emph{inter alia}, well-founded systems for which
    \( \vec y(0, \vec 0) \neq \vec 0 \).  One possible characterisation of such
    systems is the condition that the limit of a suitable iterative
    approximation procedure yields the solution of the initial functional
    system.

    The assumption that a system is \emph{polynomial} can be replaced by a more
    general assumption that the functions are \emph{analytic},
    see~\cite{drmota1997systems}. For a detailed and non-trivial study of the
    conditions regarding analytic functions and the configuration of the
    critical points in this more general case,
    see~\cite{bell2010characteristic}.
\end{remark}

\begin{prop}[Irreducible positive polynomial systems]
    \label{proposition:irreducible:polynomial:systems}
    Let
    \begin{equation}
        \vec y = \vec \Phi (z, \vec y, \vec u)
        = \seq{y_j = \Phi_j(z, y_1, \ldots, y_m, \vec u)}
        , \quad
        j = 1, \ldots, m
    \end{equation}
    be a non-linear polynomial system of equations which
    is algebraic positive, proper, and irreducible.
    Then there exists \( \varepsilon > 0 \) such that
    all component solutions \( y_j(z, \vec u) \)
    admit representation of the form
    \begin{equation}
        y_j(z, \vec u) =
        h_j \left(
            \sqrt{1 - \dfrac{z}{\rho(\vec u)}}, \vec u
        \right)
        = \sum_{k \geq 0} c_{k,j}(\vec u)
        \left(
            1 - \dfrac{z}{\rho(\vec u)}
        \right)^{k/2}
    \end{equation}
    for \( \vec u \) in a neighbourhood of \( \vec 1 \),
    \( |z - \rho(\vec u)| < \varepsilon \) and
    \( \arg(z - \rho(\vec u)) \neq 0 \), where
    \( c_{k,j}(\vec u) \) and \( \rho(\vec u) \)
    are analytic functions of \( \vec u \), and the functions
    \( h_j(t, \vec u) \) are analytic at \( (t, \vec u) = (0, \vec 1) \).
    In addition, if the system is algebraic
    aperiodic, then all \( y_j \) have \( \rho(\vec u) \) as their unique dominant
    singularity, and there exist constants \( 0 < \delta < \pi/2 \) and \( \eta
    > 0 \) such that \( \vec y(z, \vec u) \) is analytic in a region of the form
    \begin{equation}
        \Delta := \{
            z \colon
            |z| < \rho(\vec 1) + \eta,
            |\arg(z / \rho(\vec u) - 1)| > \delta
        \}.
    \end{equation}
\end{prop}

\begin{remark}
    The above Drmota--Lalley--Woods theorem has been further generalised by
    Drmota, Gittenberger and Morgenbesser in the case of strongly connected
    systems with infinitely many equations~\cite{infinitesystems}. In their
    generalisation, the authors require that the Jacobian of the system (or some
    of its power) is a compact operator.  Alas, as the system corresponding to
    closed \lterms~is not strongly connected, it does not fit into their
    framework. In the current paper we introduce a new condition of
    \emph{exponential convergence} which is independent of the Jacobian and
    conjecture that it is crucial for obtaining the respective Puiseux
    expansions of generating functions.
\end{remark}

\begin{prop}[Differential condition for the systems of equations~{\cite[see
    proof of Theorem 1]{infinitesystems}}]
\label{proposition:dominant:singularity}
    Let \( \vec y = \vec \Phi (z, \vec y) \) be a non-linear system of
    polynomial  equations \(y_j = \Phi_j(z, y_1, \ldots, y_m) \) with $j$
    ranging over $1,\ldots,m$. Assume that \( \vec y = \vec \Phi (z, \vec y) \)
    is algebraic positive, proper and irreducible.  Let \( \rho \) be the common
    singularity of its solution vector \( y_j \).  Then, the spectral radius
    (largest absolute value of its eigenvalues) of the Jacobian matrix \(
    \dfrac{\partial \vec \Phi}{\partial \vec y}\) is a strictly increasing
    function of \( z \) on the interval \([0, \rho]\) and is bounded from above
    by \( 1 \), with the equality holding if and only if \( z = \rho \).
\end{prop}

\subsection{Limit laws}\label{subsec:limit:laws}
Consider a bivariate generating function \( L(z, u) \) with non-negative
coefficients and a sequence of random variables \( (X_n)_{n \geq 0} \) such that
\begin{equation}
    L(z, u) = \sum_{n, k \geq 0} a_{n,k} z^n u^k
    \quad \text{and} \quad
    \mathbb P (X_n = k) = \dfrac{a_{n,k}}{\sum_{j \geq 0} a_{n,j}}.
\end{equation}
We say that \emph{$X_n$ is associated with variable $u$}.  In order to
understand the limiting behaviour of \( X_n \) we investigate the
\emph{probability generating function} \( p_n(u) \) of \( X_n \) defined as
\begin{equation}
    p_n(u) :=
    \sum_{k \geq 0} \mathbb P(X_n = k) u^k =
    \dfrac{[z^n] L(z, u)}{[z^n] L(z, 1)}
    .
\end{equation}

Once accessed, \( p_n(u) \) proves extremely useful in establishing the traits
of $X_n$ as $n$ tends to infinity. In what follows, we focus on two types of
limiting distributions. The first type is related to the case of a so-called
\emph{fixed} singularity, which results in discrete limit law; the second type
is related to so-called \emph{moving} singularity, and typically results in
a Gaussian limit law.

\subsubsection{Discrete limit laws}
\begin{prop}[{\cite[Section IX.2]{flajolet09}}]
    \label{proposition:discrete:limit:laws}
    Suppose that bivariate power series \( L(z, u) \) admits in a complex
    neighbourhood of \( u = 1 \) a Puiseux series expansion in form of
    \begin{equation}
        L(z, u) = \alpha(u) - \beta(u) \sqrt{1 - \dfrac{z}{\rho}}
        + O \left( \left|
            1 - \dfrac{z}{\rho}
        \right| \right)
    \end{equation}
    as \( z \to \rho \) uniformly in delta-domain \( \Delta(R) \) for some \( R
    > \rho \) (see~\autoref{proposition:transfer:theorem}).  Then, the random
    variable \( X_n \) associated with the marking variable \( u \) converges
    in distribution to a discrete limiting distribution with probability
    generating function
    \begin{equation}
        p(u) = \lim_{n \to \infty} p_n(u) = \dfrac{\beta(u)}{\beta(1)}.
    \end{equation}
    The corresponding mean values satisfy
    \begin{equation}
        \lim_{n \to \infty} \mathbb E X_n = \dfrac{\beta'(1)}{\beta(1)}.
    \end{equation}
\end{prop}

\subsubsection{Central limit theorem}
\begin{remark}\label{remark:mclt:quasi-powers}
In 1983, Bender and Richmond~\cite{BenderCLT} proved a multi-dimensional variant
    of the central limit theorem for probability generating functions taking the
    quasi-power form \mbox{\( p_n(\vec u) \sim A(\vec u) B(\vec u)^n \)}. This
    line of research was later continued by Hwang~\cite{hwang1998convergence} who
    established precise rates of convergence in the one-dimensional case.  The
    two-dimensional case by was next  investigated by
    Heuberger~\cite{heuberger2007hwang}.  More recently, in 2016, the full
    multi-dimensional version has been resolved by Heuberger and
    Kropf~\cite{higherdimensional2016quasipower} using a multi-dimensional
    version of the Berry--Esseen inequality.  Although we do not touch on the
    rates of convergence in the current paper, let us mention that they can be
    obtained using the above results.
\end{remark}

In order to formulate the multivariate central limit theorem, it is convenient
to introduce the notion of logarithmic derivative which enters the mean value
and the covariance matrix of the resulting random variable.
\begin{definition}
    The logarithmic derivative of $A(z,u)$ is given by the expression
\begin{equation}
    \dfrac{\partial}{\partial \log u} A(z, u)
    :=
    \left.\dfrac{\partial}{\partial \eta}
    A(z, e^\eta)\right|_{\eta = \log u}
    =
    u\dfrac{\partial}{\partial u} A(z, u).
\end{equation}
\end{definition}

\begin{prop}[Multivariate central limit theorem,~{\cite[Theorem 1]{BenderCLT}}]
\label{proposition:multivariate:clt}
Let \( (\vec X_n)_{n=1}^\infty \) be a sequence of coordinate-wise non-negative \( d
    \)-dimensional discrete random vectors with probability generating functions
    \( p_n(\vec u) := \mathbb E(\vec u^{X_n}) \),
    \( \vec u = (u_1, u_2, \ldots, u_d) \). Suppose that uniformly in a
    fixed complex neighbourhood of \( \vec u = \vec 1 \) one has
    \begin{equation}
    p_n(\vec u) \sim A(\vec u) \cdot B(\vec u)^{n}
    \end{equation} where \(
    A(\vec u) \) is uniformly continuous and \( B(\vec u) \) has a quadratic
    Taylor series expansion with error term \( O\big( \sum| u_k - 1|^3\big) \).
    Assume that \( B(\vec u) \) satisfies the following \emph{variability
    condition}:
    \begin{equation}\label{eq:variability:condition}
    \det
    \left[
        \dfrac
        {\partial^2 \log B(\vec u)}
        {\partial \log u_i \, \partial \log u_j}
    \right]_{i,j} > 0.
\end{equation}
Then, the sequence of random variables \( \vec X_n \), after standardization,
    converges in law to Gaussian random variable satisfying
\begin{equation}
    \dfrac{\vec X_n - \mathbb E \vec X_n}{\sqrt{n}} \overset{d}\to
\mathcal N\left(\vec 0, \Sigma\right).
\end{equation}
The mean vector and the covariance matrix satisfy
    \begin{equation}\label{eq:multivariate:clt:mean:covariance:matrix}
    \mathbb E \vec X_n
    \sim
    n \cdot \left.\dfrac{\partial B}{\partial \vec u}\right|_{\vec u = \vec 1}
        \quad \text{and} \quad
    \mathrm{Cov}\, \vec X_n \sim
    n \cdot
    \left.
    \left[
        \dfrac{\partial^2 \log B(\vec u)}{\partial \log u_i \partial \log u_j}
    \right]_{i,j}
    \right|_{\vec u = 1}.
\end{equation}
    In the one-dimensional
    case~\eqref{eq:multivariate:clt:mean:covariance:matrix} simplifies to
    \begin{equation}\label{eq:multivariate:clt:one:dimension:exp:var}
        \mathbb E(X_n) \sim n B'(1) \quad \text{and} \quad
        \mathbb V(X_n) \sim n \left(B''(1) + B'(1) - {B'(1)}^2\right).
\end{equation}
\end{prop}

\begin{remark}\label{remark:bender:richmond:framework}
Typically, when the singularity is moving
    (see~\autoref{proposition:irreducible:polynomial:systems}) the bivariate
    generating function takes the form
\begin{equation}
    A(z, \vec u) = \alpha(\vec u) - \beta(\vec u)
    \sqrt{1 - \dfrac{z}{\rho(\vec u)}} +
        O\left( \left|
            1 - \dfrac{z}{\rho(\vec u)}
        \right| \right)
\end{equation}
uniformly as \( z \to \rho(\vec u) \) for \( \vec u \) in a vicinity of \( \vec 1 \).

    Consequently, the probability generating function takes form
\begin{equation}
    p_n(\vec u) \sim \dfrac{\beta(\vec u)}{\beta(\vec 1)}
    \left(
        \dfrac{\rho(\vec 1)}{\rho(\vec u)}
    \right)^n.
\end{equation}
In this form, the probability generating function satisfies the premises of the
    multivariate quasi-power theorem (see~\autoref{remark:mclt:quasi-powers})
    and so one can also obtain the speed of convergence. In our situations this
    speed is typically of order \( O\left( \frac{1}{\sqrt n} \right) \).
\end{remark}

\begin{remark}
For convenience, we say that a random vector \( \vec X_n \) converges
in law to multivariate Gaussian distribution with mean \(n  \vec \mu \) and
variance \( n \Sigma \) writing
\begin{equation}
    \vec X_n \overset{d}\longrightarrow \mathcal N(n \vec \mu, n \Sigma)
\end{equation}
to denote that \( \dfrac{\vec X_n - n \vec \mu}{\sqrt n} \overset{d}\longrightarrow
\mathcal N(\vec 0, \Sigma) \).
\end{remark}

%% file: 3-basic-statistics.tex
\section{Basic statistics of plain lambda terms}\label{sec:basic:statistics}
In this section we investigate several basic combinatorial parameters related to
random plain \lterms. Let us start with invoking the combinatorial
specification~\eqref{eq:plain:terms:grammar} describing the set $\classL$ of
plain \lterms. Recall that $\classL$ is specified as
\begin{align}
    \begin{split}
        \classL &::= \idx{n}~|~`l \classL~|~(\classL \classL)\\
        \idx{n} &::= {\sf 0}~|~{\sf S}~\idx{n}.
    \end{split}
\end{align}
Equivalently, the set $\classL$ of \lterms~can be specified using a
pictorial tree grammar, see~\autoref{fig:plain-terms:specification}
(note the explicit @ symbol for term application).
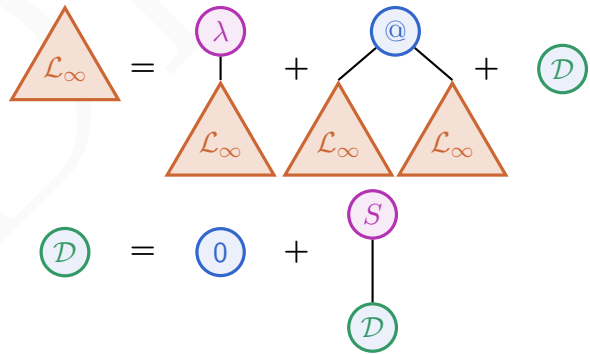
\begin{figure}[hbt]
\centering
\begin{tikzpicture}[>=stealth',level/.style={sibling distance = 1.5cm/#1,
  level distance = 1.5cm, thick}]
\matrix
{
\draw
node[triangle]{\( \classL \)}
;
&
\draw
node{\( \boldsymbol = \)}
;
&
\draw
++(0,.5)
node(up)[arn_r]{\( \lambda\)}
    child{
        node[triangle]{\( \classL \)}
    }
++(1.0,-0.5)
node{\( \boldsymbol + \)}
++(1.3, 0.5)
node(up)[arn_n]{\( @ \)}
    child[child anchor = north]{
        node[triangle]{\( \classL \)}
    }
    child[child anchor = north]{
        node[triangle]{\( \classL \)}
    }
++(1.2,-0.5)
node{\( \boldsymbol + \)}
++(1,0)
node[arn_g](var){\( \mathcal D \)}
;
\\[1ex]
\draw
node[arn_g](var){\( \mathcal D \)}
;
&
\draw
node{\( \boldsymbol = \)}
;
&
\draw
    node[arn_n](var){\( \mathsf{0} \)}
++(1.0,0)
node{\( \boldsymbol + \)}
++(1.0,0.5)
node(up)[arn_r]{\( S \)}
    child{
        node[arn_g]{\( \mathcal D \)}
    }
;
\\
};
\end{tikzpicture}
\caption{\label{fig:plain-terms:specification}Combinatorial specification for plain $\lambda$-terms.}
\end{figure}

Following symbolic methods~\cite[Part A: Symbolic Methods]{flajolet09} we note that the
generating function $D(z)$ corresponding to de~Bruijn indices takes the form
$D(z) = \dfrac{z}{1-z}$ and so the generating function $L_\infty(z)$ associated
with plain \lterms~satisfies the following functional equation:
\begin{equation}\label{eq:spec-plain-lambda-terms}
L_{\infty}(z) = z L_{\infty}(z) + z {L_{\infty}(z)}^2 + \frac{z}{1-z}.
\end{equation}
Solving~\eqref{eq:spec-plain-lambda-terms} for $L_\infty(z)$ we obtain two
formal solutions. Since we know \emph{a priori} that the resulting generating
function has non-negative coefficients $[z^n]L_\infty(z)$ we conclude that
\begin{align}\label{eq:genfun-plain-lambda-terms}
\begin{split}
    L_\infty(z) = \frac{1}{2 z} \left(1-z-\sqrt{\left(1-z\right)^2-\frac{4
    z^{2}}{1-z}}\right).
\end{split}
\end{align}

In this form, we can easily verify that the radicand expression \(
\left(1-z\right)^2-\dfrac{4 z^{2}}{1-z}\) carries the single dominant
square-root type singularity $\rho$ of $L_\infty(z)$.  At this point, a
straightforward application of the transfer theorem
(see~\autoref{proposition:transfer:theorem}) gives us access to the asymptotic
growth rate of the counting sequence corresponding to plain \lterms.

\begin{prop}[see~{\cite{Bendkowski2016}}]\label{prop:asymptotic:plain:terms}
    Let $L_\infty(z)$ be the generating function associated with plain
    \lterms~\eqref{eq:genfun-plain-lambda-terms}.  Then, the number
    $[z^n]L_\infty(z)$ of plain terms of size $n$ admits the following
    asymptotic approximation:
\begin{equation}\label{eq:plain-lambda-terms-asymptotics}
[z^n]L_\infty(z) \xrightarrow[n\to\infty]{} C \rho^{-n} n^{-3/2}
\end{equation}
where
    \begin{equation}\label{eq:plain-lambda-terms-dominant-singularity}
    \rho \doteq 0.29559774
        \quad \text{and} \quad C \doteq 0.606767.
\end{equation}
More specifically,
\( \rho \) is the positive real root of the polynomial
\( z^3 + z^2 + 3z - 1 = 0 \) whereas
\( C = \frac{1}{2(1 - \rho)} \sqrt{\frac{\rho + 2}{\pi}} \).
\end{prop}

\subsection{Variables in plain lambda terms}
We start our investigations with the variable distribution in plain \lterms.

\begin{prop}\label{prop:plain:terms:variables}
Let $X_n$ be a random variable corresponding to the number of variables in a
    random plain \lterm~of size $n$. Then, after
    standardisation, $X_n$ converges in law to a Gaussian distribution.
    Its expectation $\mu_n$ and variance $\sigma_n^2$ satisfy (up to numerical
    approximation)
    \begin{equation}\label{eq:plain:terms:expctation:and:variance}
        \mu_n \xrightarrow[n\to\infty]{} 0.306849 n \quad \text{and} \quad
        \sigma_n^2 \xrightarrow[n\to\infty]{} 0.0516364 n.
    \end{equation}
\end{prop}

\begin{proof}
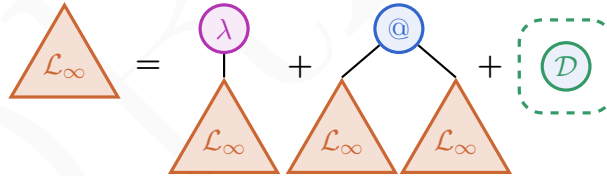
\begin{figure}[hbt]
\centering
\begin{tikzpicture}[>=stealth',level/.style={sibling distance = 1.5cm/#1,
    level distance = 1.5cm, node/.style={anchor=north}, thick}]
\draw
node[triangle]{\( \classL \)}
++(1.1,0)
node{\( \boldsymbol = \)}
++(1.0,0.5)
node(up)[arn_r]{\( \lambda\)}
    child{
        node[triangle]{\( \classL \)}
    }
++(1.0,-0.5)
node{\( \boldsymbol + \)}
++(1.3, 0.5)
node(up)[arn_n]{\( @ \)}
    child[child anchor = north]{
        node[triangle]{\( \classL \)}
    }
    child[child anchor = north]{
        node[triangle]{\( \classL \)}
    }
++(1.2,-0.5)
node{\( \boldsymbol + \)}
++(1,0)
node[arn_g](var){\( \mathcal D \)}
;
\node[rectangle,dashed,draw,fit=(var),
rounded corners=3mm,inner sep=8pt, bgreen, very thick] {};
\end{tikzpicture}
\caption{\label{fig:marking:variables}Marking variables in plain terms.}
\end{figure}

    Let us consider a bivariate generating function $L_\infty(z,u)$ in which $[z^n u^k]
    L_\infty(z,u)$, i.e.~the coefficient standing by $z^n u^k$, denotes the number of plain \lterms~of size $n$ with $k$ variables
    (equivalently $k$ occurrences of $\mathsf{0}$).  Marking all occurrences
    of $\mathsf{0}$ in the defining equation~\eqref{eq:spec-plain-lambda-terms}
    of $L_{\infty}(z)$, see~\autoref{fig:marking:variables}, we obtain the
    following combinatorial specification for $L_\infty(z,u)$:
\begin{equation}\label{eq:variables-gen-fun-spec}
L_\infty(z,u) = z L_\infty(z,u) + z {L_\infty(z,u)}^2 + \frac{u z}{1-z}
\end{equation}
and hence
\begin{equation}\label{eq:variables-bivariate-gen-fun}
L_\infty(z,u) = \frac{1}{2 z} \left(1-z-\sqrt{\left(1-z\right)^2-\frac{4 u
    z^{2}}{1-z}}\right)
\end{equation}
    as $L_\infty(z,1) = L_{\infty}(z)$.

    The dominant singularity \( \rho(u) \) of \( L_\infty(z, u) \), is the real
    positive root of the radicand expression $F(z,u) =
    \left(1-z\right)^2-\dfrac{4 u z^{2}}{1-z}$.  Moreover, the singularity has a
    non-zero derivative at \( u = 1 \). According
    to~\autoref{remark:bender:richmond:framework} (moving singularity framework)
    this yields a Gaussian limit law.

    The mean and the variance of the resulting normal distribution can be
    computed by~\autoref{proposition:multivariate:clt} using the values \(
    \rho'(1) \) and \( \rho''(1) \) from the partial derivatives of \( F(z, u)
    \). Since \( F(\rho(u), u) = 0 \), after taking the derivative with respect
    to \( u \) we obtain
    \begin{equation}
        \rho'(1) = - \dfrac{\partial_u F(\rho, 1)}{\partial_z F(\rho, 1)}
        \quad \text{and} \quad
        \rho''(1) =
        - \dfrac
        {\partial_{u}^2 F(\rho, u)
        +
        2 \rho'(1) \partial_{z}\partial_u F(\rho, 1)
        +
        {\rho'(1)}^2
        \partial_{z}^2 F(\rho,1)
        }
        {\partial_z F(\rho, 1)}
        .
    \end{equation}
\end{proof}

\begin{remark}
Let us note that, in general, $\rho(u)$ is a root of the polynomial
\mbox{$(1-z^b)F(z,u)$} whose degree depends on the specific size model parameters
$a,b,c,d$ denoting the weights of zero, successor, abstraction and application,
respectively, cf.~\cite{gittenberger_et_al:LIPIcs:2016:5741}. Specifically,
\begin{equation}\label{eq:plain:terms:variables:rho:u:implicit}
    (1-z^b) F(z,u) = (1-z^b)\left(1-z^c\right)^2 -4 u z^{a+d}.
\end{equation}
Consequently, for most admissible size notion parameters we cannot explicitly
    obtain analytic expression of $\rho(u)$. Instead, in order to check the
    premises of the multivariate central limit theorem we have to work with the
    implicit equation~\eqref{eq:plain:terms:variables:rho:u:implicit}.

The main technical obstacle lies in the verification of the requested
    variability condition $B''(1) + B'(1) -{B'(1)}^2 \neq 0$,
    see~\eqref{eq:variability:condition}. The remaining argumentation is
    virtually identical to the one presented for the specific case of
    $a=b=c=d=1$.
\end{remark}

\subsection{Redexes in plain lambda terms}
Basic marking techniques allow us also to investigate limiting distributions of
various sub-patterns in plain \lterms. In what follows we study the fundamental
pattern of \bredexes. Recall that a \bredex~is a \lterm~in form of $(`l
N) M$ where $N$ and $M$ arbitrary $`l$-terms. In other words, a
sub-pattern which can be depicted as
\raisebox{-2.0 ex}{
\begin{tikzpicture}[level/.style={sibling distance = 0.4cm/#1, level
distance = .3cm}]
\draw
++(0, -1.5)
    node[arn_nn]{\( {\scriptscriptstyle @} \)}
    child{
        node[arn_nn]{\( {\scriptscriptstyle \lambda} \)}
        child{
            node[arn_nn]{\( \cdot \)}
        }
    }
    child{
        node[arn_nn]{\( \cdot \)}
    }
;
\end{tikzpicture}
}.

\begin{prop}\label{prop:plain:terms:redexes}
Let $X_n$ be a random variable denoting the number of \bredexes~in a random
    plain \lterms~of size $n$. Then, after standardisation, $X_n$ converges in
    law to a Gaussian distribution with expectation $\mu_n$ and variance
    $\sigma_n^2$ satisfying (up to numerical
    approximation)
    \begin{equation}
        \mu_n \xrightarrow[n\to\infty]{} 0.0907039 n \quad \text{and} \quad
        \sigma_n^2 \xrightarrow[n\to\infty]{} 0.0519495 n.
    \end{equation}
\end{prop}

\begin{proof}
    We start with establishing a formal specification for \bredexes~in plain
    \lterms. For that purpose, we introduce an auxiliary class $\mathcal{N}$
    consisting of de~Bruijn indices and terms in application form.

    Note that if $N$ is in application form, then either $N = (`l M) P$,
    i.e.~$N$ is a \bredex, or $N = M P$ where $M$ belongs itself to class
    $\mathcal{N}$. Furthermore, with $\mathcal{N}$ at hand, we notice that each
    plain \lterm~$N$ takes either the form $N = `l M$ for some \lterm~$M$, or is
    an element of $\mathcal{N}$. We can therefore write down the
    following joint combinatorial
    specification~\eqref{eq:plain:terms:redexes:specification} for
    $L_\infty(z,u)$ and $N(z,u)$ corresponding to $\classL$ and $\mathcal{N}$,
    respectively, using the variable $u$ to mark the redex occurrences in
    $\mathcal{N}$, see~\autoref{fig:marking:redexes}:

    \begin{align}\label{eq:plain:terms:redexes:specification}
    \begin{split}
    L_\infty(z,u) &= z L_\infty(z,u) + N(z,u)\\
    N(z,u) &= \frac{z}{1-z} + u z^2 {L_\infty(z,u)}^2 + z N(z,u) L_\infty(z,u).
    \end{split}
    \end{align}

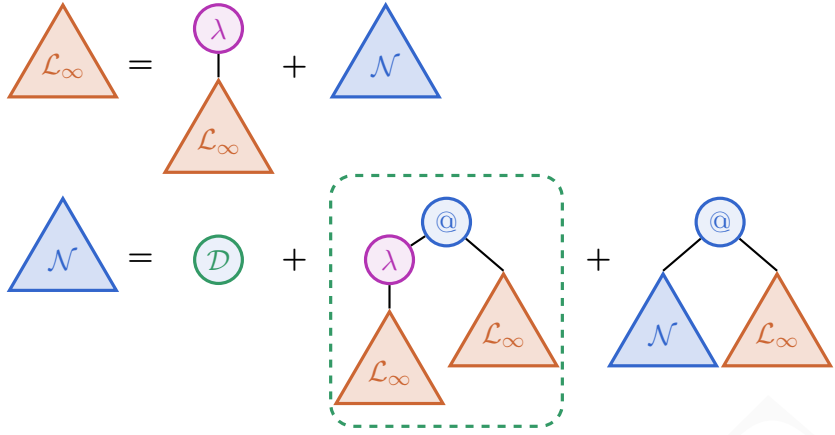
\begin{figure}[hbt]
\centering
\begin{tikzpicture}[>=stealth',level/.style={sibling distance = 1.5cm/#1,
  level distance = 1.5cm, thick}]
\matrix{
\draw
node[triangle]{\( \mathcal L_\infty \)}
;
&
\draw
node{\( \boldsymbol = \)}
;
&
\draw
++(0,0.5)
node(up)[arn_r]{\( \lambda\)}
    child{
        node[triangle]{\( \classL \)}
    }
++(1,-0.5)
node{\( \boldsymbol + \)}
++(1.2, 0)
node[triangle_b]{ \( \mathcal N \) }
;
\\
\draw
node[triangle_b]{\( \mathcal N \)}
;
&
\draw
node{\( \boldsymbol = \)}
;
&
\draw
node[arn_g]{\( \mathcal D \) }
++(1,0)
node{\( \boldsymbol + \)}
++(2.0,0.5)
node[arn_n](app){\( @ \)}
    child[level distance=0.5cm]{
        node[arn_r]{ \( \lambda \) }
        child{
            node[triangle](lt){\( \classL \)}
        }
    }
    child[level distance=1.5cm, child anchor = north]{
        node[triangle](rt){\( \classL \)}
    }
++(2.0,-0.5)
node{\( \boldsymbol + \)}
++(1.6, +0.5)
node[arn_n]{ \( @ \) }
    child[child anchor = north]{
        node[triangle_b]{ \( \mathcal N \) }
    }
    child[child anchor = north]{
        node[triangle]{ \( \classL \) }
    }
;
\node[rectangle,dashed,draw,fit=(app)(lt)(rt),
rounded corners=3mm,inner sep=8pt, bgreen, very thick] {};
\\
};
\end{tikzpicture}
\caption{\label{fig:marking:redexes}Marking redexes in plain terms.}
\end{figure}
Given that $L_\infty(z,1) = L_\infty(z)$ we
solve~\eqref{eq:plain:terms:redexes:specification} for $L_\infty(z)$ we finally
obtain
\begin{equation}\label{eq:plain:terms:redexes:gf}
L_{\infty}(z,u) = \frac{1-z - \sqrt{(1-z)^2 - \frac{4 z^{2}
    \left(1+z(u-1)\right)}{1-z}} }{2 z \left(1+z(u-1)\right)}
\end{equation}

    With the closed-form formula~\eqref{eq:plain:terms:redexes:gf} we can now
    easily access the dominant singularity $\rho(u)$ of $L_\infty(z,u)$ carried
    by the radicand expression $(1-z)^2 - \frac{4 z^{2}
    \left(1+z(u-1)\right)}{1-z}$.  Consequently, a straightforward application
    of the multivariate central limit theorem finishes the proof,
    see~\autoref{proposition:multivariate:clt}.
\end{proof}

\subsection{Joint distribution of variables, abstractions, successors and redexes}
\begin{prop}\label{prop:joint:plain}
Let \( \vec X_n = (X_{n (\textsf{var})}, X_{n (\textsf{red})},X_{n
    (\textsf{suc})},X_{n (\textsf{abs})}) \) be a random vector
    denoting
    \begin{enumerate}
        \item the number $X_{n (\textsf{var})}$ of variables;
        \item the number $X_{n (\textsf{red})}$ of $`b$\nobreakdash-redexes;
        \item the number $X_{n (\textsf{suc})}$ of successors, and
        \item the number $X_{n (\textsf{abs})}$ of abstractions
    \end{enumerate}
    in a random plane \lterms~of size $n$. Then, after standardisation, the
    random vector \( \vec X_n \) converges in law to a multivariate Gaussian
    distribution satisfying (up to numerical approximation)
    \begin{equation}\label{eq:joint:gaussian:distribution}
        X_n \overset{d}\longrightarrow
        \mathcal N \left(
            n
            \begin{pmatrix}
                0.307 \\
                0.091 \\
                0.129 \\
                0.258
            \end{pmatrix}
            ,
            n
            \begin{pmatrix}
                0.052 & -0.013 & -0.034 & -0.069 \\
               -0.013 &  0.052 & -0.022 &  0.047 \\
               -0.034 & -0.022 &  0.145 & -0.076 \\
               -0.069 &  0.047 & -0.076 &  0.214
            \end{pmatrix}
        \right).
    \end{equation}
\end{prop}

\begin{proof}
    Like in the corresponding proofs for single parameters
    (see~\autoref{prop:plain:terms:variables}
    and~\autoref{prop:plain:terms:redexes}) we base our proof on the
    multivariate central limit theorem
    (see~\autoref{proposition:multivariate:clt}).  We start with a joint
    multivariate specification for plain \lterms~including the investigated
    combinatorial parameters marked using auxiliary variable vector
    \( \vec u
    =
    (
        u_{(\textsf{var})},
        u_{(\textsf{red})},
        u_{(\textsf{suc})},
        u_{(\textsf{abs})}
    ) \)
    corresponding to respective components of \( \vec X_n
    \), see~\autoref{fig:plain:joint:multivariate:spec} (cf.~\autoref{fig:marking:redexes}).
\begin{figure}[hbt]
\centering
\begin{tikzpicture}[>=stealth',level/.style={sibling distance = 1.5cm/#1,
  level distance = 1.5cm, thick}]
\matrix{
\draw
node[triangle]{\( \mathcal L_\infty \)}
;&\draw
node{\( \boldsymbol = \)}
;&\draw
++(0,.5)
node(abstraction)[arn_r]{\( \lambda\)}
    child{
        node[triangle]{\( \mathcal L_\infty \)}
    }
++(1,-0.5)
node{\( \boldsymbol + \)}
++(1.2, 0)
node[triangle_b]{ \( \mathcal A \) }
;
\node[rectangle,dashed,draw,fit=(abstraction),
rounded corners=3mm,inner sep=4pt, bred , very thick] {};
 \\
\draw
node[triangle_b]{\( \mathcal A \)}
;&\draw
node{\( \boldsymbol = \)}
;&\draw
node(var)[arn_g]{\( \mathcal D \) }
++(1,0)
node{\( \boldsymbol + \)}
++(1.8,0.5)
node[arn_n](app){\( @ \)}
    child[level distance=0.5cm]{
        node(abstraction2)[arn_r]{ \( \lambda \) }
        child{
            node[triangle](lt){\( \mathcal L_\infty \)}
        }
    }
    child[level distance=1.5cm, child anchor=north]{
        node[triangle](rt){\( \mathcal L_\infty \)}
    }
++(1.8,-0.5)
node{\( \boldsymbol + \)}
++(1.6, +0.5)
node[arn_n]{ \( @ \) }
    child[child anchor=north]{
        node[triangle_b]{ \( \mathcal A \) }
    }
    child[child anchor=north]{
        node[triangle]{ \( \mathcal L_\infty \) }
    }
;
\node[rectangle,dashed,draw,fit=(app)(lt)(rt),
rounded corners=3mm,inner sep=8pt, bgreen, very thick] {};
\node[rectangle,dashed,draw,fit=(var),
rounded corners=3mm,inner sep=8pt, bblue, very thick] {};
\node[rectangle,dashed,draw,fit=(abstraction2),
rounded corners=3mm,inner sep=4pt, bred , very thick] {};
\\
};
\end{tikzpicture}
\caption{\label{fig:plain:joint:multivariate:spec}Marking abstractions,
variables, successors and redexes in plain \lterms.}
\end{figure}
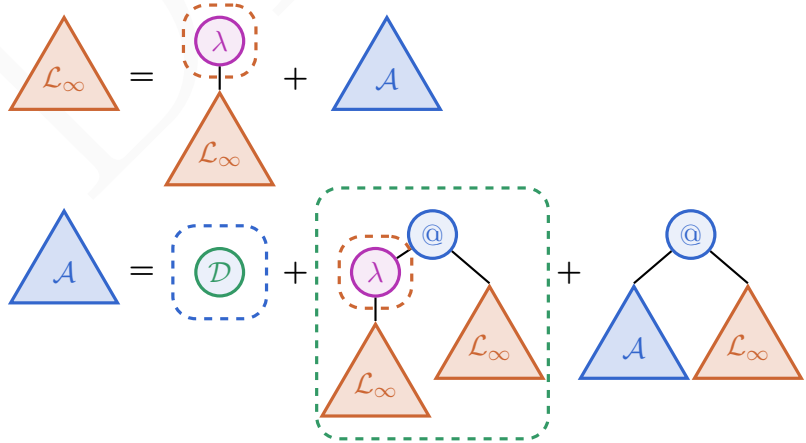

Let $\vec u$ denote the vector of considered marking variables. Such a
specification, when converted into a system of functional equations involving
the generating functions $L_\infty(z, \vec u)$ and $A(z, \vec u)$ associated
with $\classL$ and $\mathcal{A}$, respectively, yields
\begin{align}\label{eq:plain:joint:multivariate:system}
\begin{split}
    L_\infty(z, \vec u) &= u_{(\textsf{abs})} z L_\infty(z, \vec u) + A(z, \vec u) \\
        A(z, \vec u) &= \dfrac{u_{(\textsf{var})} z}{1 - u_{(\textsf{suc})} z} +
        u_{(\textsf{red})} u_{(\textsf{abs})} z^2 L_\infty(z, \vec u)^2 + z A(z,
        \vec u) L_\infty(z, \vec u).
\end{split}
\end{align}
Furthermore, reformulating~\eqref{eq:plain:joint:multivariate:system} we
find that
\begin{equation}\label{eq:plain:joint:multivariate:quadratic}
    (1 - u_{(\textsf{abs})} z ) L_\infty(z, \vec u)
    = \left(z u_{(\textsf{abs})} (u_{(\textsf{red})} - 1) + 1 \right)
    z {L_\infty(z, \vec u)}^2
    + \dfrac{u_{(\textsf{var})} z}{1 - u_{(\textsf{suc})} z}.
\end{equation}
Let $\Delta(z, \vec u)$ be the discriminant of the quadratic
equation~\eqref{eq:plain:joint:multivariate:quadratic} defining $L_\infty(z,
\vec u)$.

Note that $\Delta(z, \vec u)$ satisfies
\begin{equation}\label{eq:plain:joint:multivariate:quadratic:disc}
    \Delta(z, \vec u) = {(1 - u_{(\textsf{abs})} z)}^2 - 4
    \left(z u_{(\textsf{abs})} (u_{(\textsf{red})} - 1) + 1 \right)
    \dfrac{u_{(\textsf{var})} z^2}{1 - u_{(\textsf{suc})} z}.
\end{equation}

In this form, we can access the dominant singularity $\rho(\vec u)$ of
$L_\infty(z, \vec u)$ solving $\Delta(z, \vec u) = 0$ for $z$ as a function of
$\vec u$. Since~\eqref{eq:plain:joint:multivariate:quadratic:disc} is a cubic
equation in $z(\vec u)$ we have access to the analytic form of its roots.  We
can therefore easily check that only one solution $z(\vec u)$
of~\eqref{eq:plain:joint:multivariate:quadratic:disc} coincides at $\vec u =
\vec 1$ with the
dominant singularity $\rho$ corresponding to plain
terms~\eqref{eq:plain-lambda-terms-dominant-singularity}.
Consequently, the generating function \( L_\infty(z, \vec u) \) admits a corresponding
Puiseux expansion in form of
\begin{equation}
    L_\infty(z, \vec u) = \alpha(z, \vec u) - \beta(z, \vec u)
    \sqrt{1 - \dfrac{z}{\rho(\vec u)}} + O\left(\bigg| 1 - \dfrac{z}{\rho(\vec
    u)} \bigg|\right)
\end{equation}
where both \( \alpha(z, \vec u) \) and \( \beta(z, \vec u) \) are analytic and
non-vanishing near \( (z, \vec u) = (\rho, \vec 1) \).

The required variability condition can be directly verified once the analytic
form of $\rho(\vec u)$ is calculated.  At this point, the multivariate central
limit theorem is readily applicable yielding the asserted convergence. A direct
computation gives the vector of corresponding means and covariance matrix,
see~\eqref{eq:joint:gaussian:distribution}.
\end{proof}

\begin{remark}
    Arguably, the most interesting part of the covariance
    matrix~\eqref{eq:joint:gaussian:distribution} is the sign of the
    correlations and the absolute values of associated variances.  Note that in
    the natural size notion, the number of abstractions has greater variance
    than other constructors. Interestingly, the number of $`b$\nobreakdash-redexes is
    positively correlated with the number of abstractions. Not surprisingly,
    all other parameters are negatively correlated.
\end{remark}

\subsection{Head abstractions in plain lambda terms}\label{subsec:basic:head:abs}
In this section we turn to head abstractions in plain \lterms~showing that
the corresponding random variable converges to a discrete geometric
distribution.

\begin{prop}\protect\label{prop:plain:terms:head:abstractions}
    Let $X_n$ be a random variable denoting the number of head abstractions in a
    random plain \lterm~of size $n$. Then, $X_n$ converges in law to a geometric
    distribution $\Geom(\rho)$ with parameter $\rho$. Specifically,
    \begin{equation}\label{eq:plain:terms:head:abstractions:distribution}
        \mathbb P(X_n = h) \xrightarrow[n\to\infty]{} \mathbb P(\Geom(\rho) =
        h) = (1 - \rho) \rho^h.
    \end{equation}
\end{prop}

\begin{proof}
    Note that each \lterm~starts with a (perhaps empty)  sequence of consecutive
    head abstractions followed either by a de~Bruijn index or an application of
    two terms (recall that abstractions therein are no longer considered to be
    head abstractions).  Consequently, the set $\classL$ of plain \lterms~can be
    specified using the auxiliary class $\mathcal{H}$ of head abstractions as
    depicted in~\autoref{fig:marking:head:abstractions}.
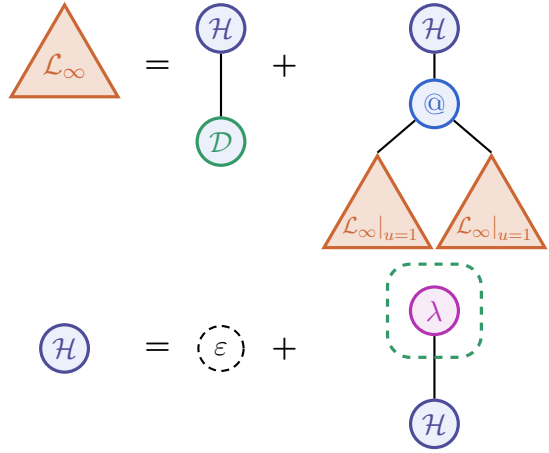
\begin{figure}[hbt]
\centering
\begin{tikzpicture}[>=stealth',level/.style={thick}]
\matrix[row sep=-\pgflinewidth, column sep=.2cm]{
\draw
node[triangle]{\( \mathcal L_\infty \)}
;
&
\node{\( \boldsymbol = \)}
;
&
\draw
++(0,.5)
node(up)[arn_e, align=center]{\(\mathcal H\)}
    child[level distance = 1.5cm]{
        node[arn_g]{\( \mathcal D \)}
    }
;
&
\node{\( \boldsymbol + \)}
;
&
\draw
++(0, +0.5)
node(up)[arn_e]{\(\mathcal H\)}
    child[level distance=1.0cm]{
        node[arn_n]{\( @ \)}
        child[level distance=1.5cm, sibling distance=1.5cm, child anchor = north]{
            node[triangle, text depth = -1.2ex]
            {\scalebox{0.8}{\hspace{-2.3ex}{
                \(
                    \left.
                        \mathcal L_\infty
                    \right|_{u=1}
                \)
            }}}
        }
        child[level distance=1.5cm, sibling distance=1.5cm, child anchor = north]{
            node[triangle, text depth = -1.2ex]
            {\scalebox{0.8}{\hspace{-2.3ex}{
                \(
                    \left.
                        \mathcal L_\infty
                    \right|_{u=1}
                \)
            }}}
        }
    }
;
\\[.5em]
\draw
node[arn_e]{\( \mathcal H \)}
;
&
\node{\( \boldsymbol = \)}
;
&
\node[arn_w, align=center]{\makebox[0pt][c]{\(\varepsilon \)}}
;
&
\node{\( \boldsymbol + \)}
;
&
\draw
++(0,0.5)
node[arn_r](abs){\( \lambda \)}
    child[level distance=1.5cm]{
        node[arn_e]{\( \mathcal H \)}
    }
;
\node[rectangle,dashed,draw,fit=(abs),
rounded corners=3mm,inner sep=8pt, bgreen, very thick] {};
\\[.5em]
};
\end{tikzpicture}
\caption{\label{fig:marking:head:abstractions}Marking head abstractions in plain terms.}
\end{figure}

The bivariate generating function \( L_\infty(z, u) \) corresponding to
plain \lterms~with marked head abstractions satisfies therefore
\begin{equation}\label{eq:head-abs-bounded-gen-fun}
    L_\infty(z, u) = \dfrac{1}{1 - zu} \left(
        \dfrac{z}{1 - z} + z L_\infty(z, 1)^2
    \right).
\end{equation}
In such a form we immediately note that the dominant singularity $\rho(u)$ of
$L_\infty(z,u)$ does not depend on $u$. In fact, it is constant and equal to
$\rho$ (i.e.~the dominant singularity of $L_\infty(z)$,
see~\eqref{eq:plain-lambda-terms-dominant-singularity}) as, by construction,
$L_\infty(z,1) = L_\infty(z)$.

Following the fact that $L_\infty(z)$ admits a Puiseux expansion near $z = \rho$
we can represent ${L_\infty(z,1)}^2$ as a corresponding Puiseux series in form of
\begin{equation}\label{eq:plain:terms:square:root:puiseux}
    {L_\infty(z, 1)}^2 = \alpha - \beta \sqrt{1 - \dfrac{z}{\rho}}
    + O \left( \left|
        1 - \dfrac{z}{\rho}
    \right| \right).
\end{equation}
Given~\eqref{eq:head-abs-bounded-gen-fun} we further note that
\begin{equation}
    {L_\infty(z, u)} = \widetilde\alpha(u) - \left(\dfrac{\rho}{1- \rho u}\right) \beta \sqrt{1 - \dfrac{z}{\rho}}
    + O \left( \left|
        1 - \dfrac{z}{\rho}
    \right| \right)
\end{equation}
for $u$ fixed sufficiently close to $1$.

At this point, we apply~\autoref{proposition:discrete:limit:laws}
and find that the limit probability generating function $p(u)$
associated with $L_\infty(z,u)$ satisfies
\begin{equation}
    p(u) = \dfrac{1 - \rho}{1 - \rho u}
\end{equation}
which indeed corresponds to the asserted limit geometric
distribution~\eqref{eq:plain:terms:head:abstractions:distribution} of $X_n$.
\end{proof}

\begin{remark}
The mean number $\mu_n$ of head abstractions in a random \lterm~of size $n$ satisfies
\begin{equation}\label{eq:plain:terms:head:abstractions:mean}
    \mu_n \xrightarrow[n\to\infty]{} \frac{\rho}{1-\rho}.
\end{equation}
The limit mean~\eqref{eq:plain:terms:head:abstractions:mean} is close to
    $0.4196$. Consequently, sufficiently large plain terms have, on average, less than one
    head abstraction. Such a result stands in sharp contrast
    to the canonical representation of David et al.~\cite{dgkrtz} where the
    number of head abstractions in a random (closed) \lterm~of size $n$ is at
    least of order $\sqrt{n/\log n}$; in particular, it is a
    moderately growing function of $n$.
\end{remark}

\subsection{De Bruijn index values in plain lambda terms}
\label{subsec:debruijn:plain}
In the current subsection we focus on the distribution of de~Bruijn index values
in random \lterms.

\begin{prop}
    \label{proposition:dbindices:plain}
    Let $X_n$ be a random variable denoting the de~Bruijn index value $m$ of an
    index $\idx{m}$ taken uniformly at random from a random plain \lterm~of size
    $n$. Then, $X_n$ converges in law to a geometric distribution $\Geom(\rho)$
    with parameter $\rho$. Specifically,
    \begin{equation}
        \mathbb P(X_n = m) \xrightarrow[n\to\infty]{} \mathbb P(\Geom(\rho) =
        m) = (1 - \rho) \rho^m.
    \end{equation}
\end{prop}

\begin{proof}
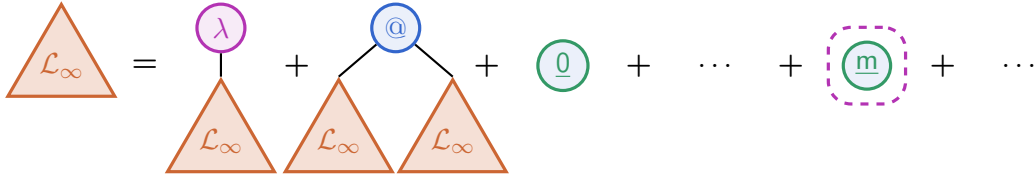
\begin{figure}[hbt]
\centering
\begin{tikzpicture}[>=stealth',level/.style={sibling distance = 1.5cm/#1,
  level distance = 1.5cm, thick}]
\draw
node[triangle]{\( \classL \)}
++(1.1,0)
node{\( \boldsymbol = \)}
++(1.0,0.5)
node(up)[arn_r]{\( \lambda\)}
    child{
        node[triangle]{\( \classL \)}
    }
++(1.0,-0.5)
node{\( \boldsymbol + \)}
++(1.3, 0.5)
node(up)[arn_n]{\( @ \)}
    child[child anchor = north]{
        node[triangle]{\( \classL \)}
    }
    child[child anchor = north]{
        node[triangle]{\( \classL \)}
    }
++(1.2,-0.5)
node{\( \boldsymbol + \)}
++(1,0)
    node[arn_g](var1){\( \idx{0} \)}
++(1.0,0)
node{\( \boldsymbol + \)}
++(1,0)
node(var2){\( \ldots \)}
++(1.0,0)
node{\( \boldsymbol + \)}
++(1,0)
    node[arn_g](var3){\( \idx{m} \)}
++(1.0,0)
node{\( \boldsymbol + \)}
++(1,0)
node(var4){\( \ldots \)}
;
\node[rectangle,dashed,draw,fit=(var3),
rounded corners=3mm,inner sep=5pt, bviolet, very thick] {};
\end{tikzpicture}
    \caption{\label{fig:marking:indices:m}Marking the index $\idx{m}$ in plain
    \lterms.}
\end{figure}

    Let $V^{(m)}_n$ be a random variable denoting the number of de~Bruijn
    indices $\idx{m}$ in a plain \lterm~of size $n$. Marking the index $\idx{m}$
    in the specification for plain terms, see~\autoref{fig:marking:indices:m},
    we note that the bivariate generating function $L^{(m)}_\infty(z,u)$
    associated with $V^{(m)}_n$ satisfies

    \begin{equation}\label{eq:plain:terms:indices:spec}
        L^{(m)}_\infty(z, u) = z L^{(m)}_\infty(z, u) + z {L^{(m)}_\infty(z, u)}^2 + \dfrac{z}{1 -
        z} + (u - 1)z^{m + 1}.
    \end{equation}

    Denote ${\deriv{u}{L^{(m)}_\infty(z, u)}}\at{u=1}$ as
    $D^{(m)}_\infty(z)$. Then, taking the partial derivative $\partial{u}$
    at $u=1$ of both sides of~\eqref{eq:plain:terms:indices:spec} we arrive at
    \begin{equation}\label{eq:plain:terms:indices:spec:deriv}
        D^{(m)}_\infty(z) = z D^{(m)}_\infty(z) + 2 z D^{(m)}_\infty(z)
        L_\infty(z) + z^{m + 1}
    \end{equation}
    as $L^{(m)}_\infty(z,1) = L_\infty(z)$ for each $m \geq 0$,
    cf.~\eqref{eq:plain:terms:indices:spec}
    and~\eqref{eq:genfun-plain-lambda-terms}.  Note that $[z^n]D^{(m)}(z)$
    corresponds to the weighted sum over all plain \lterms~of size $n$ where
    each term comes with weight equal to the total number of occurrences
    of index $\idx{m}$ in it.

    Let $S_\infty(z,w) = \sum_{m \geq 0} D^{(m)}_\infty(z) w^m$.
    Taking the weighted sum over all $m \geq 0$ of both sides
    of~\eqref{eq:plain:terms:indices:spec:deriv} such that weight corresponding
    to $m$ is $w^m$ we obtain
    \begin{equation}\label{eq:plain:terms:indices:spec:deriv:ii}
        S_\infty(z,w) = z S_\infty(z,w) + 2 z S_\infty(z,w)
        L_\infty(z) + \dfrac{z}{1- z w}.
    \end{equation}
    Consequently, $[z^n w^m]S_\infty(z,w)$ stands for $[z^n]D^{(m)}(z)$ whereas
    $[z^n]S_\infty(z,1)$ denotes the weighted sum of all \lterms~of size $n$ where
    each term has weight equal to its total number of variables. In other words,
    variable $w$ in $S_\infty(z,w)$ marks the probability mass function corresponding to $X_n$.
    Solving~\eqref{eq:plain:terms:indices:spec:deriv:ii} we find that
    \begin{equation}\label{eq:plain:terms:indices:spec:deriv:sum}
        S_\infty(z,w) = \dfrac{1}{1- z w} \left(\dfrac{z}{1-z-2 z
        L_\infty(z)}\right) = \dfrac{z}{1- z w}
        \left(\sqrt{{(1-z)}^2-\dfrac{4z^2}{1-z}}\right)^{-1}
    \end{equation}
    where the latter equality follows from the
    formula~\eqref{eq:genfun-plain-lambda-terms} for $L_\infty(z)$.

    Furthermore, given the known Puiseux expansion of the right-hand side
    square-root expression (see, e.g.~\autoref{prop:asymptotic:plain:terms}) we
    easily note that $S_\infty(z,w)$ admits a Puiseux series expansion required
    for~\autoref{proposition:discrete:limit:laws}.  Finally, a routine
    computation verifies the asserted geometric limit distribution $\Geom(p)$
    corresponding to the variable $X_n$.
\end{proof}

\begin{remark}
The mean value $\mu_n$ of a random index with a random plain terms satisfies
    \begin{equation}\label{eq:plain:terms:index:value:mean} \mu_n
    \xrightarrow[n\to\infty]{} \frac{\rho}{1-\rho}.  \end{equation} The limit
    value~\eqref{eq:plain:terms:index:value:mean}, coinciding in the natural
    size notion with the corresponding mean for head
    abstractions~\eqref{eq:plain:terms:head:abstractions:mean}, is close to
    $0.4196$. This result stands, again, is sharp contrast to the canonical
    model of David et al.~\cite{dgkrtz} in which variables (in closed terms)
    tend to be arbitrarily far from their binding abstractions.

    Let us point out that such a disparity is a consequence of the different
    combinatorial models for \lterms. In the canonical representation, the
    distance from a variable to its binding abstraction does not contribute to
    the weight of the corresponding variable (all variables have weight zero).
    On the other hand, in the de~Bruijn representation weights of bound indices
    are proportional to the distances to their binding abstractions.
    Consequently, de~Bruijn indices in large random \lterms~tend to be, on
    average, shallow. This central difference of both combinatorial models leads
    to remarkably contrasting asymptotic properties, including normalisation of
    large random \lterms, cf.~\cite{dgkrtz,Bendkowski2016,BendkowskiThesis}.
\end{remark}

\subsection{Leftmost-outermost redex search}\label{subsec:basic:redex:discovery}
In order to evaluate an expression represented by a \lterm~$M$ we iteratively
choose a \bredex~(i.e.~a subterm in form of $(`l P)Q$) in $M$ and contract it
using \breduction, see e.g.~\autoref{fig:beta:reduction:example}
(cf.~\autoref{subsec:lambda:calculus}).

\begin{figure}[hbt]
\centering
\begin{tikzpicture}[>=stealth',level/.style={sibling distance = 1.5cm,
    level distance = 1.0cm, thick}]
    \matrix[column sep = .5cm]{
\draw
node[arn_n]{\( @ \)}
child[sibling distance = 3.0cm]{
    node[arn_r](app1){\( \lambda \)}
    child{
        node[arn_n]{\( @ \)}
        child[sibling distance = 2.0cm]{
            node[arn_n]{\( @ \)}
            child{
                node[arn_g](db1){\( \underline 0 \)}
            }
            child{
                node[arn_n]{\( @ \)}
                child{
                    node[arn_g]{\( \underline 1 \)}
                }
                child{
                    node[arn_g](db2){\( \underline 0 \)}
                }
            }
        }
        child[sibling distance = 2.0cm]{
            node[arn_r]{\( \lambda \)}
            child{
                node[arn_g]{\( \underline 0 \)}
            }
        }
    }
}
child[child anchor = north, level distance = 1.8cm, sibling distance = 2.5cm]{
    node[triangle]{\( T \)}
}
;
\node[rectangle,dashed,draw,fit=(db1),
rounded corners=3mm,inner sep=4pt, bred, very thick] {};
\node[rectangle,dashed,draw,fit=(db2),
rounded corners=3mm,inner sep=4pt, bred, very thick] {};
\path[->] (db1) edge [bred,dashed,thick,bend left=48, in=120] node {} (app1);
\path[->] (db2) edge [bred,dashed,thick,bend right=48, in=-100, out=-80] node {} (app1);
&
\draw
++(0,-3)
node{\( \xrightarrow{\quad\beta\quad}\) }
;
&
\draw
node[arn_n]{\( @ \)}
child[sibling distance = 3.0cm]{
    node[arn_n]{\( @ \)}
    child[child anchor = north, sibling distance = 2.0cm, level distance = 1.8cm]{
        node[triangle]{\( T \)}
    }
    child[sibling distance = 2.0cm]{
        node[arn_n]{\( @ \)}
        child{
            node[arn_g]{\( \underline 1 \)}
        }
        child[child anchor = north, level distance = 1.8cm]{
            node[triangle]{\( T \)}
        }
    }
}
child[sibling distance = 3.0cm]{
    node[arn_r]{\( \lambda \)}
    child{
        node[arn_g]{\( \underline 0 \)}
    }
}
;
\\
};
\end{tikzpicture}
\caption{\label{fig:beta:reduction:example}An example of \( \beta \)-reduction
$(`l \idx{0} (\idx{1} \idx{0}) (`l \idx{0})) T \to_{`b} T (\idx{1} T) (`l
\idx{0})$.}
\end{figure}

The order of evaluation (i.e.~the order in which redexes are contracted) has a
crucial impact on the computational effect of repeated \breduction.  Consider
\lterms~$\Omega := (`l \idx{0} \idx{0})(`l \idx{0} \idx{0})$ and $P := (`l `l
\idx{0}) \Omega$. Note that $\Omega$ can be evaluated \emph{ad infinitum} as
$\Omega \to_\beta \Omega$. If we choose the $\Omega$ redex over the main $(`l `l
\idx{0}) \Omega$ redex in $P$, then $P \to_\beta P$.  Otherwise, if we choose
the \emph{leftmost-outermost} redex $(`l `l \idx{0}) \Omega$ instead of $\Omega$
we note that $P \to_{\beta} `l \idx{0}$ since the topmost abstraction in $P$ has
no associated indices. We cannot continue $\beta$\nobreakdash-reducing $`l
\idx{0}$ as it contains no more redexes (such terms are in so-called
\emph{($\beta$\nobreakdash-)normal form}) and so we terminate the evaluation
process.

Terms from which it is possible to reach a $`b$\nobreakdash-normal form using
repeated \breduction~are said to be \emph{normalisable}. Since it is possible to
emulate computations of arbitrary Turing machines in \lcalculus~by means of
normalisable terms, it is undecidable to determine whether a given \lterm~is
normalisable or not, see~\cite{barendregt1984}.  Remarkably, the following
classical result provides a normalising evaluation strategy guaranteed to find
normal forms of normalisable \lterms.

\begin{theorem}[Standardisation theorem, see e.g~\cite{barendregt1984}] Let $N$
    be a normalising \lterm. Then, the iterated process of applying
    \breduction~to the leftmost-outermost redex in $N$ leads to the (unique)
    normal form of $N$.
\end{theorem}

When searching for the leftmost-outermost redex in a given \lterm, we traverse
the associated $`l$\nobreakdash-tree in a depth-first manner favouring left
branches of application nodes. Both cases of handling abstractions and indices
are trivial -- when visiting an abstraction node, we immediately recurse to its
subterm looking for the leftmost-outermost redex; indices cannot contain
\bredexes~hence after arriving at an index, we terminate the traversal.

The most interesting part of the traversal algorithm lies in visiting
application nodes.  Suppose that we are currently visiting an application node.  If
its left branch starts with an abstraction node, we terminate the traversal as
we have just found the leftmost-outermost redex.  Otherwise, we have two
possibilities based on whether the left branch is in $`b$\nobreakdash-normal
form or not.  If it is, we move into the left branch and, as we cannot find a
\bredex, return from the recursion moving to the corresponding right branch.
Otherwise if the left branch is not in $`b$\nobreakdash-normal form, we handle
it recursively, however since it contains a \bredex, we eventually terminate the search
before ever returning to visit the right branch.

Let us note that such a traversal algorithm, henceforth abbreviated to {\sf LO},
introduces some \emph{a priori} non-trivial computational overhead to the
execution cost of \breduction~in \lcalculus. If carried out on a
$`b$\nobreakdash-redex, {\sf LO} takes constant time to run. In contrast, when
carried out on a \lterm~in $`b$\nobreakdash-normal form, {\sf LO} traverses
nearly the whole \lterm. Such a varying traversal cost poses the natural
question of the average-case performance of {\sf LO}. In what follows, we show
that the execution cost of {\sf LO}, when viewed as a random variable ranging
over random \lterms, tends to a discrete limit law with constant expectation.
Consequently, finding the leftmost-outermost redex introduces, on average, only
a constant overhead to the cost of carrying out a single \breduction.

\begin{prop}\label{proposition:redex:discovery:plain}
    Let \( X_n \) be a random variable denoting the number of nodes visited by
    the {\sf LO} traversal algorithm while searching for the leftmost-outermost
    \bredex~in a random plain \lterm~of size $n$. Then, \( X_n \) converges in
    law to a discrete limiting distribution with computable probability
    generating function and constant expectation. The corresponding means
    $\mu_n$ satisfy (up to numerical approximation)
    \begin{equation}\label{eq:plain:terms:redex:discovery:mean}
        \mu_n \xrightarrow[n\to\infty]{} 6.222262521.
    \end{equation}
\end{prop}

\begin{proof}
    We start with providing a combinatorial specification for plain
    \lterms~marking all nodes that are visited by the leftmost-outermost redex
    traversal algorithm {\sf LO}. For that purpose we introduce the following
    three auxiliary classes:
\begin{itemize}
    \item $\mathcal{N}$ for the class of $\beta$-normal forms;
    \item $\mathcal{M}$ for the class of so-called \emph{neutral terms}, and
    \item $\mathcal{A}$ for the class of de~Bruijn indices and plain \lterms~starting
        with an application.
\end{itemize}
In order to give their combinatorial specification we follow the presentation
    of~\cite{Bendkowski2016} and give a joint description for the class
    $\mathcal{N}$ of $`b$\nobreakdash-normal forms and the associated class
    $\mathcal{M}$ of neutral terms.  A $`b$\nobreakdash-normal form is either a
    plain \lterm~starting with an abstraction followed by another
    $`b$\nobreakdash-normal form, or a neutral term. A neutral term, in turn, is
    either a de~Bruijn index, or an application of a neutral term to a
    $`b$\nobreakdash-normal form, see~\autoref{fig:normal:forms}.
\begin{figure}[hbt]
\centering
\begin{tikzpicture}[>=stealth',level/.style={sibling distance = 1.5cm/#1,
  level distance = 1.5cm, thick}]
  \matrix{
\draw
node[triangle]{\( \mathcal N \)}
;&
\draw
node{\( \boldsymbol = \)}
;&
\draw
++(0,0.5)
node(up)[arn_r]{\( \lambda\)}
    child{
        node[triangle]{\( \mathcal N \)}
    }
++(1,-0.5)
node{\( \boldsymbol + \)}
++(1.2, 0)
node[triangle_b]{ \( \mathcal M \) }
;
\\
\draw
node[triangle_b]{\( \mathcal M \)}
;&
\draw
node{\( \boldsymbol = \)}
;&
\draw
node[arn_g]{\( \mathcal D \) }
++(1,0)
node{\( \boldsymbol + \)}
++(1.6, +0.5)
node[arn_n]{ \( @ \) }
    child[child anchor = north]{
        node[triangle_b]{ \( \mathcal M \) }
    }
    child[child anchor = north]{
        node[triangle]{ \( \mathcal N \) }
    }
;
\\
};
\end{tikzpicture}
    \caption{\label{fig:normal:forms}Joint specification for $`b$-normal forms
        $\mathcal{N}$ and neutral terms $\mathcal{M}$.}
\end{figure}
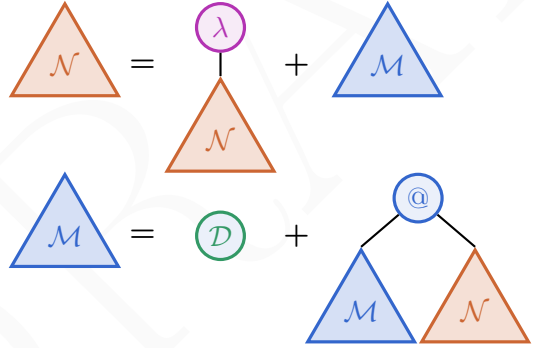

Such a specification provides the following system of functional equations
defining the generating functions $N(z)$ and $M(z)$ corresponding to the class
of $`b$\nobreakdash-normal forms and neutral terms, respectively:
\begin{align}\label{eq:normal:forms:neutral:terms}
    \begin{split}
        N(z) &= z N(z) + M(z)\\
        M(z) &= z M(z) N(z) + \frac{z}{1-z}.
    \end{split}
\end{align}
Solving~\eqref{eq:normal:forms:neutral:terms}
we note that $M(z)$ and $N(z)$ satisfy
\begin{equation}\label{eq:normal:forms:gf}
   N(z) = \frac{M(z)}{1-z} \quad \text{and} \quad
    M(z) = \frac{1-z-\sqrt{(1+z)(1-3z)}}{2z}.
\end{equation}
With both $N(z)$ and $M(z)$ at hand, we can now proceed and give the announced
specification for plain \lterms~with marked nodes visited during the execution
of {\sf LO}, see~\autoref{fig:search:time:plain}.
\begin{figure}[hbt]
\centering
\begin{tikzpicture}[>=stealth',level/.style={sibling distance = 1.5cm/#1,
  level distance = 1.5cm, thick}]
\matrix
{
\draw
node[triangle]{\( \mathcal L_\infty \)};
&
\draw
node{\(\boldsymbol =\)};
&
\draw
--(0,0.5)
node(abstraction1)[arn_r]{\( \lambda\)}
    child{
        node[triangle]{\( \mathcal L_\infty \)}
    }
++(1,-0.5)
node{\( \boldsymbol + \)}
++(1.2, 0)
node[triangle_b]{ \( \mathcal A \) }
;
\node[rectangle,dashed,draw,fit=(abstraction1),
rounded corners=3mm,inner sep=4pt, bblue , very thick] {};
\\
\draw
node[triangle_b]{\( \mathcal A \)};
&
\draw
node{\(\boldsymbol =\)};
&
\draw
node(var)[arn_g]{\( \mathcal D \) }
++(1,0)
node{\( \boldsymbol + \)}
++(2.0,0.5)
node[arn_n](app){\( @ \)}
    child[level distance=0.5cm, sibling distance=2.3cm]{
        node(abstraction2)[arn_r]{ \( \lambda \) }
        child{
            node[triangle, text depth = -1.2ex](lt)
            {\scalebox{0.8}{\hspace{-2.3ex}{
                \(
                    \left.
                        \mathcal L_\infty
                    \right|_{u=1}
                \)
            }}}
        }
    }
    child[level distance=1.5cm, child anchor=north]{
        node[triangle, text depth = -1.2ex](rt)
        {\scalebox{0.8}{\hspace{-2.3ex}{
            \(
                \left.
                    \mathcal L_\infty
                \right|_{u=1}
            \)
        }}}
    }
++(1.5, -0.5)
node{\( \boldsymbol + \)}
++(2.3, +0.5)
node[arn_n](app2){ \( @ \) }
    child[child anchor=north, sibling distance = 2.5cm]{
        node(nf1)[triangle_g]{ \( \mathcal M \) }
    }
    child[child anchor=north]{
        node[triangle]{ \( \mathcal L_\infty \) }
    }
++(1.6, -0.5)
node{\( \boldsymbol + \)}
++(1.6, +0.5)
node[arn_n](app3){ \( @ \) }
    child[child anchor=north]{
        node[triangle_v, text depth = -1.2ex]{
            \scalebox{0.8}{
                \hspace{-.4cm}
            \(
         \mathcal A \backslash \mathcal M
             \)} }
    }
    child[child anchor=north]{
        node[triangle, text depth = -1.2ex]
        {\scalebox{0.8}{\hspace{-2.3ex}{
            \(
                \left.
                    \mathcal L_\infty
                \right|_{u=1}
            \)
        }}}
    }
;
\node[rectangle,dashed,draw,fit=(app),
rounded corners=3mm,inner sep=4pt, bblue, very thick] {};
\node[rectangle,dashed,draw,fit=(app2),
rounded corners=3mm,inner sep=4pt, bblue, very thick] {};
\node[rectangle,dashed,draw,fit=(app3),
rounded corners=3mm,inner sep=4pt, bblue, very thick] {};
\node[rectangle,dashed,draw,fit=(nf1),
rounded corners=3mm,inner sep=8pt, bblue, very thick] {};
\node[rectangle,dashed,draw,fit=(var),
rounded corners=3mm,inner sep=8pt, bblue, very thick] {};
\node[rectangle,dashed,draw,fit=(abstraction2),
rounded corners=3mm,inner sep=4pt, bblue , very thick] {};
\\
};
\end{tikzpicture}
\caption{\label{fig:search:time:plain}Specification for plain \lterms~with
marked nodes following the execution of the redex finding traversal algorithm {\sf LO}.}
\end{figure}
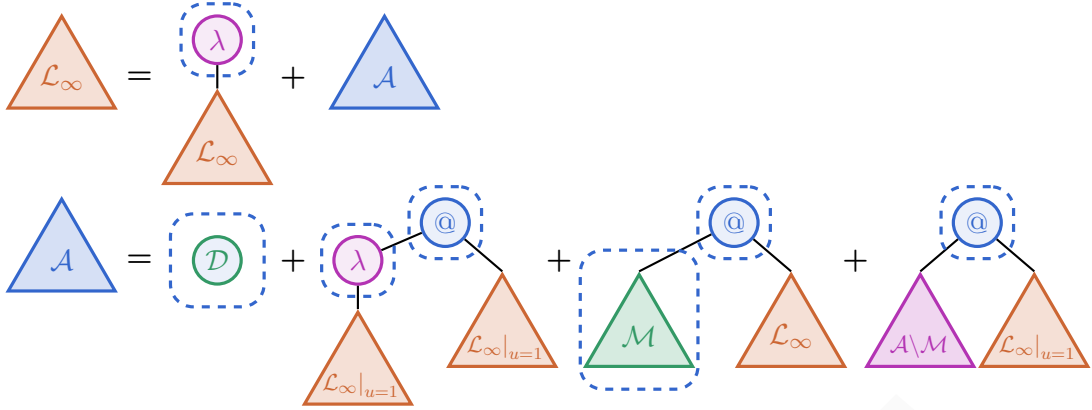

Let $T$ be a plain \lterm~in the class $\mathcal{A}$. Certainly, if $T$ is a
de~Bruijn index, we mark only its topmost atom (i.e.~the topmost successor or
$\textsf{0}$ if $T$ is equal to $\idx{0}$). Otherwise, $T$ starts with an
application. If $T$ is a \bredex, we mark two atoms -- the topmost application
and the abstraction starting the left branch of $T$. Remaining nodes are left
unmarked. If $T$ is not a redex however its left branch is a neutral term, we
mark the entire left branch as well as the topmost application. Finally, if the
left branch of $T$ is not a neutral term, we take the difference class
$\mathcal{A} \setminus \mathcal{M}$ of $\mathcal{A}$ and (marked) neutral
terms $\mathcal{M}$ for the left branch. The right branch remains unmarked.

Such a specification yields the following system of functional equations
defining the generating functions $L_\infty(z,u)$ and $A(z,u)$ corresponding to
classes $\mathcal{L}_\infty$ and $\mathcal{A}$, respectively:
\begin{align}\label{eq:plain:redex:spec}
    \begin{split}
        L_\infty(z,u) &= z u L_\infty(z,u) + A(z,u)\\
            A(z, u) &= \dfrac{z u}{1- z} + z^2 u^2 {L_\infty(z,1)}^2 + z u M(z
            u) L_\infty(z,u)\\
            &\qquad + z u \left(A(z,u) - M(z u)\right) L_\infty(z,1).
    \end{split}
\end{align}
Knowing \emph{a priori} that $L_\infty(z,1)$ corresponds to the generating
function for plain \lterms~\eqref{eq:genfun-plain-lambda-terms} we solve
system~\eqref{eq:plain:redex:spec} and find that $L_\infty(z,u)$ satisfies
\begin{align}\label{eq:plain:redex:genfun}
    L_\infty(z,u) &= \frac{z u M(z u) L_\infty(z,1) -\left(z u {L_\infty(z,1)}
    \right)^2-\frac{z u}{1-z}}{z u M(z u) -\left(1- z u L_\infty(z,1)\right)
    (1-z u)}.
\end{align}

What remains to finish the proof is to check that $L_\infty(z,u)$ meets the
premises of~\autoref{proposition:discrete:limit:laws}. Specifically, it admits a
single, fixed dominant singularity $\rho(u) = \rho$ and a corresponding Puiseux
expansion in form of
\begin{equation}
    L_\infty(z,u) = \alpha(u) - \beta(u)\sqrt{1-\dfrac{z}{\rho}}
    + O \left( \left|
        1 - \dfrac{z}{\rho}
    \right| \right).
\end{equation}

Denote the denominator expression of~\eqref{eq:plain:redex:genfun} as $F(z,u)$.
Note that $F(\rho,1) > 0$ and hence in a fixed neighbourhood of $u=1$ the
denominator $F(z,u)$ is non-zero.  Consequently, $L_\infty(z,u)$ shares its
(fixed) dominant singularity with $L_\infty(z,1)$.  The required form of the
Puiseux expansion of $L_\infty(z,u)$ follows as a consequence of the Puiseux
expansions of both $L_\infty(z,1)$ and its power ${L_\infty(z,1)}^2$,
see~\eqref{eq:plain:terms:square:root:puiseux}.  A direct computation gives
access to the corresponding probability generating function (omitted for
brevity) and also the specific limit
mean~\eqref{eq:plain:terms:redex:discovery:mean} of $X_n$.
\end{proof}

\begin{remark}\label{rem:motzkin:numbers:neutral:terms}
    The generating function $M(z)$ associated with neutral terms,
    see~\eqref{eq:normal:forms:neutral:terms}, also corresponds to the
    well-known class of Motzkin numbers enumerating, \emph{inter alia}, plane
    unary-binary trees, see e.g.~\cite[Note I.39, p.68]{flajolet09}.  We refer
    the curious reader to~\cite{Bendkowski2016} for a size-preserving
    correspondence between neutral terms of size $n$ and Motzkin trees with $n$
    nodes.
\end{remark}

\subsection{Height profile in plain lambda terms}
\label{subsection:height:profile:plain}
The goal of this section is to obtain some insight into the mean height profile
of plain \lterms. We distinguish essentially two different notions of height.
The first notion which we call \emph{unary height}, takes into account only the
number of abstractions above the considered node.  The second notion concerns
the \emph{natural height} of a tree, i.e.~the number of predecessors of a node
which can be either abstractions or applications. In both situations, the
semi-large powers theorem (see~\autoref{proposition:semilarge:power:theorem})
can be applied.  Consequently, the mean profile is always related to the
Rayleigh distribution.

We are interested in mean profile of different types of nodes.
In what follows we consider three types of mean profiles:
\begin{itemize}
    \item
  the mean (unary or natural) profile of leaves;
    \item
  the mean (unary or natural) profile of abstractions, and
    \item
  the mean (unary or natural) profile of applications.
\end{itemize}

\begin{prop}
\label{proposition:hprofile:vars:plain}
    Let \( H_n \) be a random variable denoting the unary (respectively natural)
    height of a randomly chosen variable (application or abstraction) in a
    random plain \lterm.  Then, with \( x \) in any compact subinterval of \(
    (0, +\infty) \), \( H_n \) admits a limiting Rayleigh distribution
    \begin{equation}
        \mathbb P(H_n = k) \sim
        \dfrac{C}{\sqrt n} \cdot \dfrac{x}{2} e^{-x^2/4}
        \quad \text{where} \quad
        x = \dfrac{k}{\sqrt n} \cdot C
    \end{equation}
with \( C \doteq 4.30187 \) for unary height, and \( C \doteq 1.27162 \) for
    the natural height. The average value of mean height is \( \sqrt{\pi n} / C
    \) whereas the peak value is attained at \( k^\ast = \sqrt{2n} / C \).

More specifically, the average number of
\begin{itemize}
    \item variables at unary height \( k \) is asymptotically equal to
        \( 2.839\, k e^{-C^2 k^2 / 4n} \);
    \item variables at natural height \( k \)
is asymptotically equal to \( 0.248\, k e^{-C^2 k^2/4n} \);
    \item abstractions
        at unary height \( k \) is asymptotically equal to
        \( 2.383 \, k e^{-C^2 k^2 / 4n}\);
    \item abstractions
        at natural height \( k \) is asymptotically equal to
        \( 0.208 \, k e^{-C^2 k^2 / 4n}\);
    \item applications
        at unary height \( k \) is asymptotically equal to
        \( 2.839 \, k e^{-C^2 k^2 / 4n}\);
    \item applications at natural height \( k \) is asymptotically equal to
        \( 0.248 \, k e^{-C^2 k^2 / 4n} \).
\end{itemize}
\end{prop}

\begin{proof}
We start with the unary height profile of variables.  Consider generating
    functions \( C_{k}(z, u) \) corresponding to plain \lterms~with $u$ marking
    the variables at the unary height $k$. These functions satisfy the following
    system of equations:
\begin{equation}
    \label{eq:marking:height:profile}
\begin{cases}
        C_{0}(z, u) = \dfrac{uz}{1-z} + z L_\infty(z) + z C_{0} (z, u)^2
        , & \text{ if } k = 0; \\[.3cm]
        C_{k}(z, u) = \dfrac{z}{1-z} + z C_{k - 1}(z, u) + z C_{k} (z, u)^2
        , & \text{ if } k  > 0. \\[.3cm]
\end{cases}
\end{equation}
    Taking partial derivatives of each equation
    in~\eqref{eq:marking:height:profile} with respect to \( u \), we obtain a
    linear system for derivatives of generating functions. Setting $u = 1$ we
    can solve this linear system and obtain
\begin{equation}
    \left.
    \dfrac{\partial}{\partial u}
    C_{k}(z, u) \right|_{u=1}
    =  \dfrac{1}{1-z} \left(
        \dfrac{z}{1 - 2 z L_\infty(z)}
    \right)^{k+1}.
\end{equation}
Furthermore, a direct computation provides
the following Puiseux series expansions as \( z \to \rho \):
\begin{equation}
    \dfrac{z}{1 - 2 z L_\infty(z)}
    \sim
    1 - \beta \sqrt{1 - \dfrac{z}{\rho}}
    \quad \text{and} \quad
    L_\infty(z)
    \sim
    \dfrac{1 - \rho}{2 \rho} - b_\infty \sqrt{1 - \dfrac{z}{\rho}}
\end{equation}
where \( \beta = 2 b_\infty \doteq 4.301868701457 \).
Consequently, the numbers \( M_{n,k} \) of variables at unary level
$k$ in a random plain lambda term of size $n$ satisfy
    \begin{equation}
    M_{n,k} =
    \dfrac{
        [z^n] \dfrac{1}{1 - z}
        \left(
            \dfrac{z}{1 - 2z L_\infty(z)}
        \right)^{k+1}
    }{
        [z^n] L_\infty(z)
    }
    .
    \end{equation}
    An application of the semi-large powers theorem
    (see~\autoref{proposition:semilarge:power:theorem}) and the transfer theorem
    (see~\autoref{proposition:transfer:theorem}) to the numerator and
    denominator, respectively, result in the following asymptotic estimate:
    \begin{equation}
    M_{n,k} \sim \dfrac{2k}{1 - \rho} e^{-\beta^2 k^2 / 4n}.
    \end{equation}
After normalising by the total sum $\sum_{k=0}^n M_{n,k}$ we obtain the
declared limiting distribution.

Next, we turn to the case of natural height profile of variables.
Consider generating functions \( C_k(z, u) \) where now \( u \) marks the
variables at the natural height \( k \), instead of the unary height. As in the
previous case, we obtain a system of equations
\begin{equation}
\begin{cases}
    C_0(z, u) = \dfrac{uz}{1 - z} + z L_\infty(z) + z L_\infty(z)^2,
    & \text{ if } k = 0;
    \\
    C_k(z, u) = \dfrac{z}{1 - z} + z C_{k-1}(z, u) + z C_{k-1}(z, u)^2,
    & \text{ if } k > 0.
\end{cases}
\end{equation}
Again, taking partial derivatives $\partial u$ at $ u = 1$ we can solve the
resulting system and find that
\begin{equation}
    \left.\dfrac{\partial}{\partial u} C_k(z, u) \right|_{u=1}
    = \dfrac{z}{1 - z}
    \left(
        z + 2z L_\infty(z)
    \right)^{k} .
\end{equation}
In this case, a direct computation verifies that the function \( z + 2z L_\infty(z) \)
admits a Puiseux series expansion in form of \( 1 - \gamma \sqrt{1 - z/\rho} +
O\left(|1-z/\rho|\right) \) where
\( \gamma = \beta \rho \doteq 1.27162265120953 \). Consequently, this estimate yields a Rayleigh distribution with parameter
\( 1.27162265120953 \). In particular, the average number \( M_{n,k} \) of variables at natural
height $k$ in a random plain \lterm~of size $n$ satisfies
\begin{equation}
    M_{n,k} \sim \dfrac{2 \rho^2}{1 - \rho} k e^{-\gamma^2 k^2 / 4n}.
\end{equation}

In order to mark remaining nodes, i.e.~abstractions and applications, it is
sufficient to change the first equation of the
system~\eqref{eq:marking:height:profile}. Accordingly, only the constant
multiple behind the mean tree width at level $k$ changes.  For abstractions, we
obtain, respectively,
\begin{equation}
    C_0(z, u) = \dfrac{z}{1 - z} + zu L_\infty(z) + z C_0(z, u)^2
\end{equation}
for unary height whereas
\begin{equation}
    C_0(z, u) = \dfrac{z}{1 - z} + z u L_\infty(z) + z L_\infty(z)^2
\end{equation}
for natural height. This change gives the constants
\( 2 L_\infty(\rho) = \frac{(1 - \rho)}{\rho} \doteq 2.383 \) for unary height, and
\( 2 \rho^2 L_\infty(\rho) = \rho(1 - \rho) \doteq 0.208 \) for the natural
height, respectively.

Similarly, marking applications yields a change in the first equation for the
generating function
\begin{equation}
    C_0(z, u) = \dfrac{z}{1 - z} + z L_\infty(z) + zu C_0(z, u)^2
\end{equation}
for unary height, and
\begin{equation}
    C_0(z, u) = \dfrac{z}{1 - z} + z L_\infty(z) + z u L_\infty(z)^2
\end{equation}
for natural height. We obtain the constants
\( 2 L_\infty^2(\rho) = \frac{(1 - \rho)^2}{ 2 \rho^2} \doteq  2.839 \) for
unary height, and
\( 2 \rho^2 L_\infty^2(\rho) = \frac{(1 - \rho)^2}{2} \doteq 0.248 \)
for natural height, respectively.

The mean value is obtained by using the integral approximation for the ratio of
sums \( k M_{n,k} / \sum_{k=0}^n M_{n,k} \) whereas the peak value is obtained by
finding the maximum value of \( M_{n,k} \) as a function of \( k \).
\end{proof}

%% file: 4-empirical-results.tex
\section{Empirical results}\label{sec:empirical:results}

The analysis of various combinatorial parameters corresponding to plain
\lterms~can be approached using standard proof templates, typical for algebraic
structures. Starting from a combinatorial specification associated with the
investigated parameter, its analysis follows as an, either direct or indirect,
examination of the resulting system of multivariate generating functions
and related singularities, see~\cite[Chapter IX]{flajolet09}.

Alas, such a general approach to the analysis of combinatorial parameters
related to plain \lterms~does not readily transcend to closed \lterms. Standard
analytic tools, such as Bender and Richmond's multivariate central limit theorem
(see~\autoref{proposition:multivariate:clt}) or the Drmota--Lalley--Woods
theorem (see~\autoref{proposition:irreducible:polynomial:systems}) require that
the investigated system of generating functions is, \emph{inter alia}, finite.
Although this is true for plain terms, closed \lterms~give rise to a more
involved, infinite system of generating
functions~\eqref{eq:m-open:terms:grammar} based on the hierarchical notion of
$m$\nobreakdash-openness.  Consequently, closed \lterms~escape the usual course
of parameter analysis, successfully carried out in the case of plain terms,
see~\autoref{sec:basic:statistics}.

In the current section we present two, somewhat complementary, empirical
approaches to the analysis of combinatorial parameters related to closed
\lterms.  We start with an experimental scheme based on the recent development
of efficient Boltzmann samplers for closed \lterms~due to Gittenberger, Bodini
and Gołębiewski~\cite{BodiniGitGol17}. We generate large, uniformly random
(conditioned on size) closed \lterms~and collect empirical data for various
interesting parameters related to generated terms.  The second approach is based
on the empirical evaluation of, appropriately truncated, systems of multivariate
generating functions. We compute the coefficients of the corresponding formal
power series and consequently analyse the distribution of investigated parameters
for relatively small term sizes.

The benefits of such empirical approaches are threefold. Firstly, the empirical
parameter distribution for large random \lterms~is closely related to its
theoretical, limiting counterpart; hence, develops solid intuitions underlying
the successful analysis of a broad class of combinatorial parameters,
see~\autoref{sec:advanced:marking}. Secondly, the empirical data for various
term sizes provides insight in the convergence rates at which the considered
random variables tend to respective limit laws, relating them in effect with
practically attainable term sizes.  Finally, experimental results for large
closed \lterms~provide strong evidence for conjectures regarding
practical, though even more advanced parameters escaping the restrictions of our
analysis.

\subsection{Empirical evaluation of Boltzmann samplers}
\label{subsec:empirical:boltzmann}
Boltzmann samplers are a prominent sampling framework meant for the random
generation of large combinatorial structures~\cite{Duchon04boltzmannsamplers}.
Given an admissible combinatorial specification, it is possible to construct an
appropriate sampler whose outcome are uniformly random (conditioned on size)
structures built according to the input specification. Although Boltzmann
samplers are guaranteed to return uniformly random structures, their eventual
size is not deterministic, but instead random. Following a proper calibration
during the construction of Boltzmann samplers, the randomness of their output
structure size can be controlled so that it is centred around a given (not
necessarily finite) expectation. A final rejection phase, dismissing structures
of undesired properties, such as for instance inadmissible size, provides the
means of controlling the generated structures.

In order to conduct our experiments, we have implemented three kinds of
Boltzmann
samplers\footnote{see~\url{https://hackage.haskell.org/package/lambda-sampler}.}
for plain, closed and so-called $h$\nobreakdash-shallow \lterms, i.e.~closed
terms in which de~Bruijn indices do not exceed the shallowness bound $h$, see
e.g.~\cite{gittenberger_et_al:LIPIcs:2016:5741,BendkowskiThesis}.  The
respective sampler for closed \lterms~follows the ideas
of~\cite{BodiniGitGol17}. For each investigated type of terms we sample $k =
50,000$ terms of sizes in the interval $[20,000; 50,000]$ using a dedicated
singular Boltzmann sampler (i.e.~a Boltzmann sampler with unbounded outcome size
expectation, see e.g.~\cite{BodGenRo2015}) whose controlling parameter is
calculated numerically with accuracy of order $10^{-9}$. For closed
$h$\nobreakdash-shallow terms, we fix $h = 30$. Finally, we record several
combinatorial parameters related to so obtained terms and use them in the
subsequent evaluation.

For each considered parameter and term type, we plot a histogram relating the
collected samples and the respective parameter values. The $x$\nobreakdash-axis
denotes the parameter value (either raw or aptly normalised) whereas the
$y$\nobreakdash-axis corresponds to the number of samples attaining the
associated value. Averages, variances and standard deviations corresponding to
the investigated parameter are rounded up to the fifth decimal point and
summarised in respective tables. For brevity, we include only histograms for
plain and closed \lterms. We comment on the missing $h$\nobreakdash-shallow
\lterms~at the end of the current section.

\subsubsection{Head abstractions}\label{subsubsec:head:abs}
We start with head abstractions, see~\autoref{subsec:basic:head:abs}.
\autoref{fig:experimental:plot:closed:headabs} depicts the distribution
histogram of the obtained data sets. The corresponding numerical approximations
of averages, variances and standard deviations are listed
in~\autoref{fig:experimental:plot:closed:headabs:mean}.

\begin{figure}[!htb]
\centering
\begin{subfigure}{.5\textwidth}
  \centering
  \includegraphics[width=\linewidth]{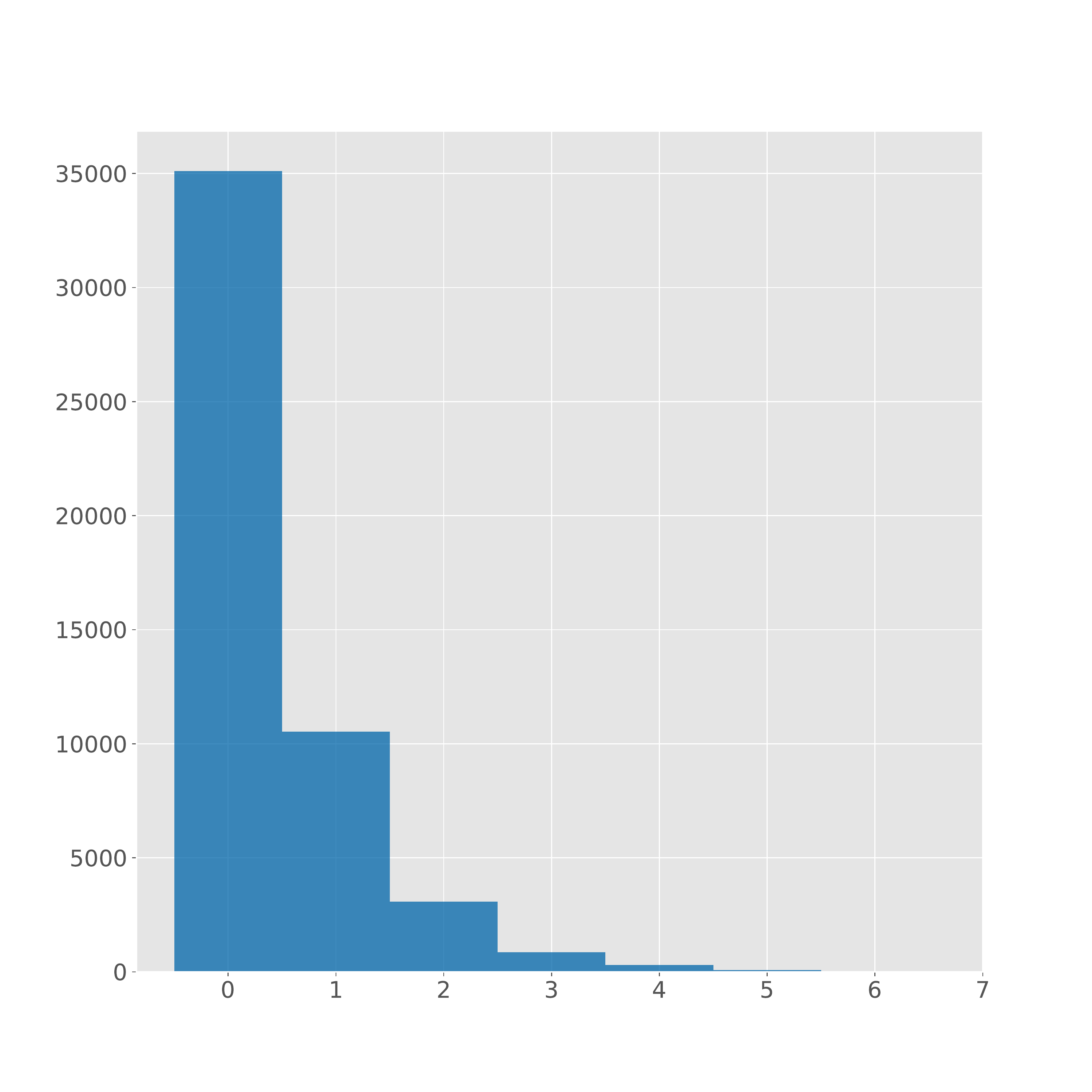}
  \caption{Plain \lterms.}
\end{subfigure}%
\begin{subfigure}{.5\textwidth}
  \centering
  \includegraphics[width=\linewidth]{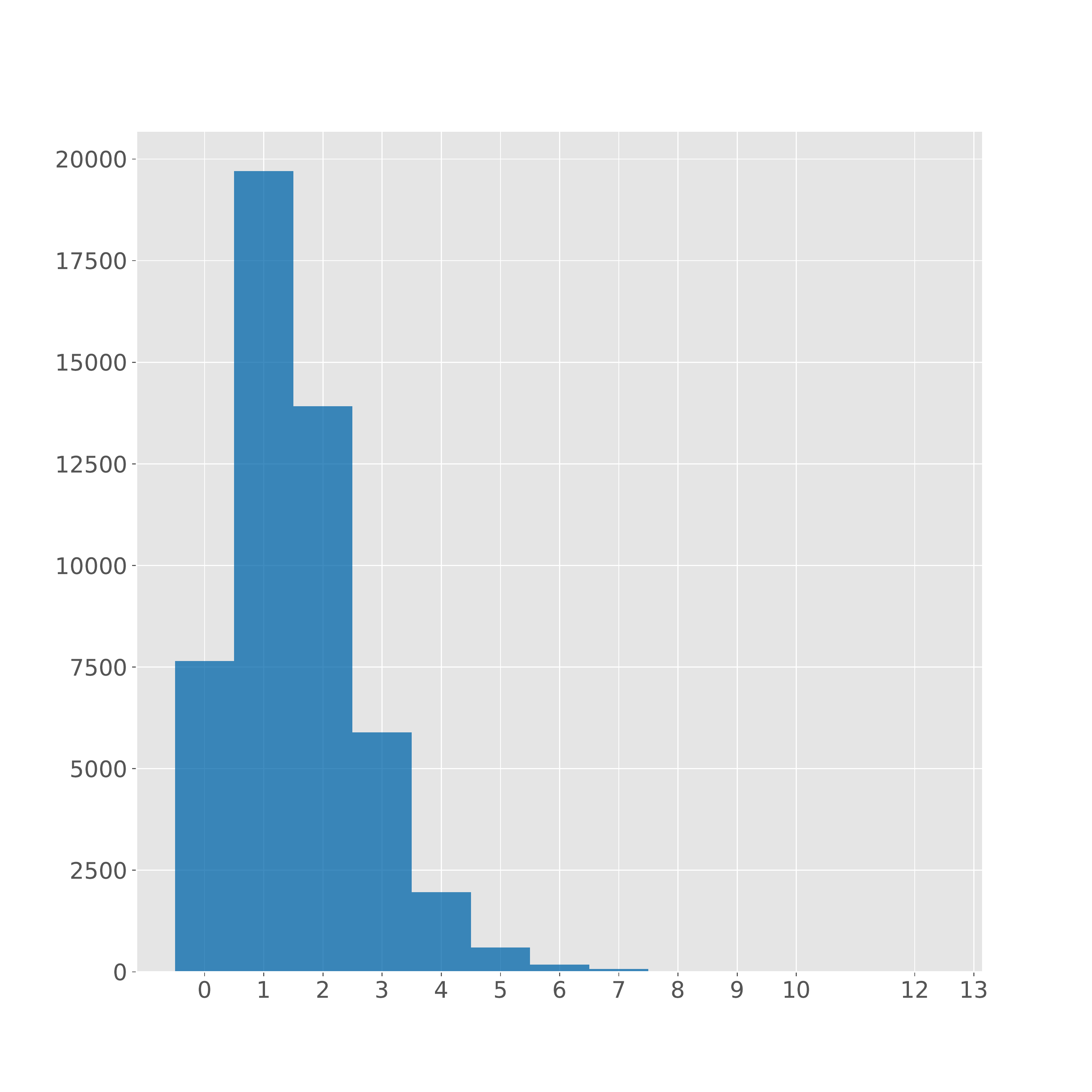}
  \caption{Closed \lterms.}
\end{subfigure}
\caption{Distribution histograms for head abstractions in plain and closed \lterms.}
\label{fig:experimental:plot:closed:headabs}
\end{figure}

\begin{table}[!htb]
    \caption{The average, variance and standard deviation
    corresponding to  head abstractions in plain, closed and $h$\nobreakdash-shallow
    \lterms~within the obtained data sets.}
\label{fig:experimental:plot:closed:headabs:mean}
\centering
\begin{tabular}{cccc}
    & Average & Variance & Std. dev. \\
  \hline
    plain & $0.42202$ & $0.59632$ & $0.77222$ \\
    closed & $1.55712$ & $1.30818$ & $1.14376$ \\
    $h$-shallow & $1.56644$ & $1.28803$ & $1.13491$ \\
\end{tabular}
\end{table}

\begin{remark}
The distribution corresponding to plain terms resembles a geometric law, as
    expected by~\autoref{prop:plain:terms:head:abstractions}. The observed
    average corresponds closely to the theoretical limit
    average~\eqref{eq:plain:terms:head:abstractions:mean}. Notably, head
    abstractions in closed \lterms~do not follow the same distribution as their
    plain counterpart.
\end{remark}

\subsubsection{Free variables in plain lambda terms}

When viewed as programs of an abstract programming language, variables in
\lterms~become formal arguments bound to abstractions introducing them in
respective name scopes (namespaces). In this perspective, free variable
occurrences correspond to expressions available in the global namespace such as for
instance predefined operators or constants, see e.g.~\cite{peytonJones1987}.

\autoref{fig:experimental:plot:closed:freevars} depicts the distribution
histogram of the obtained data set for plain terms (recall that closed
\lterms~contain, by definition, no free variable occurrences). The corresponding
numerical approximations of averages, variances and standard deviations are
listed in~\autoref{fig:experimental:plot:closed:freevars:mean}.

\begin{figure}[!htb]
  \centering
  \includegraphics[width=0.5\linewidth]{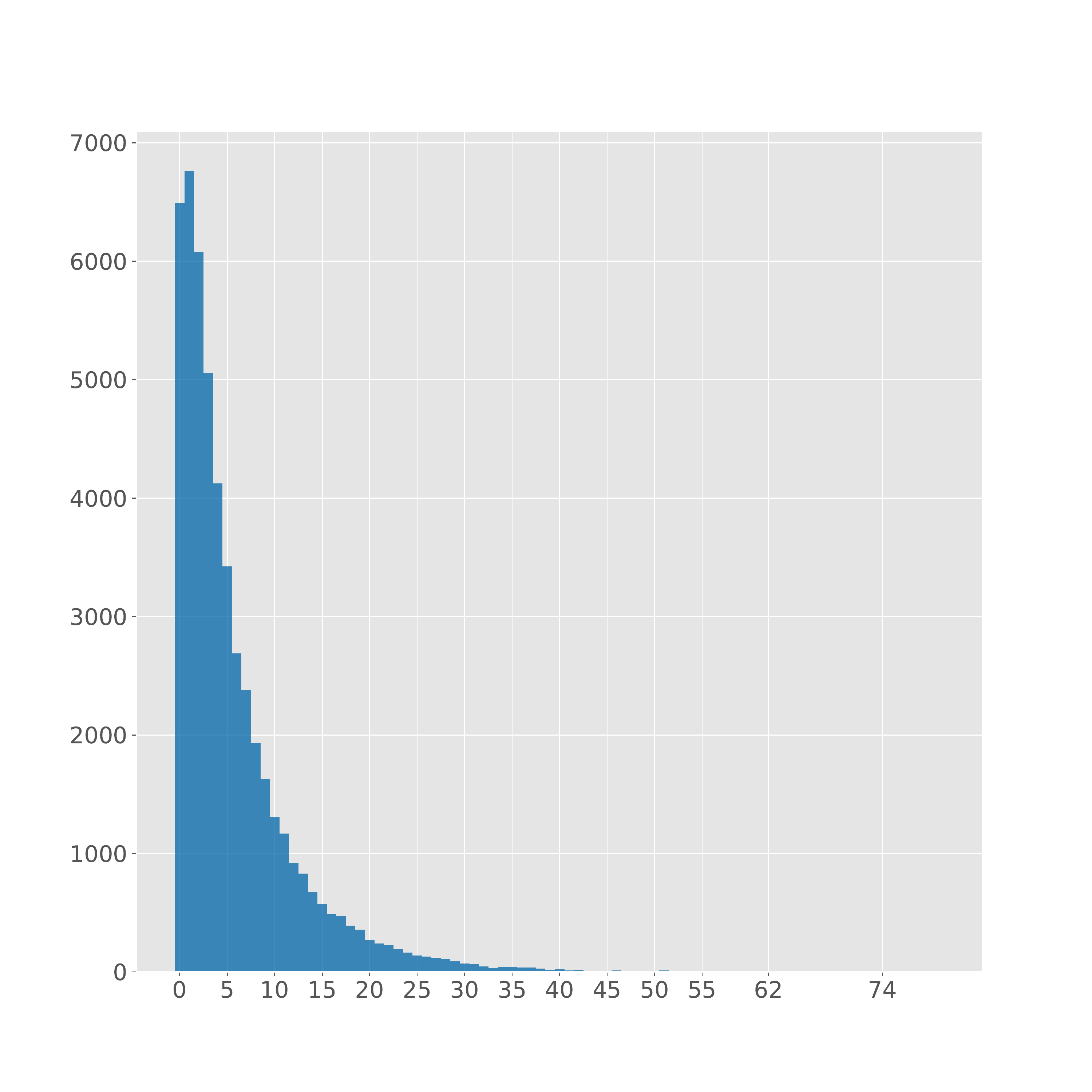}
\caption{Distribution histogram for free variables in plain \lterms.}
\label{fig:experimental:plot:closed:freevars}
\end{figure}

\begin{table}[!htb]
    \caption{The average, variance and standard deviation
    corresponding to free variables in plain \lterms~within the obtained data set.}
\label{fig:experimental:plot:closed:freevars:mean}
\centering
\begin{tabular}{cccc}
    & Average & Variance & Std.\ dev. \\
  \hline
    plain & $5.74976$ & $43.40258$ & $6.58806$ \\
\end{tabular}
\end{table}

\begin{remark}
The empirical distribution of free variables in plain terms resembles closely a
    geometric law ${\mathbb{P}(\Geom = k) = {(1-p)}^k p}$ with parameter $p
    \doteq 0.14815$. Note however that the observed distribution is not
    geometric; its global maximum is not attained at the parameter value
    corresponding to the lack of free variables,
    see~\autoref{fig:experimental:plot:closed:freevars}.  Remarkably,
    the observed value $p$ is a quite good estimate for the probability that a
    sufficiently large random plain \lterm~is closed. Using the numerical
    estimate for the multiplicative constant $C$ in the asymptotic growth rate
    $C \cdot n^{-3/2} \rho^{-n}$ corresponding to closed \lterms~\cite[Section
    6.1]{gittenberger_et_al:LIPIcs:2016:5741} one can find that this probability
    is in fact close to $0.12840$.
\end{remark}

\subsubsection{Leftmost-outermost redex search}
The next parameter we evaluate is the cost of finding the leftmost-outermost
$`b$\nobreakdash-redex in plain and closed \lterms,
see~\autoref{subsec:basic:redex:discovery}.
\autoref{fig:experimental:plot:closed:redex} depicts the distribution histogram
of the obtained data sets. Corresponding numerical approximations of averages,
variances and standard deviations are given
in~\autoref{fig:experimental:plot:closed:redex:mean}.

\begin{figure}[!htb]
\centering
\begin{subfigure}{.5\textwidth}
  \centering
  \includegraphics[width=\linewidth]{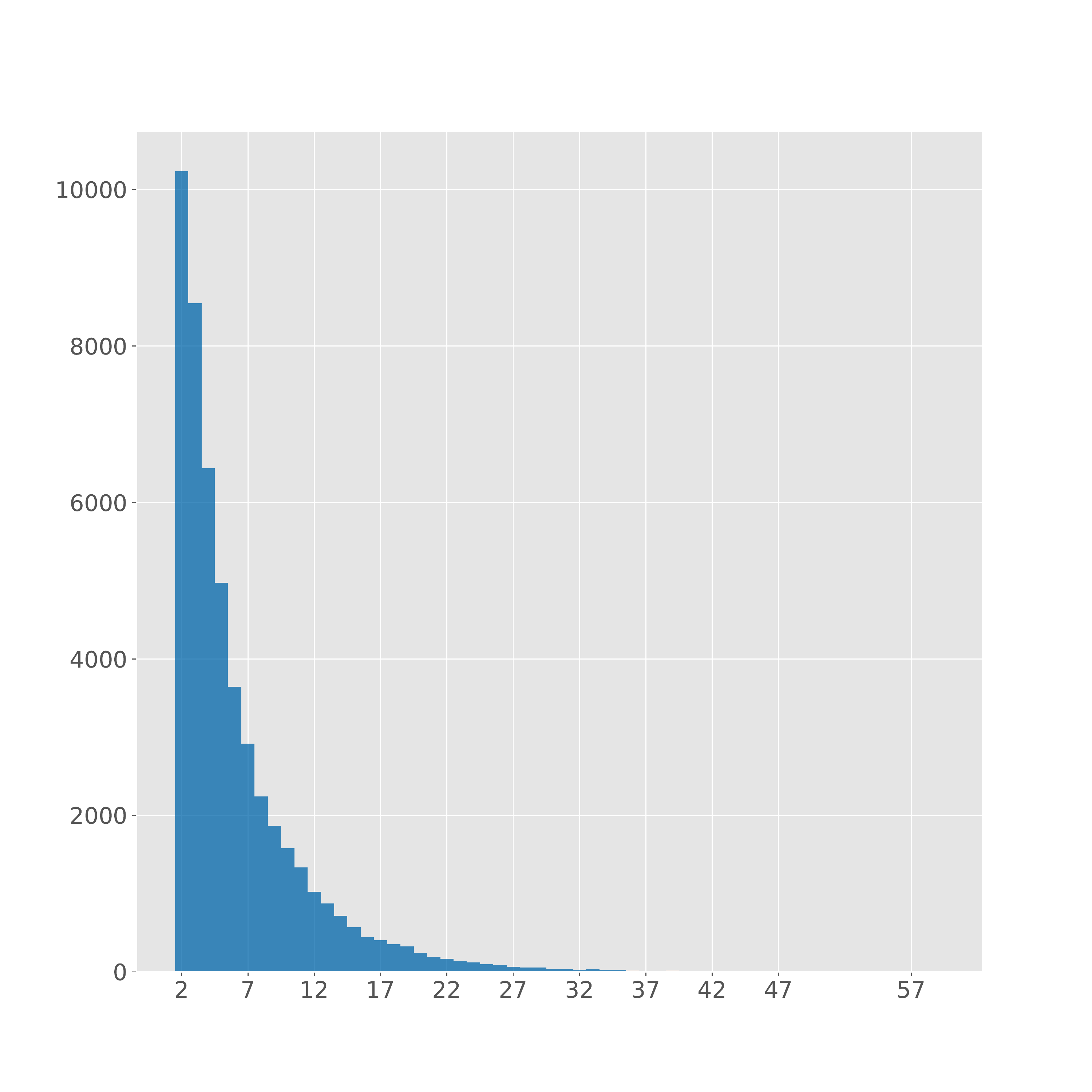}
  \caption{Plain \lterms.}
\end{subfigure}%
\begin{subfigure}{.5\textwidth}
  \centering
  \includegraphics[width=\linewidth]{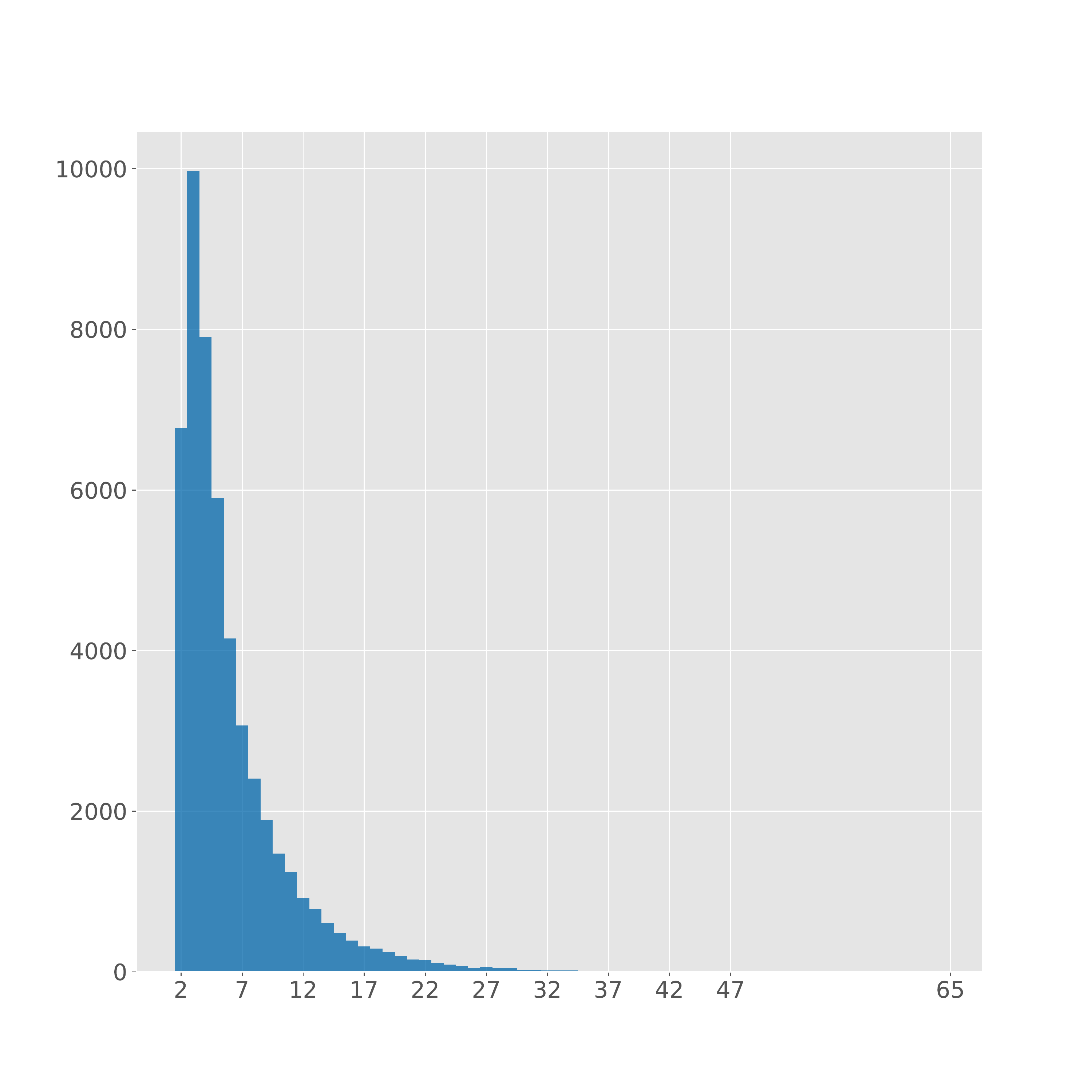}
  \caption{Closed \lterms.}
\end{subfigure}
\caption{Distribution histograms for leftmost-outermost redex search time in plain and closed \lterms.}
\label{fig:experimental:plot:closed:redex}
\end{figure}

\begin{table}[!htb]
    \caption{The average, variance and standard deviation
    corresponding the search cost of the leftmost-outermost redex in plain, closed and $h$\nobreakdash-shallow
    \lterms~within the obtained data sets.}
\label{fig:experimental:plot:closed:redex:mean}
\centering
\begin{tabular}{cccc}
    & Average & Variance & Std. dev. \\
  \hline
    plain & $6.22118$ & $26.72478$ & $5.1696$ \\
    closed & $6.06994$ & $21.78197$ & $4.66712$ \\
    $h$-shallow & $6.12196$ & $22.40393$ & $4.73328$ \\
\end{tabular}
\end{table}

\begin{remark}
The empirical histograms corresponding to plain and closed terms do not follow
    the same law, similarly to the case of head abstractions,
    see~\autoref{subsubsec:head:abs}. However, here the observed variance is in
    both cases significantly larger. Note that all discovered averages are remarkably
    close to each other.
\end{remark}

\subsubsection{Open subterms}
In the current subsection we are interested in the degree to which variables
connect various subterms in random \lterms. In other words, in the proportion of
open subterms in a random term. Let us note that such a parameter provides some
insight in the extend to which subterms, viewed as components of the entire term
(i.e.~functional program) depend syntactically on each other. In this
perspective, terms consisting of many closed subterms correspond to programs
with equally many independent subprograms. On the other hand, \lterms~with few
open subterms represent computations in which various parts of the program
depend on each other through common variable usage.  Since each index $\idx{n}$
consists of $n$ open proper subterms which do not contribute to the intended
degree of functional dependence (variables in functional programs are atomic
expressions) for our current considerations we assume that $\idx{n}$ is itself
atomic, i.e.~does not have proper subterms.  Consequently, the total number of
subterms of a given \lterm~$T$ becomes equal to the number of indices,
applications and abstractions occurring in $T$, without accounting for
successors.

The resulting distribution histograms for plain and closed \lterms~are depicted
in~\autoref{fig:experimental:plot:closed:opensubterms}. The $x$-axis denotes the
(normalised with respect to the total number of subterms as defined above)
number of open subterms whereas the $y$-axis corresponds to the number of terms
attaining the corresponding value. Respective averages, variances and standard
deviations are listed in~\autoref{fig:experimental:plot:closed:opensubterms:mean}.

\begin{figure}[!htb]
\centering
\begin{subfigure}{.5\textwidth}
  \centering
  \includegraphics[width=\linewidth]{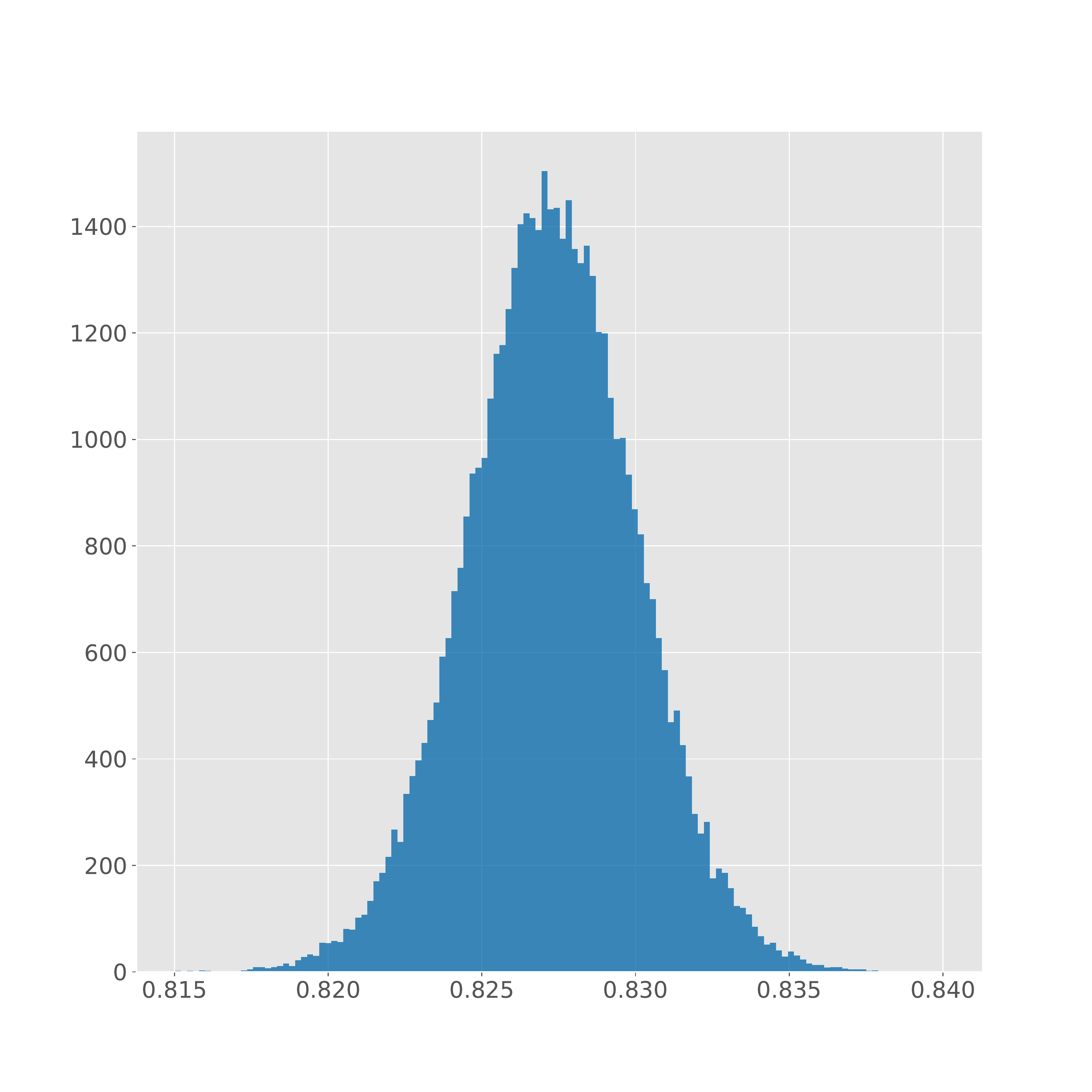}
  \caption{Plain \lterms.}
\end{subfigure}%
\begin{subfigure}{.5\textwidth}
  \centering
  \includegraphics[width=\linewidth]{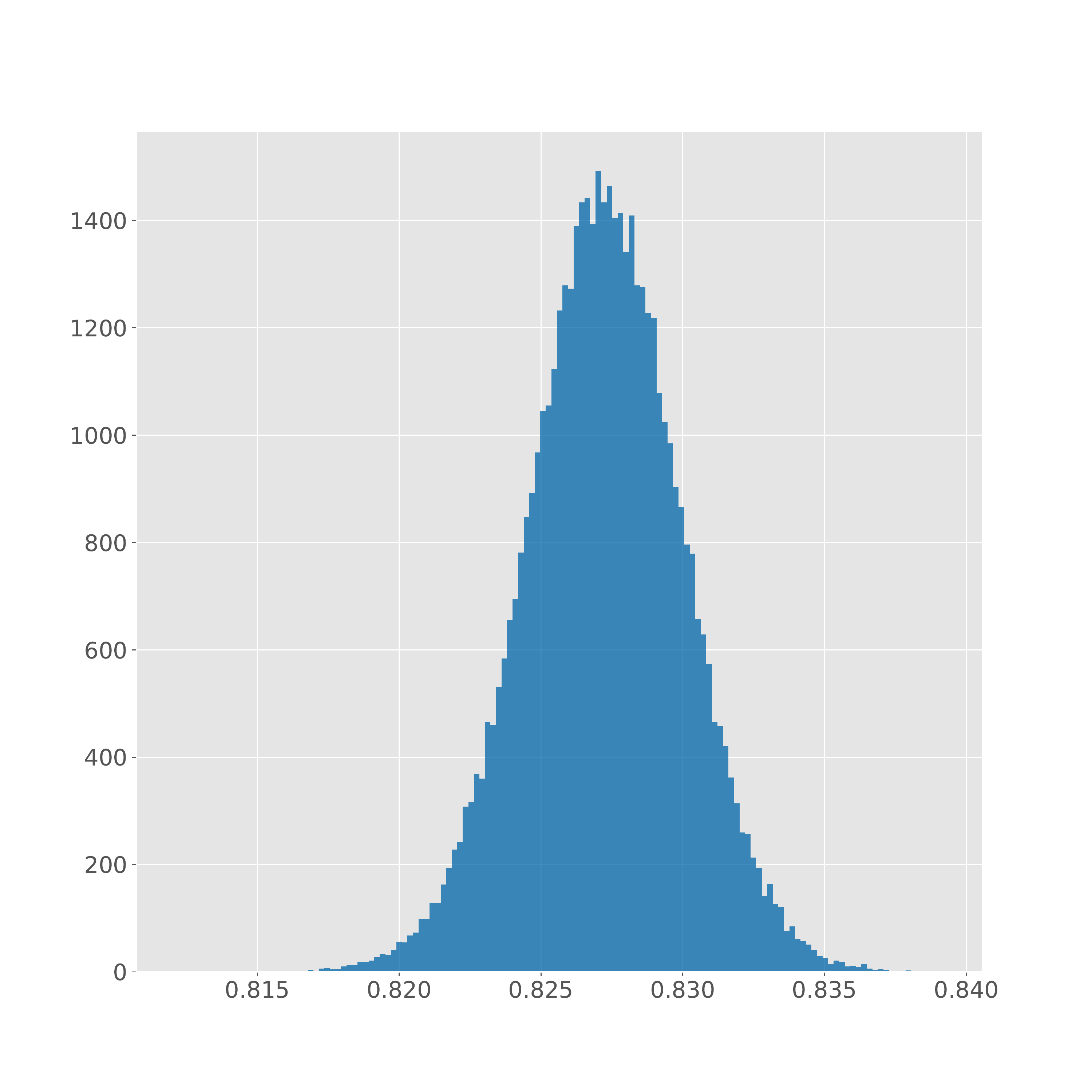}
  \caption{Closed \lterms.}
\end{subfigure}
\caption{Distribution histograms for open subterms in plain and closed \lterms.}
\label{fig:experimental:plot:closed:opensubterms}
\end{figure}

\begin{table}[!htb]
    \caption{The average, variance and standard deviation corresponding the
    proportion of open subterms in plain, closed and $h$\nobreakdash-shallow
    \lterms~within the obtained data sets.}
\label{fig:experimental:plot:closed:opensubterms:mean}
\centering
\begin{tabular}{cccc}
    & Average & Variance & Std. dev. \\
  \hline
    plain & $0.82728 $ & $0.00001$ & $0.00276$ \\
    closed & $0.82723$ & $0.00001$ & $0.00275$ \\
    $h$-shallow & $0.82724$ & $0.00001$ & $0.00276$ \\
\end{tabular}
\end{table}

\begin{remark}
Both empirical distributions resemble Gaussian laws with respective means and
    variances, see~\autoref{fig:experimental:plot:closed:opensubterms:mean}.
    The distributions seem to be concentrated around their means whereas the observed
    variances and standard deviations are strikingly modest and (almost)
    identical, though still positive. The high expectations in the order of $0.8272$
    suggest that the vast majority of subterms in a random \lterm~is open,
    independently of whether it is itself open or closed.
\end{remark}

\subsubsection{Binding abstractions}
The next parameter we consider is the proportion of abstractions binding
variables.~\autoref{fig:experimental:plot:closed:bindingabs} depicts
the empirical distribution histograms for plain and closed \lterms,
respectively.  The $x$-axis denotes the proportion of binding abstractions among
all abstractions. Corresponding averages, variances and standard deviations are
presented in~\autoref{fig:experimental:plot:closed:bindingabs:mean}.

\begin{figure}[!htb]
\centering
\begin{subfigure}{.5\textwidth}
  \centering
  \includegraphics[width=\linewidth]{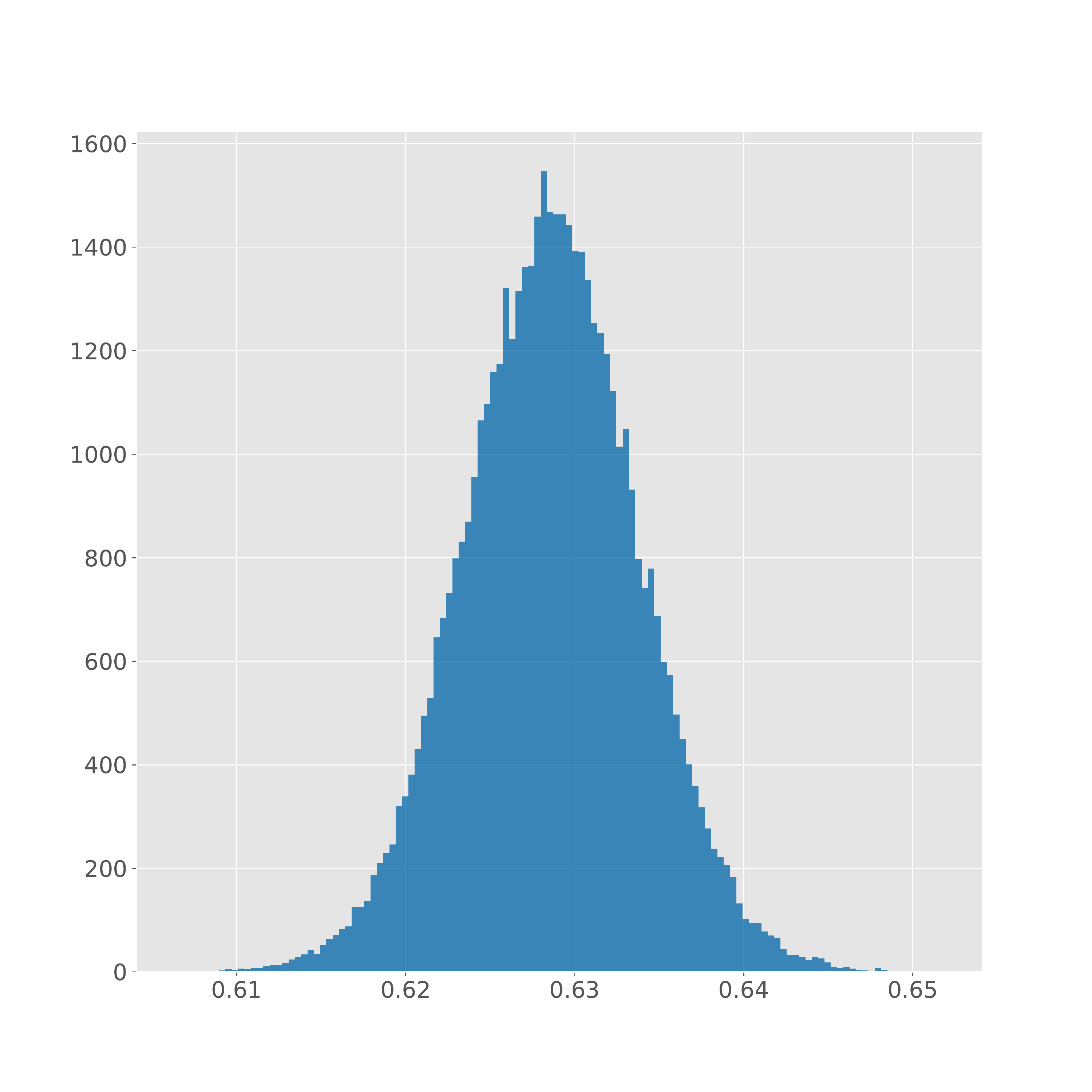}
  \caption{Plain \lterms.}
\end{subfigure}%
\begin{subfigure}{.5\textwidth}
  \centering
  \includegraphics[width=\linewidth]{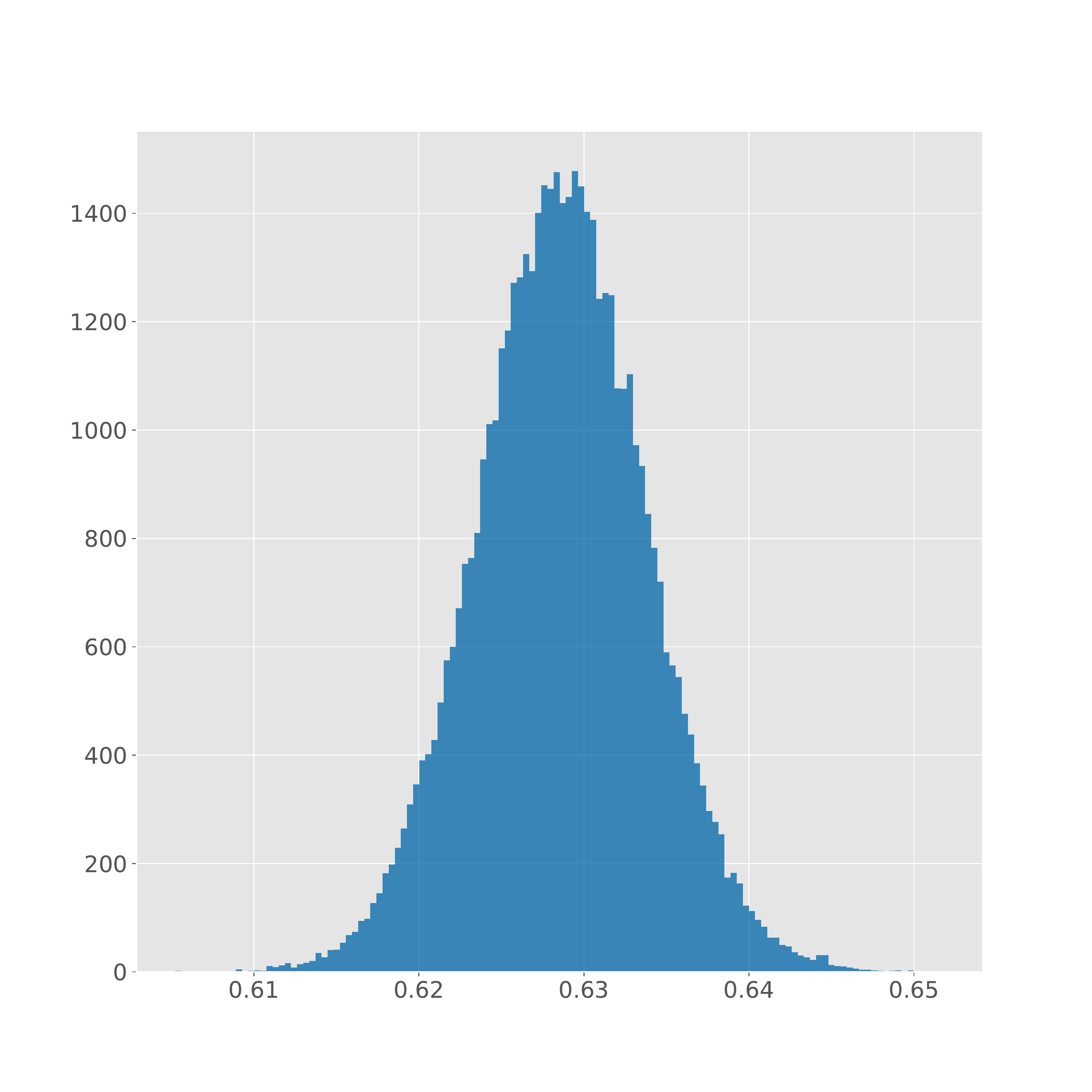}
  \caption{Closed \lterms.}
\end{subfigure}
\caption{Distribution histograms for binding abstractions in plain and closed \lterms.}
\label{fig:experimental:plot:closed:bindingabs}
\end{figure}

\begin{table}[!htb]
    \caption{The average, variance and standard deviation
    corresponding to the proportion of  binding abstractions in plain, closed and $h$\nobreakdash-shallow
    \lterms~within the obtained data sets.}
\label{fig:experimental:plot:closed:bindingabs:mean}
\centering
\begin{tabular}{cccc}
    & Average & Variance & Std. dev. \\
  \hline
    plain & $0.62862$ & $0.00003$ & $0.00521$ \\
    closed & $0.62863$ & $0.00001$ & $0.00518$ \\
    $h$-shallow & $0.62862$ & $0.00001$ & $0.00512$ \\
\end{tabular}
\end{table}

\begin{remark}
Like in the case of open subterms, the distribution of binding abstractions in
    plain and closed \lterms, respectively, mirrors a Gaussian law.
    Corresponding variances and standard deviations, though positive, are
    minuscule. The listed averages of order $0.6286$ suggest that most
    abstractions are binding, both in plain and closed terms.
\end{remark}

\subsubsection{Maximal number of variables bound to a single abstraction}
In the current subsection we turn to extremal statistics related to binding
abstractions. Specifically, we investigate the maximal number of variables bound
to a single abstraction in random \lterms.
\autoref{fig:experimental:plot:closed:maxvarabs} illustrates the related
distribution histogram for plain and closed \lterms, respectively. Numerical
approximations of averages, variances and standard deviations
are given in~\autoref{fig:experimental:plot:closed:maxvarabs:mean}.

\begin{figure}[!htb]
\centering
\begin{subfigure}{.5\textwidth}
  \centering
  \includegraphics[width=\linewidth]{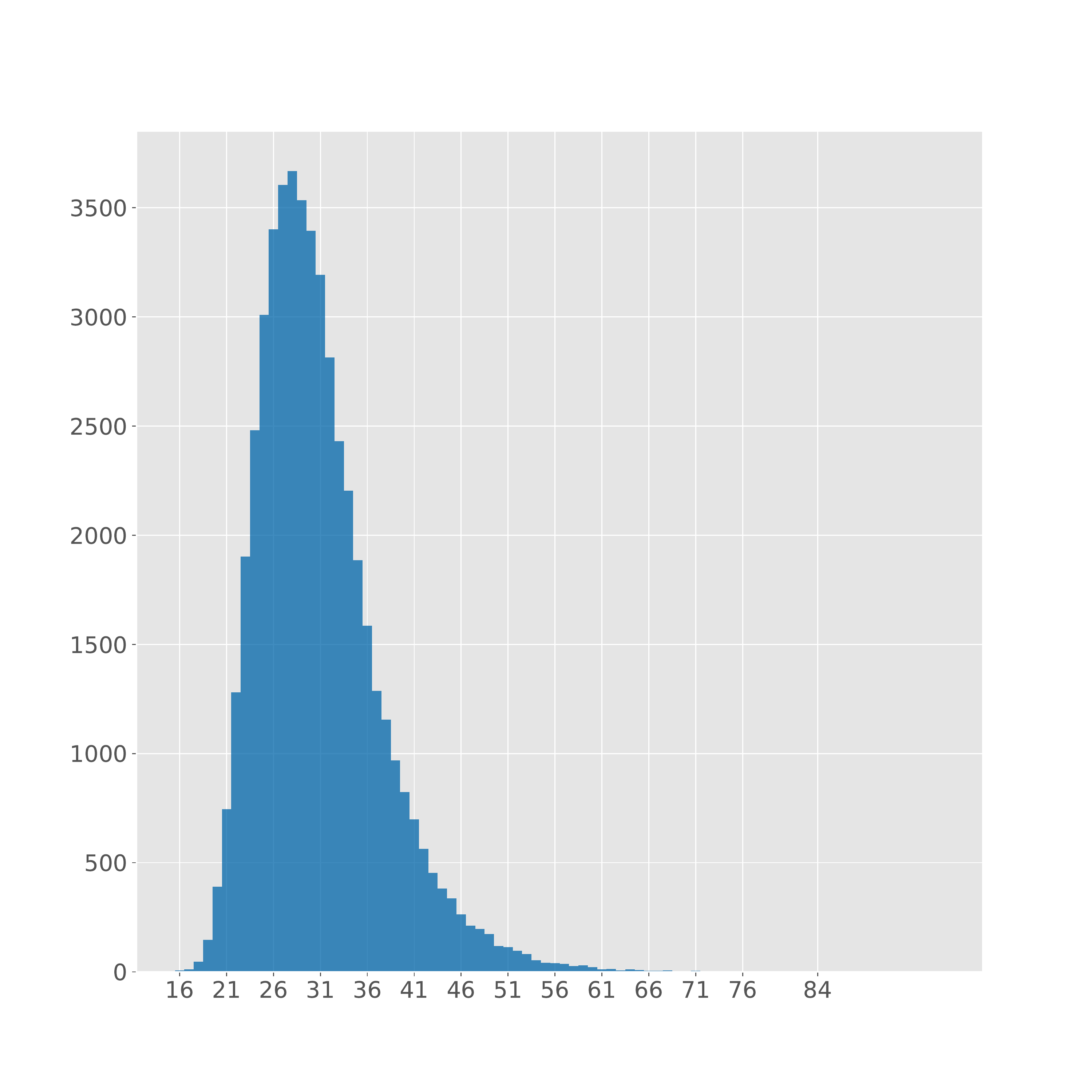}
  \caption{Plain \lterms.}
\end{subfigure}%
\begin{subfigure}{.5\textwidth}
  \centering
  \includegraphics[width=\linewidth]{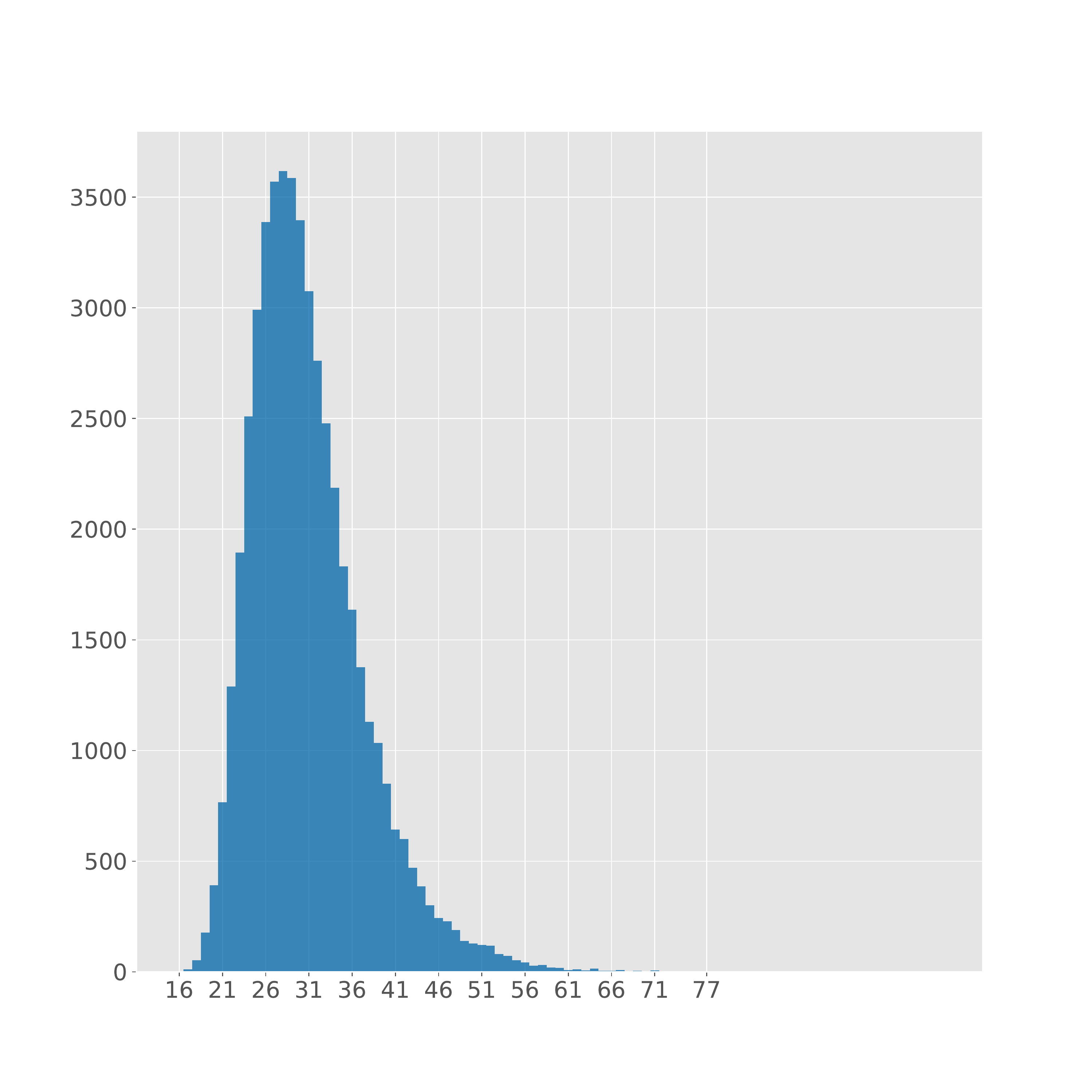}
  \caption{Closed \lterms.}
\end{subfigure}
\caption{Distribution histograms for the maximal number of variables bound to a
    single abstraction in plain and closed \lterms.}
\label{fig:experimental:plot:closed:maxvarabs}
\end{figure}

\begin{table}[!htb]
    \caption{The average, variance and standard deviation
    corresponding the maximal number of variables bound to a single abstraction
    in plain, closed and $h$\nobreakdash-shallow
    \lterms~within the obtained data sets.}
\label{fig:experimental:plot:closed:maxvarabs:mean}
\centering
\begin{tabular}{cccc}
    & Average & Variance & Std. dev. \\
  \hline
    plain & $30.83028$ & $43.08992$ & $6.56429$ \\
    closed & $30.8323$ & $43.12418$ & $6.5669$ \\
    $h$-shallow & $30.87194$ & $43.46242$ & $6.5926$ \\
\end{tabular}
\end{table}

\begin{remark}
Obtained distributions and moments are reminiscent of a double-exponential
    distribution related to extreme parameters in various combinatorial
    structures, see~\cite[p. 311]{flajolet09}. Notably, $\log_{1/\rho} n$ for $n
    \in [20,000;50,000]$ lies in the interval $[8.1259,8.8777]$ whereas as the
    same time $\sqrt{n} \geq 141.421$.  Consequently, the obtained average
    suggests a similar $C \cdot \log_{1/\rho} n$ type of behaviour occurring,
    for instance, in the analysis of longest runs in random words,
    see~\cite[Example V.4]{flajolet09}. Nonetheless, unlike the limit
    distribution of longest runs, the discovered distribution does not suggest a
    limit concentration due to the significant values of the observed variances.
\end{remark}

\subsubsection{$m$-openness of plain lambda terms}
The central notion of $m$\nobreakdash-openness plays an important rôle in the
analysis of closed \lterms. Given the complete analysis of $m$\nobreakdash-open
terms~\cite{BodiniGitGol17} and efficient techniques allowing to obtain
numerical estimates for the relative asymptotic density of $m$\nobreakdash-open
terms in the set of all plain \lterms~\cite{gittenberger_et_al:LIPIcs:2016:5741}
it is possible to approximate the limit distribution of the
$m$\nobreakdash-openness parameter in plain terms. However, in order to avoid the
laborious computations involved in this approach, we offer a simple Monte Carlo
approximation scheme.  \autoref{fig:experimental:plot:plain:mopenness} depicts
the distribution histogram of the obtained data set for plain \lterms. The
corresponding average, variance and standard deviation are listed
in~\autoref{fig:experimental:plot:plain:mopenness:mean}.

\begin{figure}[!htb]
  \centering
  \includegraphics[width=0.5\linewidth]{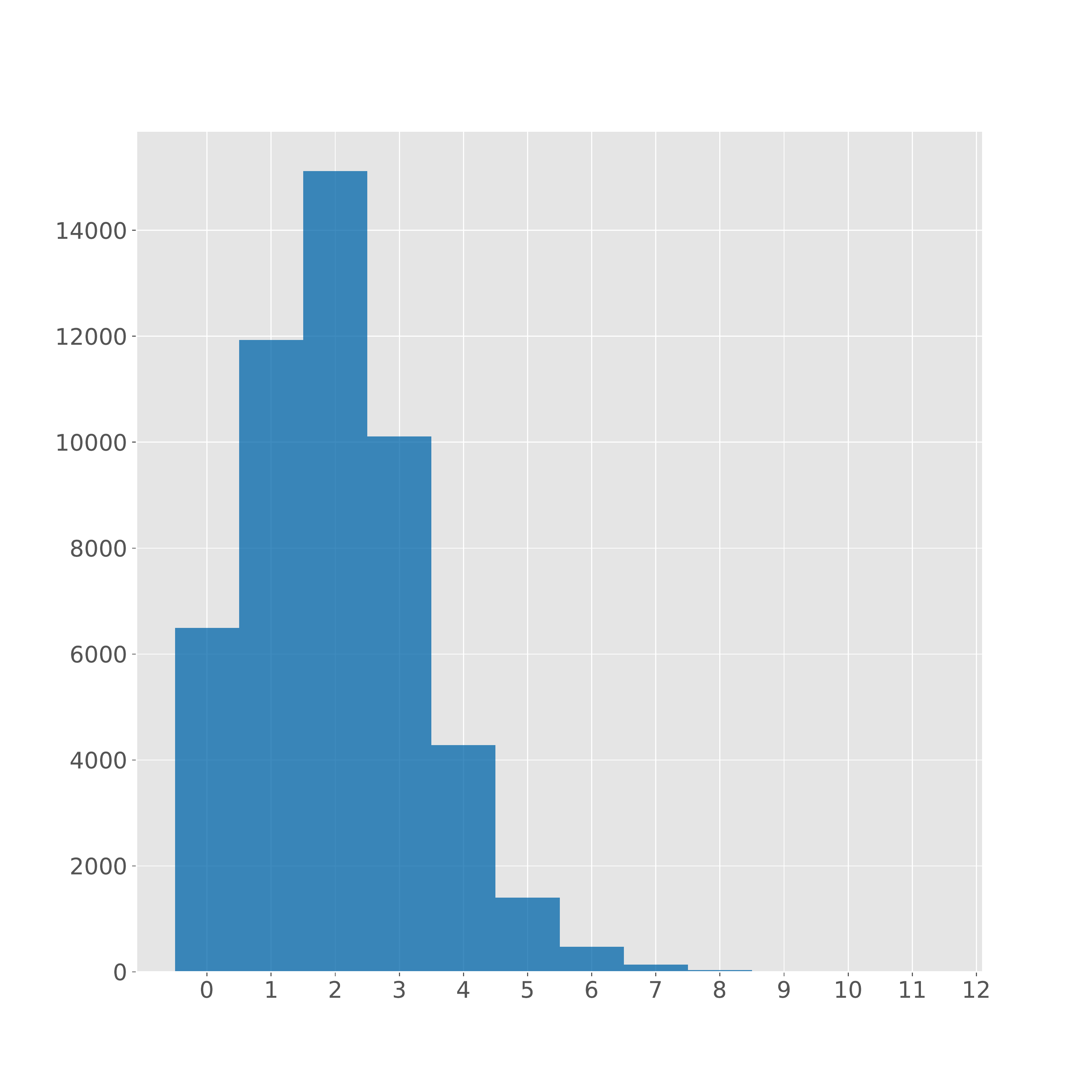}
\caption{Distribution histogram for $m$-openness of plain \lterms.}
\label{fig:experimental:plot:plain:mopenness}
\end{figure}

\begin{table}[!htb]
    \caption{The average, variance and standard deviation
    corresponding to the $m$\nobreakdash-openness of plain \lterms~within the obtained data set.}
\label{fig:experimental:plot:plain:mopenness:mean}
\centering
\begin{tabular}{cccc}
    & Average & Variance & Std. dev. \\
  \hline
    plain & $2.01856$ & $1.82538$ & $1.35106$ \\
\end{tabular}
\end{table}

\subsubsection{Generalised $m$-openness of plain and closed terms}
Though the notion of $m$\nobreakdash-openness is defined only for non-negative
values of $m$, we propose a natural generalisation to all integers in the
following manner.  We say that a closed \lterm~$T$ is
\emph{$m$\nobreakdash-closed}\footnote{Though this notion extends
$m$\nobreakdash-openness to negative values of $m$, we prefer to avoid the
term $m$\nobreakdash-open while referring to closed terms and instead use the
term $m$\nobreakdash-close. Hence, a term is $m$\nobreakdash-closed if it is,
intuitively, $(-m)$\nobreakdash-open.} for $m \geq 0$ if it is possible to
discard $m$ head abstractions of $T$ and still retain a closed \lterm. In other
words, a closed term is $m$\nobreakdash-closed if its top $m$ head abstractions
are non-binding.

Note that this new parameter provides a degree of term closeness; the higher the
factor $m$ of an $m$\nobreakdash-closed \lterm, the more closed it is. Like in
the case of $m$\nobreakdash-open \lterms, closed \lterms~are
$0$\nobreakdash-closed.  Moreover, if a term is $(m+1)$\nobreakdash-closed, then
it is also $m$\nobreakdash-closed.

\autoref{fig:experimental:plot:closed:mopenness} depicts the distribution
histogram of the obtained data sets for plain and closed \lterms, respectively.
The $x$-axis denotes the generalised $m$-openness factor  whereas the $y$-axis
corresponds to the number of terms attaining the corresponding cost value.
Numerical approximations of averages, variances and standard deviations
are listed in~\autoref{fig:experimental:plot:closed:mopenness:mean}.

\begin{figure}[!htb]
\centering
\begin{subfigure}{.5\textwidth}
  \centering
  \includegraphics[width=\linewidth]{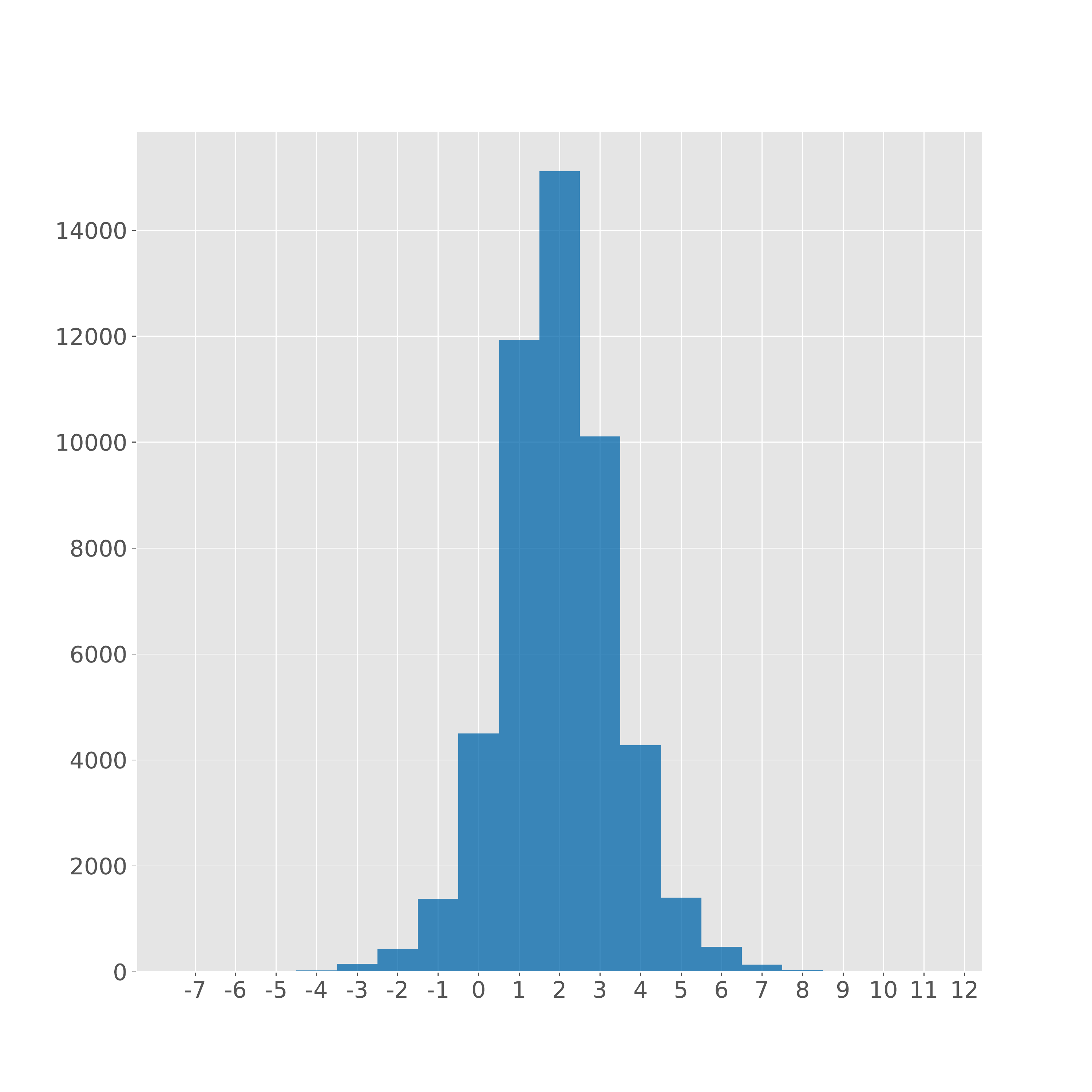}
  \caption{Plain \lterms.}
\end{subfigure}%
\begin{subfigure}{.5\textwidth}
  \centering
  \includegraphics[width=\linewidth]{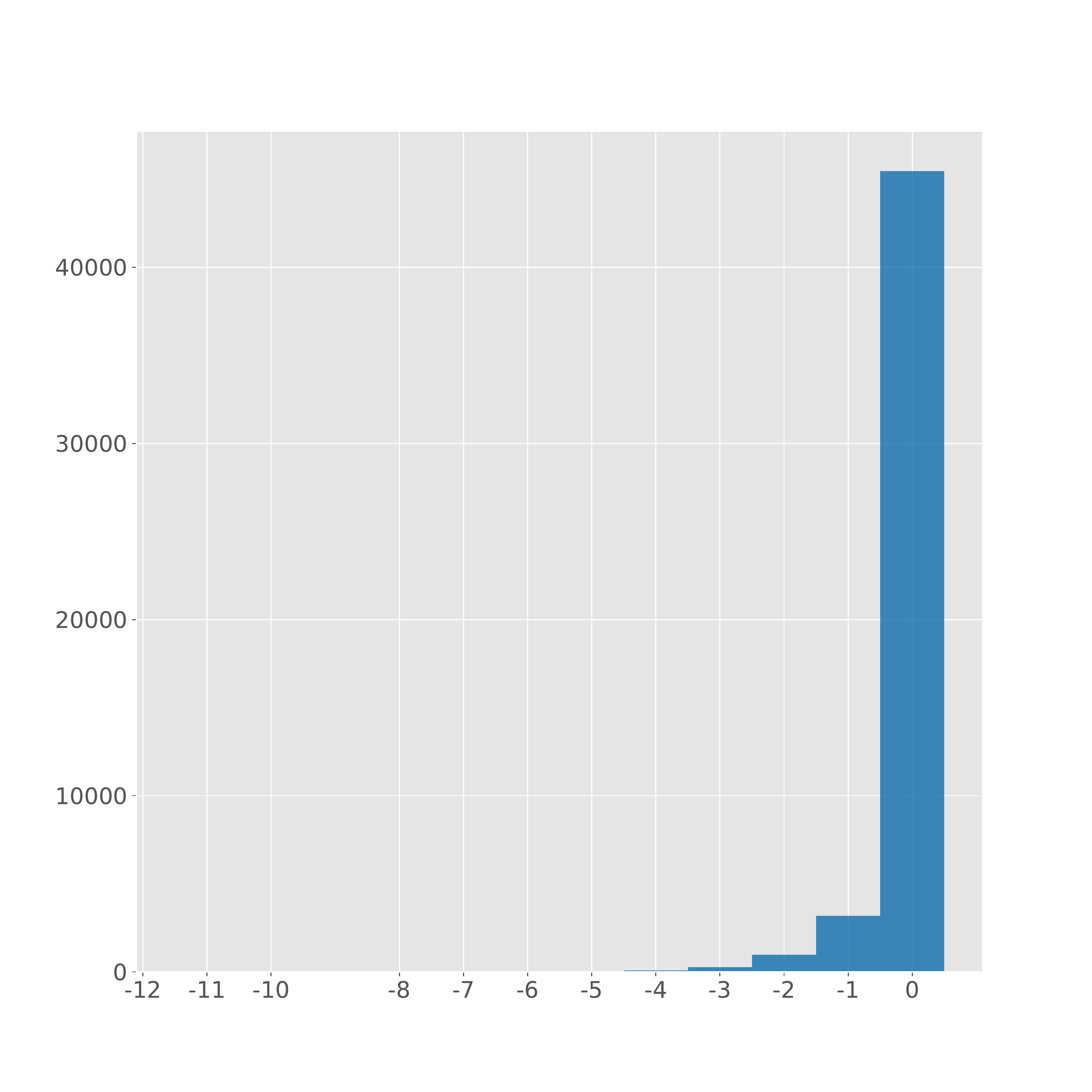}
  \caption{Closed \lterms.}
\end{subfigure}
\caption{Distribution histograms for the generalised $m$-openness in plain and closed \lterms.}
\label{fig:experimental:plot:closed:mopenness}
\end{figure}

\begin{table}[!htb]
    \caption{The average, variance and standard deviation
    corresponding the generalised $m$-openness factor
    in plain, closed and $h$\nobreakdash-shallow
    \lterms~within the obtained data sets.}
\label{fig:experimental:plot:closed:mopenness:mean}
\centering
\begin{tabular}{cccc}
    & Average & Variance & Std. dev. \\
  \hline
    plain & $1.96196$ & $2.15235$ & $1.46709$ \\
    closed & $-0.43814$ & $0.61569$ & $0.78466$ \\
    $h$-shallow & $-0.44054$ & $0.62834$ & $0.79268$ \\
\end{tabular}
\end{table}

\begin{remark}
The empirical distribution corresponding generalised $m$\nobreakdash-openness of
    closed \lterms~suggests that the vast majority of closed terms are
    $0$\nobreakdash-closed but not $1$\nobreakdash-closed. Given the related
    distribution of head abstractions,
    see~\autoref{fig:experimental:plot:closed:headabs}, such a result suggests
    that if a closed term has a leading head abstraction, it is more likely that
    it is in fact binding.
\end{remark}

Let us note that all of the presented empirical histograms exhibit a close
correspondence between parameters of closed \lterms~and their corresponding
equivalents in $h$\nobreakdash-shallow terms. Such a result should not be
surprising given the exponential convergence speed at which
$h$\nobreakdash-shallow \lterms~of size $N$ tend to closed \lterms~of size $N$
as $h \to \infty$, see~\autoref{sec:infinite:systems},
cf.~\cite{BodiniGitGol17}. Consequently, virtually all presented histograms for
closed \lterms~are also correct histograms for $h$\nobreakdash-shallow terms.

\subsection{Empirical evaluation of systems of generating functions}
The probability generating function \( p_n(u) \) of a random variable \( X_n \)
corresponding to a certain parameter inside a random \lterm~of size $n$, plain
or closed, can be obtained by evaluating the bivariate generating function \(
L(z, u) \) and taking the respective ratio of coefficients,
see~\autoref{subsec:limit:laws}:
\begin{equation}
    p_n(u) = \dfrac{[z^n] L(z, u)}{[z^n] L(z, 1)}.
\end{equation}
Remarkably, for statistics whose limit laws are discrete distributions, the
convergence is already evident at $n = 20$. We empirically evaluate the
coefficients of generating functions whenever the associated functional
equations are available.

For closed \lterms, the corresponding systems of functional equations are
infinite and hence cannot be directly evaluated. Fortunately, truncating the
infinite system at the \( k \)th level and replacing the \( k \)th function \( L_k(z,
u) \) by its infinite counterpart \( L_\infty(z, u) \) yields an exponentially
small difference in the dominant term coefficients,
see~\autoref{sec:infinite:systems}.  In what follows, we truncate the systems at
height \( k = 15 \).

\subsubsection{Head abstractions} We return to head abstractions in plain and
closed \lterms. The distribution of this parameter for large term sizes was
discussed in~\autoref{subsubsec:head:abs}. Here, we study this parameter for
small term sizes. The equations for corresponding generating functions are
presented in~\autoref{subsec:basic:head:abs,subsec:advanced:head:abs}.

\begin{figure}[!htb]
\centering
\begin{subfigure}{.5\textwidth}
  \centering
        \includegraphics[width=1.0\textwidth]{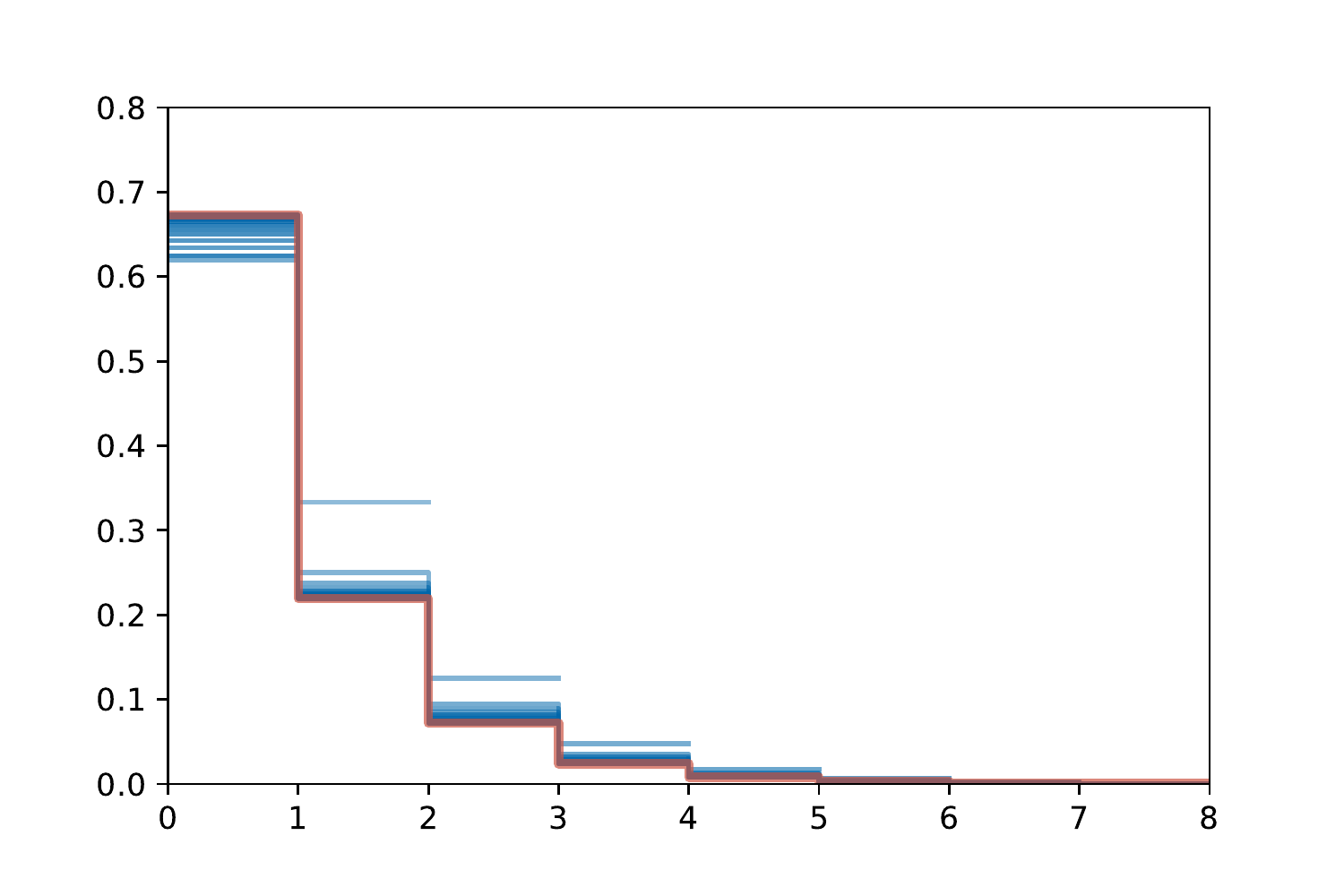}
  \caption{Plain \lterms.}
\end{subfigure}%
\begin{subfigure}{.5\textwidth}
  \centering
        \includegraphics[width=1.0\textwidth]{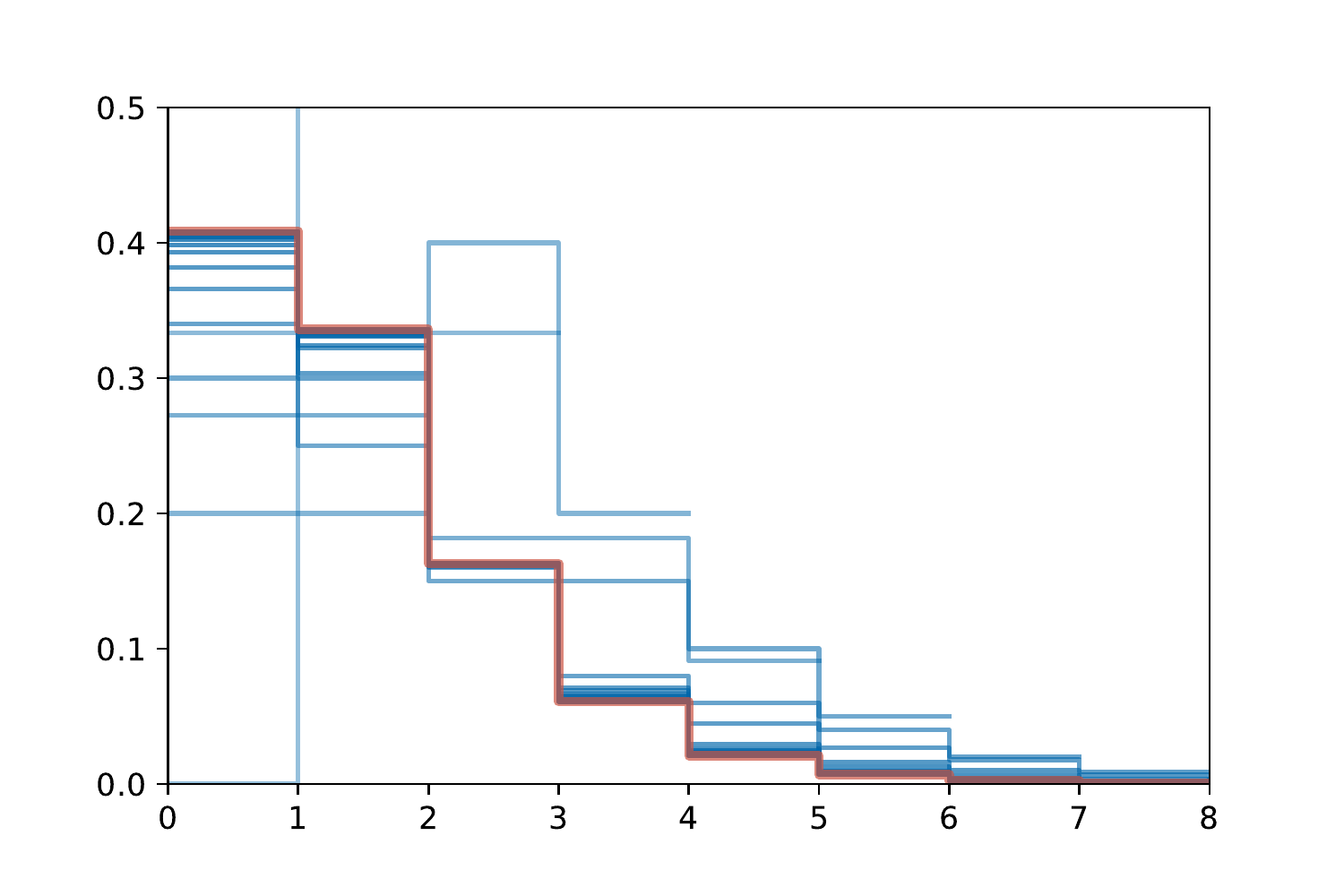}
  \caption{Closed \lterms.}
\end{subfigure}
\caption{Distribution histograms for the number of head abstractions in plain
    and closed \lterms.}
\label{fig:discrete:plot:headabs}
\end{figure}
The convergence process for $ n \in [1, 15] $ is depicted
in~\autoref{fig:discrete:plot:headabs}. Head abstractions in plain terms admit a
geometric distribution while the number of head abstractions in closed lambda
terms converges to a certain computable, discrete distribution.

\subsubsection{Leftmost-outermost redex search}
As already mentioned in~\autoref{subsec:empirical:boltzmann}, the distribution
associated with the cost of finding the leftmost-outermost
$`b$\nobreakdash-redex in plain and closed \lterms~tends to a discrete limiting
distribution. The functional systems for bivariate generating functions can be
found in~\autoref{subsec:basic:redex:discovery,subsec:advanced:redex:discovery}.

\begin{figure}[!htb]
\centering
\begin{subfigure}{.5\textwidth}
  \centering
        \includegraphics[width=1.0\textwidth]{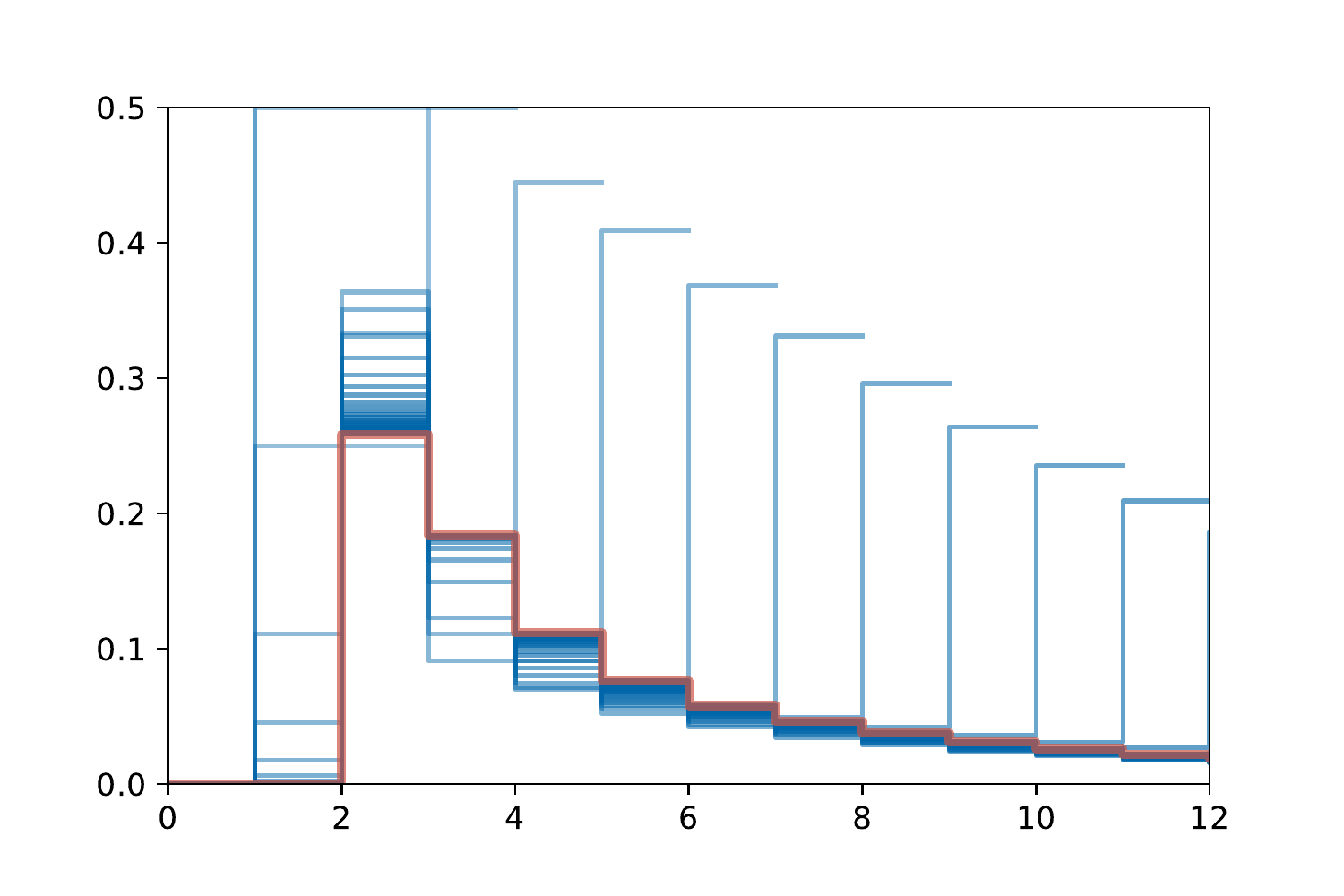}
  \caption{Plain \lterms.}
\end{subfigure}%
\begin{subfigure}{.5\textwidth}
  \centering
        \includegraphics[width=1.0\textwidth]{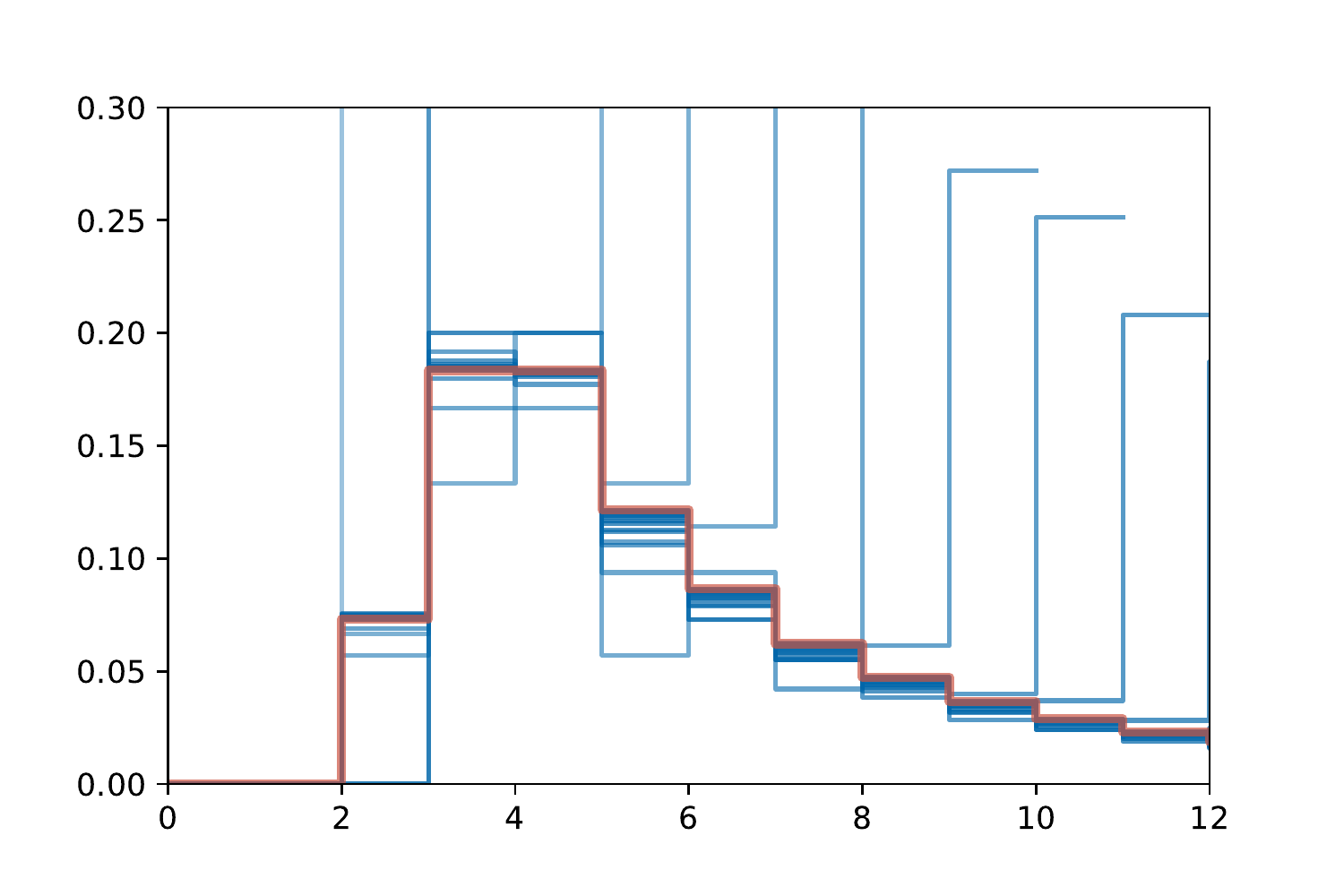}
  \caption{Closed \lterms.}
\end{subfigure}
\caption{Distribution histograms associated with the cost of redex search time in plain and
    closed \lterms.}
\label{fig:discrete:plot:redexsearch}
\end{figure}
We depict the convergence process in~\autoref{fig:discrete:plot:redexsearch}.
For both plain and closed \lterms, a remarkable detail can be observed. For
\lterms~of small size, the distribution has a peak at the value corresponding to
the term size. Note that such a phenomenon is related to $`b$\nobreakdash-normal
forms, witnessing the worst-case time complexity of the traversal algorithm.
For larger terms, the proportion of $`b$\nobreakdash-normal forms in all lambda
terms becomes exponentially negligible.

\subsubsection{Free variables and $m$-openness in plain terms}
Next, we consider two parameters of plain \lterms~that require advanced marking
techniques. The functional equations for the cases of free variables and
$m$\nobreakdash-openness are given in
subsequent~\autoref{subsec:m:openness,subsec:free:variables}.

\begin{figure}[!htb]
\centering
\begin{subfigure}{.5\textwidth}
  \centering
        \includegraphics[width=1.0\textwidth]{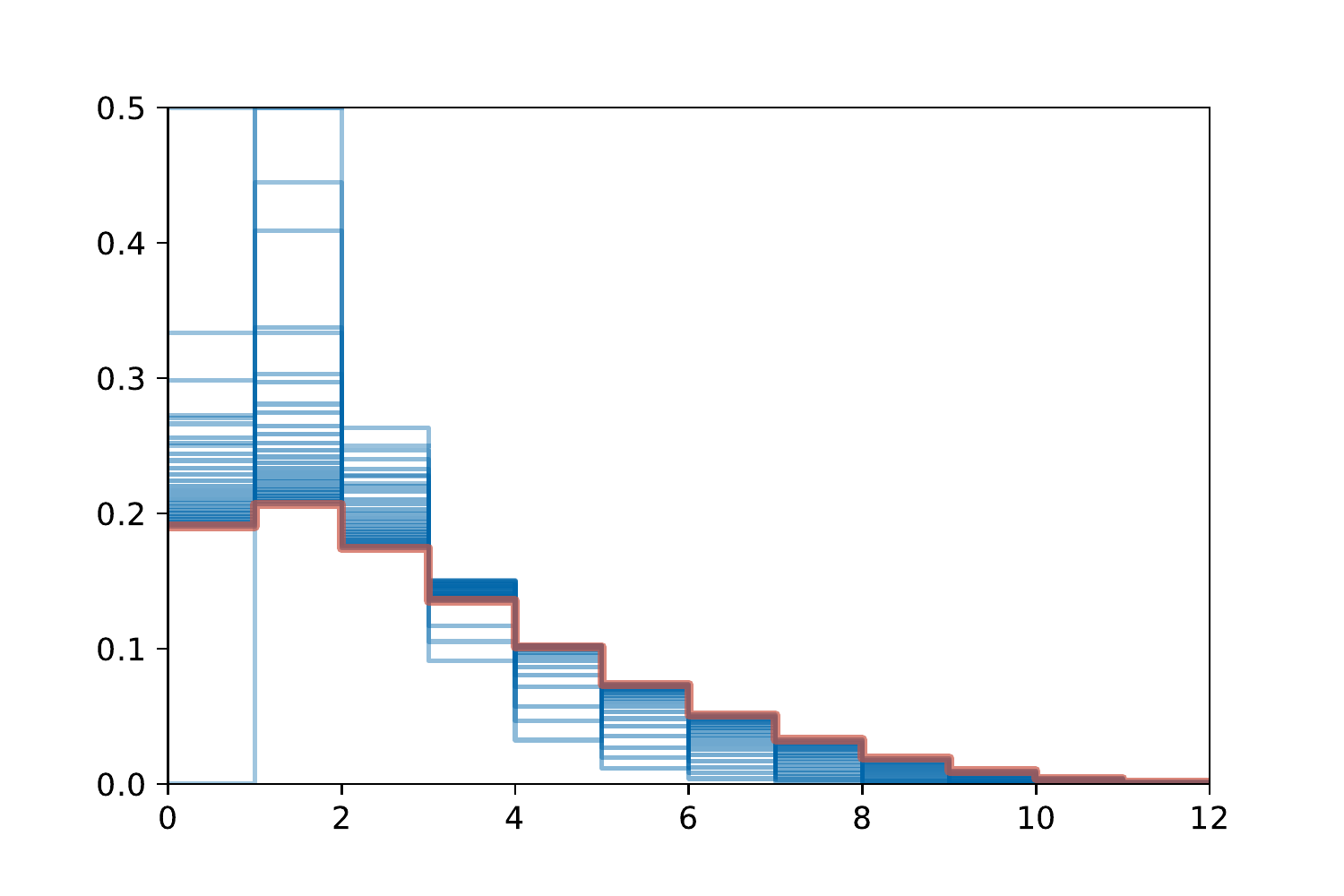}
  \caption{Free variables.}
\end{subfigure}%
\begin{subfigure}{.5\textwidth}
  \centering
        \includegraphics[width=1.0\textwidth]{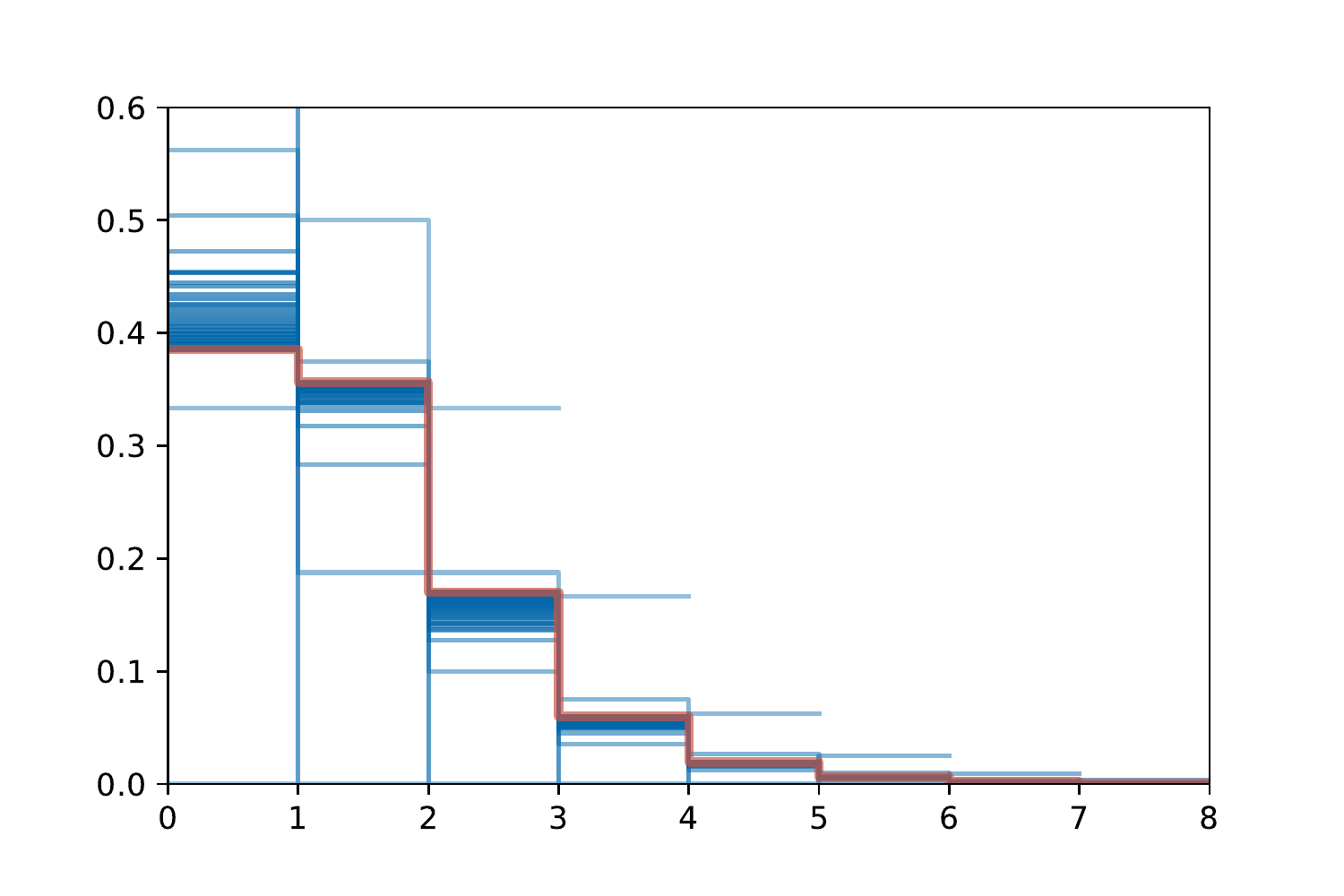}
  \caption{$m$-openness.}
\end{subfigure}
\caption{Distribution histograms for the number of free variables and
$m$-openness in plain \lterms.}
\label{fig:discrete:plot:advanced:plainterms}
\end{figure}
The convergence process is depicted
in~\autoref{fig:discrete:plot:advanced:plainterms}.  Let us remark that the
convergence process corresponding to the number of free variables proceeds more
slowly than all the other statistics that we consider in the current paper.

\subsubsection{De Bruijn index profile}
The final pair of plots concerns the profile of variables or, in other words,
the profile of de~Bruin indices in random plain and closed \lterms.
In~\autoref{subsec:debruijn:plain,subsec:debruijn:closed} we show that both
parameters tend to geometric limiting distributions.
\autoref{fig:discrete:dbindex:profile} depicts the probability distributions for
lambda terms of the size in the interval $n \in [1, 15]$.

\begin{figure}[!htb]
\centering
\begin{subfigure}{.5\textwidth}
  \centering
        \includegraphics[width=1.0\textwidth]{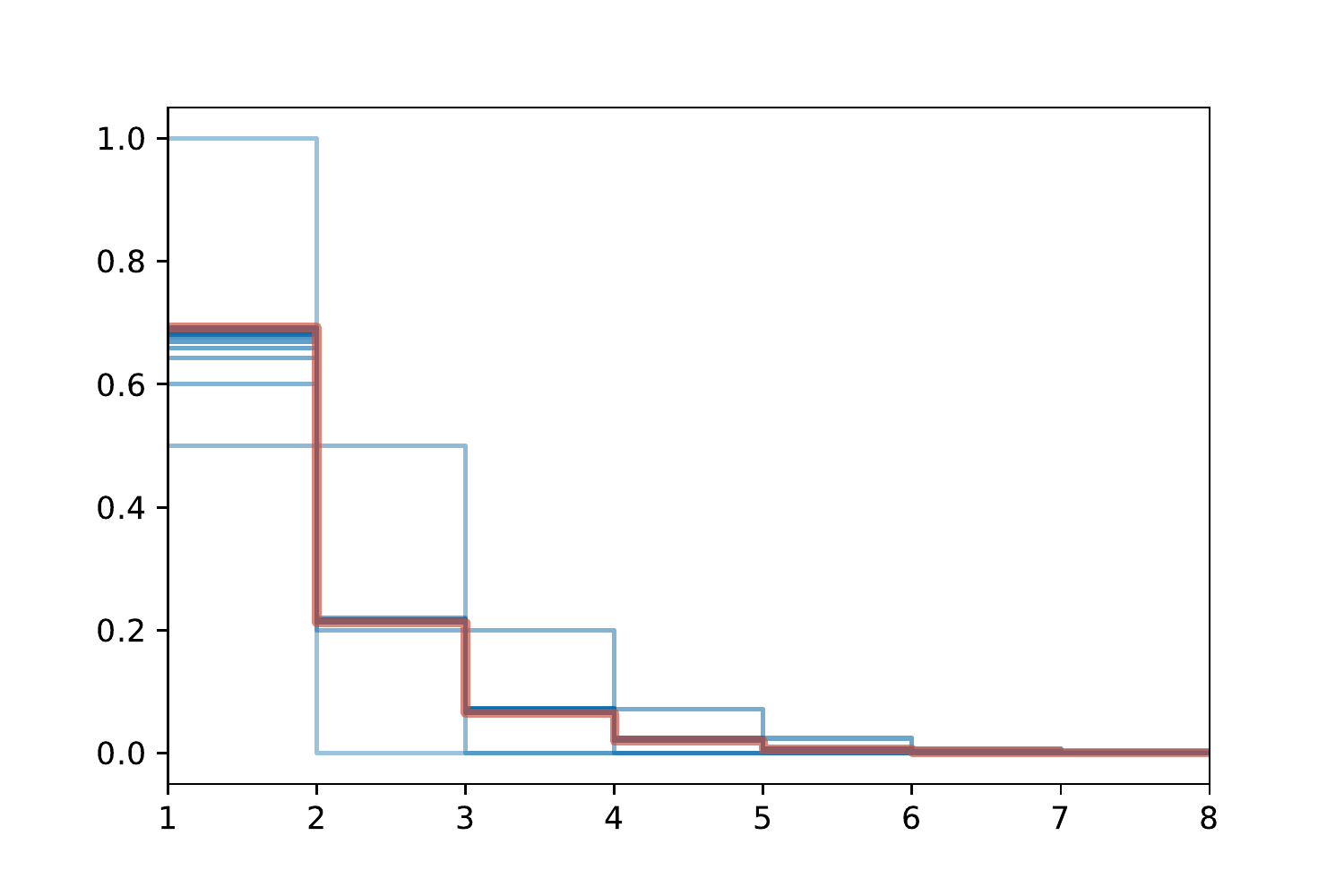}
  \caption{Plain \lterms.}
\end{subfigure}%
\begin{subfigure}{.5\textwidth}
  \centering
        \includegraphics[width=1.0\textwidth]{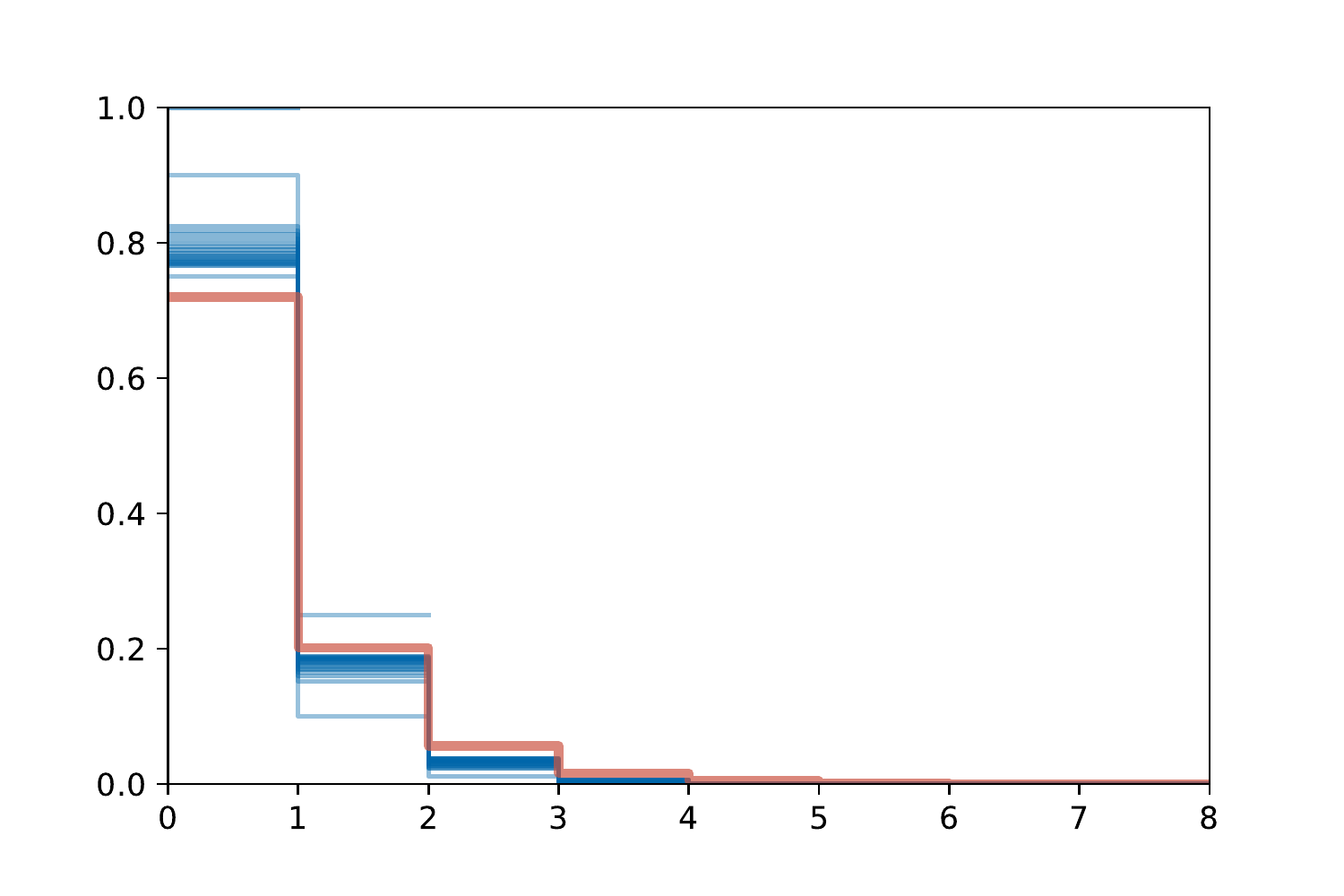}
  \caption{Closed \lterms.}
\end{subfigure}
\caption{Distribution histograms for de~Bruijn index profile in plain and closed
    \lterms.}
\label{fig:discrete:dbindex:profile}
\end{figure}

%% file: 5-infinite-systems.tex
\section{Infinite systems of algebraic equations}\label{sec:infinite:systems}
In this section we present our main technical contribution within analytic
combinatorics meant for dealing with certain recursive, infinite systems of
generating functions, i.e.~\autoref{proposition:infinite:system}.  Although our
results admits broader applications than the one presented in the current paper,
for consistency, we focus only on the enumeration and statistical analysis of
various combinatorial parameters in closed \lterms.  Our proof is motivated by
the papers~\cite{BodiniGitGol17,gittenberger_et_al:LIPIcs:2016:5741} where the
authors consider the enumeration problem of closed \lterms~without additional
marking parameters.  For that purpose, they construct a series of sequences
$\seq{L_{m,N}}_{m \geq 0}$ that approximate $m$\nobreakdash-open \lterms~with
convergence rate of order $O\left(\frac{1}{\sqrt{N}}\right)$.  In what follows
we improve this rate to an exponential one.  Moreover, we abstract the
considered system away from \lterms~in the de~Bruijn notation, allowing for an
analysis of more general varieties of combinatorial systems.

Remarkably, in~\cite[Section 5]{gittenberger_et_al:LIPIcs:2016:5741} the authors
obtain the asymptotic estimate \( \Theta(n^{-3/2} \rho^{-n}) \) for the number
of closed lambda terms of size \( n \) with the difference between upper and
lower bounds on the constant multiple within \( 10^{-7} \). Let us notice that
this technique is quite different from the technique used
in~\cite{BodiniGitGol17}. Our approach is more similar to the former method,
however admits certain improvements. Specifically, we simplify the procedure of
improved constant estimation, give a rigorous proof about the exponential
convergence rate of the constant multiple, and provide more general and simpler
tools based on the properties of the Jacobian of the limiting system (instead of
exploiting the particular form of this equation in the case of \lterms).

\subsection{Calculus techniques for formal power series}
We start with some basic notation and properties of coefficient-wise
inequalities on formal power series and some geometric results on matrices of
formal power series. In what follows, we denote the spectral
radius of matrix \( \vec{\mathcal A} \), i.e.~the largest absolute value of its
eigenvalues, by \( r(\vec{\mathcal A}) \).

\begin{definition}[Formal power series domination]
    \label{definition:fps:domination}
We say that \( f(z) \) \emph{is dominated by} \( g(z) \), denoted as \( f(z)
    \preceq g(z) \), if for each \( n \geq 0 \) we have \( {[z^n] f(z) \leq
    [z^n] g(z)} \).  For multivariate formal power series
    \begin{equation}
        f(\vec z) = \sum_{\vec n \geq \vec 0} a_{\vec n} \vec z^{\vec n}
        \quad \text{and} \quad
        g(\vec z) = \sum_{\vec n \geq \vec 0} b_{\vec n} \vec z^{\vec n}
    \end{equation}
    where \( \vec z = (z_1, \ldots, z_d) \), \( \vec n = (n_1, \ldots, n_d) \)
        and \( \vec z^{\vec n} = z_1^{n_1} z_2^{n_2} \ldots z_d^{n_d} \) the
        domination \( {f(\vec z) \preceq g(\vec z)} \) means that for each
        vector of indices \( \vec n \) it holds \( a_{\vec n} \leq b_{\vec n}
        \). Certainly, if a combinatorial class \( \mathcal F \) is included in
        a combinatorial class \( \mathcal G \), then the generating functions \(
        f(z) \) and \( g(z) \) corresponding to respective classes satisfy \(
        f(z) \preceq g(z) \).  The same holds for marked classes and
        associated multivariate generating functions. Finally, for vectors \(
        \vec A \) and \( \vec B \) of identical (however not necessarily finite)
        dimension, we write \( \vec A \preceq \vec B \) to denote a
        coordinate-wise domination of respective components.
\end{definition}

In real analysis, the \emph{squeeze lemma} is a theorem regarding the limit of
the sequence which is upper- and lower-bounded by two sequences with the same
limit value. The following statement is a variant of this lemma, stated in the
context of formal power series admitting coefficient asymptotics suitable for
analysis of corresponding limit laws, see~\autoref{subsec:limit:laws}.

\begin{lemma}[Squeeze lemma for formal power series]
    \label{lemma:squeeze}
    Let \( z \in \mathbb C \) and
    \( {\vec u = (u_1, \ldots, u_r)} \in \mathbb C^r \).
    Assume that
    \( f(z, \vec u) \),
    \( \seq{h_m(z, \vec u)}_{m \geq 0} \), and
    \( \seq{g_m(z, \vec u)}_{m \geq 0} \)
    are multivariate formal power series in \( (z, \vec u) \)
    with non-negative coefficients,
    such that for
    every \( n \) and \( m \), the functions
    \begin{equation}
        [z^n] f(z, \vec u), \
        [z^n] h_m(z, \vec u) \ \text{and} \
        [z^n] g_m(z, \vec u)
    \end{equation}
    are polynomials in \( \vec u \).

    Moreover, assume that the following conditions hold:
    \begin{itemize}
        \item for each \( m \geq 0 \), we have
            \begin{equation}
                \label{ineq:domination}
                h_m(z, \vec u) \preceq f(z, \vec u) \preceq g_m(z, \vec u)
            \end{equation}
            in the sense of multivariate
            formal power series domination;
        \item there exists a sequence of real positive numbers \( C_n \)
            and functions
            \( \seq{\underline{A_m}(\vec u)}_{m\geq 0} \),
            \( \seq{\overline{A_m}(\vec u)}_{m\geq 0} \),
            \( B(\vec u) \) analytic in a common neighbourhood of
            \( \vec u = \vec 1 \),
            such that
            uniformly for \( m \geq 0 \) and
            uniformly in a fixed complex vicinity of
            \( \vec u = \vec 1 \), it holds
            \begin{equation}
            \label{eq:limit:upperandlower:bounds}
                \lim_{n \to \infty}
                \dfrac{
                    [z^n] h_m(z, \vec u)
                }{
                    C_n \underline{A_m}(\vec u) B(\vec u)^n
                }
                = 1
                \quad \text{and} \quad
                \lim_{n \to \infty}
                \dfrac{
                    [z^n] g_m(z, \vec u)
                }{
                    C_n \overline{A_m}(\vec u) B(\vec u)^n
                }
                = 1
                ;
            \end{equation}
        \item there exists a function \( A(\vec u) \) analytic
            near \( \vec u = \vec 1 \) such that
            \( A(\vec 1) \neq 0 \) satisfying
            uniformly in a
            complex vicinity of \( \vec u = \vec 1 \)
            \begin{equation}
                \label{eq:limit:m}
                \lim_{m \to \infty} \underline{A_m}(\vec u) =
                \lim_{m \to \infty} \overline{A_m}(\vec u) =
                A(\vec u).
            \end{equation}
    \end{itemize}
    Then, uniformly in a complex vicinity of \( \vec u = \vec 1 \), as
    \( n \to \infty \):
\begin{equation}
    [z^n] f(z, \vec u) \sim C_n A(\vec u) B(\vec u)^n.
\end{equation}
\end{lemma}
\begin{proof}
We divide the proof into two parts. We start with showing that the statement holds for
vectors \( \vec u \) whose components are real positive numbers. Next,
    we extend this property onto all complex components of \( \vec u \).

First, take \( \vec u \in \mathbb R^r \) in the vicinity where the functions
    \( \underline{A_m}(\vec u) \), \( \overline{A_m}(\vec u) \), \( B(\vec u) \)
    are analytic.  Then, following~\eqref{ineq:domination}
    and~\eqref{eq:limit:upperandlower:bounds}, for every positive \( \varepsilon
    \) there exists \( N := N(\varepsilon) \), independent of \( m \) and \( \vec
    u \), such that
\begin{equation}
    \forall n > N \quad
    \underline{A_m}(\vec u)(1 - \varepsilon) \leq
    \dfrac{[z^n] f(z, \vec u)}{C_n B(\vec u)^n}
    \leq
    \overline{A_m}(\vec u)(1 + \varepsilon).
\end{equation}
Taking the limit with respect to \( m \) we note that condition~\eqref{eq:limit:m}
    guarantees that for arbitrary small \( {\varepsilon > 0} \) and sufficiently large
    \( N \) (again, independent of \( \vec u \)) we have
\begin{equation}
    \forall n > N \quad
    A(\vec u)(1 - \varepsilon) \leq
    \dfrac{[z^n] f(z, \vec u)}{C_n B(\vec u)^n}
    \leq
    A(\vec u)(1 + \varepsilon).
\end{equation}
In other words, for values of \( \vec u \in \mathbb R^r \) within a fixed
    vicinity of \( \vec u = \vec 1 \) it holds
\begin{equation}
    \label{eq:newly:established}
    [z^n] f(z, \vec u) \sim C_n A(\vec u) B(\vec u)^n.
\end{equation}

Now, let us consider \( \vec u \in \mathbb C^r \).
Note that since for each \( n, m \geq 0 \) the formal power series
\( [z^n] f(z, \vec u) \), \( [z^n] h_m(z, \vec u) \) and \( [z^n] g_m(z, \vec u) \)
are polynomials in \( \vec u \), they are analytic in \( \mathbb C^r \).
Moreover, as
\(
\psi_{n,m}(\vec u) := [z^n] f(z, \vec u) - [z^n] h_m (z, \vec u)
\) is a polynomial with non-negative coefficients,
for every \( \vec u \in \mathbb C^r \) we have
\(
    |\psi_{n,m}(\vec u)| \leq \psi_{n,m}(|\vec u|)
\)
and, consequently,
\begin{equation}
    \big|
    [z^n] f(z, \vec u) - [z^n] h_m (z, \vec u)
    \big| \leq
    [z^n] f(z, |\vec u|) - [z^n] h_m (z, |\vec u|).
\end{equation}
After dividing both parts by \( C_n |A(\vec u)| |B(\vec u)^n| \) we obtain
\begin{equation}
    \left|
    \dfrac{[z^n] f(z, \vec u)}
    {C_n A(\vec u) B(\vec u)^n}
    -
    \dfrac{[z^n] h_m (z, \vec u)}
    {C_n A(\vec u) B(\vec u)^n}
    \right|
    \leq
    \dfrac{[z^n] f(z, |\vec u|)}
    {C_n |A(\vec u)| |B(\vec u)^n|}
    -
    \dfrac{[z^n] h_m (z, |\vec u|)}
    {C_n |A(\vec u)| |B(\vec u)^n|}.
\end{equation}
Following condition~\eqref{eq:limit:upperandlower:bounds} and the
estimate~\eqref{eq:newly:established} for \( \vec u \in
\mathbb{R}^r\), we note that for every \( \varepsilon > 0 \) there exists \( N
:= N(\varepsilon) \) independent of \( m \) and \( \vec u \), such that for all
\( n > N \) we further have
\begin{equation}
    \left|
        \dfrac{[z^n] f(z, \vec u)}{C_n A(\vec u) B(\vec u)^n}
        -
        \dfrac{\underline{A_m}(\vec u)}{A(\vec u)}
    \right|
    \leq \dfrac{A(|\vec u|) - \underline{A_m}(|\vec u|)}
    {| A(\vec u) |}
    \cdot \dfrac{B(|\vec u|)^n}{|B(\vec u)^n|}
    + \varepsilon.
\end{equation}
Since \( \varepsilon \) does not depend on \( m \),
we can take the limit with respect to \( m \).
Given condition~\eqref{eq:limit:m} we note that
for sufficiently large \( n \) we have
\begin{equation}
    \left|
        \dfrac{[z^n] f(z, \vec u)}{C_n A(\vec u) B(\vec u)^n}
        -
        1
    \right|
    \leq
    \varepsilon.
\end{equation}
Hence, uniformly in a fixed complex vicinity of
\( \vec u = \vec 1 \)
\begin{equation}
    \lim_{n \to \infty}
    \dfrac{[z^n] f(z, \vec u)}{C_n A(\vec u) B(\vec u)^n}
    = 1,
\end{equation}
which finishes the proof.
\end{proof}

\begin{remark}\label{remark:squeeze}
    Using the same technique, higher-order error terms can also be transferred,
    provided that the function is squeezed between two sequences of formal power
    series with known Puiseux expansions. In such a situation, higher-order
    terms correspond to coefficients obtained from the summands of Puiseux
    expansion
\begin{equation}
    f(z, \vec u) \sim
    \sum_{k \geq 0} c_k(\vec u) \left(
        1 - \dfrac{z}{\rho(\vec u)}
    \right)^{k/2}.
\end{equation}
\end{remark}

The next lemma is a formal power series analogue of Lagrange's mean value
theorem.
\begin{lemma}[Mean value lemma for formal power series]\label{lemma:mean:value}
Let $f(z)$ and $g(z)$ be two formal power series such that \( f(z) \preceq g(z)
    \). Assume that \( \Psi( t ) = \sum_{n \geq 0} \psi_n t^n \) is a formal
    power series with non-negative coefficients. Then,
    \begin{equation}\label{eq:infinite:system:mean:value:lemma}
    \Psi(g(z)) - \Psi(f(z)) \preceq \left(g(z)-f(z)\right) \Psi'(g(z)).
\end{equation}
Likewise, the statement holds for vectors of formal power series in a
    coordinate-wise manner.
\end{lemma}
\begin{proof}
    Coefficient-wise subtraction of the left-hand side
    of~\eqref{eq:infinite:system:mean:value:lemma} yields
\begin{align}
    \begin{split}
        \Psi(g(z)) - \Psi(f(z)) &= \sum_{n \geq 0} \psi_n \left( {g(z)}^n - {f(z)}^n
        \right)\\
        &= \left( g(z) - f(z) \right) \sum_{n \geq 0} \psi_n \left(\sum_{i=0}^{n-1}
        {g(z)}^{i} {f(z)}^{n-i-1}\right)\\
        &\preceq \left( g(z) - f(z) \right) \sum_{n \geq 0} \psi_n n
        {g(z)}^{n-1}\\
        &= \left( g(z) - f(z) \right) \Psi'(g(z)).
    \end{split}
\end{align}

    The first two equalities hold as a consequence of formal power series
    composition and the identity $a^n - b^n = \left(a-b\right) \sum_{i=0}^{n-1}
    a^i b^{n-i-1}$. The subsequent domination follows from the assumption that
    $f(z) \preceq g(z)$.
\end{proof}
\newcommand{\uidx}[1]{^{\langle #1 \rangle}}

\subsection{Forward recursive systems}

The following definition of \emph{forward recursive systems} encapsulates the
general, abstract features of the infinite systems that we consider in the
current paper. Core characteristics of the infinite system corresponding to
closed \lterms~are abstracted and divided into three general conditions which
are sufficient to access the asymptotic form of respective coefficients.

\begin{definition}[Forward recursive systems]\label{definition:infinitely:nested:systems}
    Let $z$ be a formal variable and $\vec{u} = (u_1, \ldots, u_r)$ be a vector of \( r \) formal
    variables. Consider
    infinite sequences $\seq{\vec{L} \uidx m}_{m \geq 0}$ and
    $\seq{\vec{\mathcal K} \uidx m}_{m \geq 0}$
    of $d$\nobreakdash-dimensional vectors
    \begin{equation}
        \vec{L} \uidx m = \left(L_1 \uidx m, \ldots, L_d \uidx
        m \right) \quad \text{and} \quad
        \vec{\mathcal K} \uidx m = \left(K_1 \uidx m, \ldots, K_d \uidx
        m \right)
    \end{equation}
    consisting of formal power series $L_i \uidx m \left(z,\vec{u}\right)$
    and $K_i \uidx m \left(\vec \ell_1, \vec \ell_2, z,\vec{u}\right)$ where
    \( \vec \ell_1 \) and \( \vec \ell_2 \) are vectors of \( d \) variables
    and $i = 1,\ldots,d$.

    Assume that $\seq{\vec{L} \uidx m}_{m \geq 0}$ and
    $\seq{\vec{\mathcal K} \uidx m}_{m \geq 0}$ satisfy
    \begin{equation}\label{eq:infinite:system:forward:recursive:system}
        \vec{L} \uidx m = \vec{\mathcal K} \uidx m \left(\vec{L} \uidx m, \vec{L} \uidx
        {m+1}, z, \vec{u} \right).
    \end{equation}
    Then, we say that the
    system~\eqref{eq:infinite:system:forward:recursive:system} is \emph{forward
    recursive}.

    Furthermore, consider a \emph{limiting system} in form of
    \begin{equation}\label{eq:infinite:system:limit:system}
        \vec{L} \uidx \infty = \vec{\mathcal K} \uidx \infty \left(\vec{L} \uidx
        \infty, \vec{L} \uidx {\infty}, z, \vec{u} \right)
    \end{equation}
    where $\vec{L} \uidx \infty$ and $\vec{\mathcal K} \uidx \infty$ are
    $d$\nobreakdash-dimensional vectors of formal power series $L_i \uidx
    \infty(z,\vec u)$ and $K_i \uidx \infty(\vec \ell_1, \vec \ell_2, z, \vec u)$,
    respectively, and moreover all
    series $K_i \uidx \infty$ are analytic at
    $\left(\vec \ell_1, \vec \ell_2, z,\vec{u}\right) =
    \left(\vec{0},\vec{0},0,\vec{1}\right)$. In this setting, we say that the
    system~\eqref{eq:infinite:system:forward:recursive:system}:
    \begin{enumerate}
        \item is \emph{infinitely nested} if
            $\vec{\mathcal K} \uidx m (\vec \ell_1, \vec \ell_2, z, \vec u)
        \preceq
        \vec{\mathcal K} \uidx \infty (\vec \ell_1, \vec \ell_2, z, \vec u)$
        for each $m \geq 0$;
        \item \emph{tends to an irreducible context-free schema} if it is
            infinitely nested and its corresponding limiting
            system~\eqref{eq:infinite:system:limit:system} satisfies the
            premises of the Drmota--Lalley--Woods theorem
            (see~\autoref{proposition:irreducible:polynomial:systems}) i.e.~is a
            polynomial, non-linear system of functional equations which is
            algebraic positive, proper, irreducible and aperiodic;
        \item
            is \emph{exponentially converging} if
            there exists a vector \( {\vec
            A(z, \vec u) = \left(A_1(z, \vec{u}),\ldots,A_d(z,\vec{u})\right)} \)
            and a function \( B(z, \vec u) \) such that:
\begin{itemize}
    \item for each \( m \geq 0 \) we have
        \begin{equation}
            \vec{\mathcal K} \uidx \infty
            (\vec L \uidx \infty, \vec L \uidx \infty, z, \vec u)
            -
            \vec{\mathcal K} \uidx m
            (\vec L \uidx \infty, \vec L \uidx \infty, z, \vec u)
            \preceq
            \vec A(z, \vec u) \cdot {B(z, \vec u)}^m;
        \end{equation}
    \item \( A_1(z, \vec u),\ldots,A_d(z, \vec u) \) and \( B(z, \vec u) \) are
        analytic functions in the disk
        \( |z| < \rho + \varepsilon \) for some \( \varepsilon > 0 \) and \(
        \vec u = \vec 1 \) where \( \rho \) is the dominant singularity of the
        limit system~\eqref{eq:infinite:system:limit:system} at point \( \vec u
        = \vec 1 \). Moreover,
    at \( \vec u = \vec 1 \) we have \( \left|B(\rho(\vec
        u), \vec u)\right| < 1 \) where \( \rho(\vec u) \) is the singularity of the
        limiting system~\eqref{eq:infinite:system:limit:system}.
\end{itemize}
    \end{enumerate}
\end{definition}

\begin{example}
Consider the infinite system corresponding to $m$\nobreakdash-open \lterms,
    see~\eqref{eq:m-open:terms:grammar}.  Recall that the sequence
    $\seq{L_m(z)}_{m \geq 0}$ of respective generating functions satisfies
\begin{align}
    \label{eq:example:forward:recursive}
    \begin{split}
        L_0(z) &= z L_1(z) + z {L_0(z)}^2\\
        L_1(z) &= z L_2(z) + z {L_1(z)}^2 + z\\
               & \cdots\\[-.3cm]
        L_m(z) &= z L_{m + 1}(z) + z {L_m(z)}^2
    + z \dfrac{1 - z^m}{1 - z}
    \\[-.15cm]
    & \cdots
    \end{split}
\end{align}
    Let us show that~\eqref{eq:example:forward:recursive} is an infinitely nested,
    forward
    recursive system which tends to an irreducible context-free schema of
    $L_\infty(z)$ at an exponential convergence rate.
    Here, each intermediate
    system $\vec{L} \uidx m$ consists of a single equation defining $L_m(z)$.
    Note that there are no additional marking variables $\vec{u}$.
    The vectors \( \vec{\mathcal K} \uidx m \) are
    one-dimensional and the corresponding functions \( K_m \) are given by
    \begin{equation}
        K_m(\ell_1, \ell_2, z) :=
        z \ell_2 + z \ell_1^2 + z \dfrac{1 - z^m}{1 - z}.
    \end{equation}
    The limiting system \( L_\infty(z) \) satisfies
    \begin{equation}
        L_\infty(z) = z L_\infty(z) + z L_\infty(z)^2 + \dfrac{z}{1 - z}.
    \end{equation}
    One can easily check that it also satisfies the premises
    of~\autoref{proposition:irreducible:polynomial:systems}; hence, the
    considered system~\eqref{eq:example:forward:recursive} tends to an irreducible
    context-free schema.
    Since the trivariate formal power series \( K_\infty(\ell_1, \ell_2, z)
    - K_m(\ell_1, \ell_2, z) \) has non-negative coefficients, the
    system~\eqref{eq:example:forward:recursive} is also infinitely nested.
    Moreover, the
    difference between the limiting equation and the $m$th
    equation computed at \( \ell_1 = \ell_2 = L_\infty(z) \)
    is equal to $\frac{z^{m+1}}{1-z}$ and corresponds to a subset of
    de~Bruijn indices. Certainly, as $m$ tends to infinity, this difference
    converges to zero exponentially fast.
\end{example}

Given the combinatorial relation between $m$\nobreakdash-open \lterms~and plain
terms, we readily obtain the requested condition \( L_m(z) \preceq L_\infty(z)
\).  However, for arbitrary forward recursive systems~\eqref{eq:infinite:system}
establishing that \( \vec L \uidx m \preceq \vec L \uidx \infty \) is no longer
so straightforward. In what follows, we prove that for this inequality to hold
it is sufficient that the limiting system is well-founded.

\begin{lemma}\label{lemma:inclusion}
Let \( \mathcal S \) be an infinitely nested, forward recursive
system~\eqref{eq:infinite:system:forward:recursive:system}. Assume that the
coefficients of the formal power series
\( \vec{\mathcal K} \uidx \infty(\vec \ell_1, \vec \ell_2, z, \vec u) \)
corresponding to the limiting system
are non-negative and the limiting system~\eqref{eq:infinite:system:limit:system}
is well-founded (i.e.~algebraic proper in the sense
of~\autoref{definition:nested:systems}) and has a non-zero solution
\( \vec L \uidx \infty(z, \vec u) \).
Finally, assume that
\( \vec{\mathcal K} \uidx \infty(\vec 0, \vec 0, 0, \vec u) = \vec 0\).
Then,
\begin{equation}
    \vec L \uidx m (z, \vec u) \preceq \vec L \uidx \infty(z, \vec u ).
\end{equation}
\end{lemma}

\begin{proof}
    Consider the vectors
    \( \vec{\mathcal L} = (\vec L \uidx 0, \vec L \uidx 1, \ldots) \)
    and
    \( \vec{\mathcal L}^+
    =
    (\vec L \uidx \infty, \vec L \uidx \infty, \ldots) \) consisting of
    aptly concatenated and flattened systems \( \seq{\vec L \uidx m}_{m
    \geq 0}\) and \(
    \vec L \uidx \infty\), respectively.
    Intuitively, \( \vec{\mathcal L} \) and \( \vec{\mathcal L}^+ \) are in a
    sense \emph{vectors of vectors}, but for convenience we
    call them just \emph{vectors}. Note that
    both \( \vec{\mathcal L}(z, \vec u) \) and
    \( \vec{\mathcal L}^+(z, \vec u) \) satisfy
    \begin{align}
        \begin{split}\label{eq:infinite:systems:phi:psi}
        \vec{\mathcal L}(z, \vec u)
            &=
        \vec \Phi(\vec{\mathcal L}(z, \vec u), z, \vec u)\\
        \vec{\mathcal L}^+(z, \vec u)
            &=
        \vec \Psi(\vec{\mathcal L}^+(z, \vec u), z, \vec u)
        \end{split}
    \end{align}
    where
    \begin{align}
        \begin{split}
            \vec \Phi(\vec \lambda, z, \vec u) &=
        (
        \vec{\mathcal K} \uidx 0(\vec \lambda_0, \vec \lambda_1, z, \vec u),
        \vec{\mathcal K} \uidx 1(\vec \lambda_1, \vec \lambda_2, z, \vec u),
        \ldots)\\
            \vec \Psi(\vec \lambda, z, \vec u) &=
        (
        \vec{\mathcal K} \uidx \infty(\vec \lambda_0, \vec \lambda_1, z, \vec u),
        \vec{\mathcal K} \uidx \infty(\vec \lambda_1, \vec \lambda_2, z, \vec u),
        \ldots)
        \end{split}
    \end{align}
    with \( \vec \lambda \) taken as a flattening concatenation of \( d
    \)-dimensional vectors of free variables \( (\vec \lambda_0, \vec \lambda_1,
    \ldots) \).

    Since for each \( m \) we have
    \( \vec{\mathcal K} \uidx m(\vec \lambda_m, \vec \lambda_{m+1}, z, \vec u)
    \preceq
    \vec{\mathcal K} \uidx \infty(\vec \lambda_m, \vec \lambda_{m+1}, z, \vec u)
    \) it also holds
    \begin{equation}
        \vec \Phi(\vec{\lambda}, z, \vec u)
        \preceq
        \vec \Psi(\vec{\lambda}, z, \vec u).
    \end{equation}

    The idea of the current proof is to consider the difference \( \vec{\mathcal
    L}^+(z, \vec u) - \vec{\mathcal L}(z, \vec u) \) and show that it is
    non-negative.  According to~\eqref{eq:infinite:systems:phi:psi} this
    difference can be represented as
    \begin{equation}
    \label{eq:difference:Linfty:Lm}
        \vec{\mathcal L}^+(z, \vec u) - \vec{\mathcal L}(z, \vec u) =
        \vec \Psi(\vec{\mathcal L}^+(z, \vec u), z, \vec u)
        -
        \vec \Phi(\vec{\mathcal L}(z, \vec u), z, \vec u).
    \end{equation}
    Since \( \vec \Psi(\vec \lambda, z, \vec u) \succeq \vec \Phi(\vec \lambda,
    z, \vec u)\),
    the formal power series \( \vec \Psi \) can be represented as a sum
    \( \vec \Psi(\vec \lambda, z, \vec u) = \vec \Phi(\vec \lambda, z, \vec u)
    + \vec \Theta(\vec \lambda, z, \vec u) \) with \( \vec \Theta(\vec \lambda,
    z, \vec u) \succeq
    \vec 0 \). Hence, the
    difference~\eqref{eq:difference:Linfty:Lm} becomes
    \begin{equation}
        \label{eq:difference:Linfty:Lm:second}
        \vec{\mathcal L}^+(z, \vec u) - \vec{\mathcal L}(z, \vec u) =
        \vec \Theta(\vec{\mathcal L}^+, z, \vec u)
        +
        \left(
        \vec \Phi(\vec{\mathcal L}^+, z, \vec u)
        -
        \vec \Phi(\vec{\mathcal L}, z, \vec u)
        \right).
    \end{equation}
    At this point, our tactic is to apply an analog of the mean value theorem
    to the right-hand side difference and obtain an equation of the form
    \begin{equation}
        \vec{\mathcal L}^+(z, \vec u) - \vec{\mathcal L}(z, \vec u) =
        \vec \Theta(\vec{\mathcal L}^+, z, \vec u)
        +
        \vec{\mathcal J} \cdot
        \left(
        \vec{\mathcal L}^+
        -
        \vec{\mathcal L}
        \right)
    \end{equation}
    and consequently
    \begin{equation}
        \vec{\mathcal L}^+(z, \vec u) - \vec{\mathcal L}(z, \vec u) =
        (\mathbf I - \vec{\mathcal J})^{-1} \vec{\Theta}
        = \sum_{k \geq 0} \vec{\mathcal J}^k \vec{\Theta}
        \succeq \vec 0
    \end{equation}
    where \( \vec{\mathcal{J}}\) is some non-negative operator whereas \(
    \mathbf I \) is the corresponding identity.
    The rest of the proof is dedicated to formalising the above approach, in
    particular showing that the Neumann series
    \( \sum_{k \geq 0} \vec{\mathcal J}^k \) is well-defined.

    We start by noticing that due to the well-foundedness of the limiting
    system~\eqref{eq:infinite:system:limit:system} we have \( \vec{\mathcal
    L}^+(0, \vec u) = \vec 0 \). Since \( \vec \Phi(\vec \lambda,
    z, \vec u) \preceq \vec \Psi(\vec \lambda, z, \vec u) \) there also holds \(
    \vec \Phi(\vec 0, 0, \vec u) = \vec 0 \).  Furthermore, since there exists a unique
    formal power series solution of the equation \( \vec{\mathcal L}(0, \vec u)
    = \vec \Phi(\vec{\mathcal L}(0, \vec u), 0, \vec u)\) and \( \vec{\mathcal
    L}(0, \vec u) = \vec 0 \) satisfies this equation, we note that indeed \( \vec{\mathcal
    L}(0, \vec u) = \vec 0 \).

    Consider two formal infinite-dimensional variables
    \( \vec{\lambda} \) and \( \vec{\lambda}^+ \) which are both
    flattened concatenations of \( d \)-dimensional vectors of free variables.
    Let us show that the difference
    \(
        \vec \Phi(\vec{\lambda}^+, z, \vec u)
        -
        \vec \Phi(\vec{\lambda}, z, \vec u)
    \) can be represented as
    \begin{equation}
    \label{eq:intermediate:value:for:fps}
        \vec \Phi(\vec{\lambda}^+, z, \vec u)
        -
        \vec \Phi(\vec{\lambda}, z, \vec u)
        =
        \vec{\mathcal J}(z, \vec u, \vec \lambda, {\vec \lambda}^+)(
            \vec \lambda^+ - \vec \lambda
        )
    \end{equation}
    where \( \vec{\mathcal J} = \vec{\mathcal J}(z, \vec u, \vec \lambda, \vec
    \lambda^+) \) is some non-negative operator
    (i.e.~infinite-dimensional matrix).
    Moreover, after substituting
    \( \vec \lambda = \vec{\mathcal L}(z, \vec u) \) and
    \( \vec \lambda^+ = \vec{\mathcal L}^+(z, \vec u) \)
    into \( \vec{\mathcal J} \),
    there exists a non-negative integer \( K > 0 \) such that \( \vec{\mathcal J}^K \) is
    element-wise divisible by \( z \).
    Since we have established that both
    \( \vec{\mathcal L}^+(0, \vec u) = \vec 0 \) and
    \( \vec{\mathcal L}(0, \vec u) = \vec 0 \),
    the latter condition is equivalent to the nilpotency of the operator
    \( \vec{\mathcal J} \) evaluated at
    \( \vec \lambda = \vec{\mathcal L}(z, \vec u) \),
    \( \vec \lambda^+ = \vec{\mathcal L}^+(z, \vec u) \)
    and \( z = 0 \).

    Note that the function \( \vec \Phi(\vec{\lambda}, z, \vec u) \) is a sum of
    (finite) monomials in formal variables \( (\vec{\lambda}, z, \vec u) \);
    although \( \vec \lambda \) is infinitely-dimensional, each of the monomials
    involves only finitely many factors of \( \vec{\lambda} \).
    Let us consider the difference of arbitrary monomials in form of
    \begin{equation}\label{eq:monomial:diff}
    x_1^{n_1} \cdots x_k^{n_k} - y_1^{n_1} \cdots y_k^{n_k}.
    \end{equation}
    Note that we can rewrite~\eqref{eq:monomial:diff} as
\begin{align}
    \begin{split}
        x_1^{n_1} \cdots x_k^{n_k}
        - y_1^{n_1} \cdots y_k^{n_k} &=
        \left( x_1^{n_1} \cdots x_k^{n_k}
        - y_1^{n_1} x_2^{n_2} \cdots x_k^{n_k} \right)\\
        & \qquad +
        \left( y_1^{n_1} x_2^{n_2} \cdots x_k^{n_k}
        - y_1^{n_1} y_2^{n_2} x_3^{n_3} \cdots x_k^{n_k} \right) + \cdots\\
        & \qquad + \left( y_1^{n_1} \cdots y_{k-1}^{n_{k-1}} x_k^{n_k}
        - y_1^{n_1} \cdots y_k^{n_k} \right)\\
        &= \left(x_1 - y_1\right) x_2^{n_2} \cdots x_k^{n_k}
            \dfrac{\left(x_1^{n_1} - y_1^{n_1}\right)}{x_1 - y_1}\\
        & \qquad +
            \left(x_2 - y_2\right) y_1^{n_1} x_3^{n_3} \cdots x_k^{n_k}
            \dfrac{\left(x_2^{n_2} - y_2^{n_2}\right)}{x_2 - y_2} + \cdots\\
        & \qquad +
            \left(x_k - y_k\right) y_1^{n_1} \cdots y_{k-1}^{n_{k-1}}
        \dfrac{\left(x_k^{n_k} - y_k^{n_k}\right)}{x_k - y_k}.
    \end{split}
\end{align}
Note that each factor \(\dfrac{\left(x_i^{n_i} - y_i^{n_i}\right)}{x_i - y_i}\)
in the final sum is in fact a polynomial \( \displaystyle\sum_{j=0}^{n_i-1}
x_i^j y_i^{n_i-j-1} \).  Therefore, the difference \( x_1^{n_1} \cdots x_k^{n_k}
- y_1^{n_1} \cdots y_k^{n_k} \) can be represented as a scalar product of \(
{\vec x - \vec y := \left(x_i - y_i \right)_{i=1}^k} \) and a vector of formal
power series in \( \left(\vec x, \vec y\right) \). Furthermore, the
difference \( {\vec \Phi(\vec \lambda^+, z, \vec u) - \vec \Phi(\vec \lambda, z,
\vec u)} \) consists of the sums of such differences of monomials multiplied by
appropriate non-negative coefficients. Grouping these differences together,
we obtain the desired form~\eqref{eq:intermediate:value:for:fps}.

Next, as an intermediate step, let us now show that the Jacobian operator
\( \dfrac{\partial \vec \Psi}{\partial \vec{\lambda}}(\vec \lambda, z, \vec u) \)
is nilpotent at \( (z, \vec{\lambda}) = (0, \vec 0) \).
For convenience, set
\begin{equation}
    J_1(\vec u) := \left.
    \dfrac{
            \partial \vec{\mathcal K} \uidx \infty(
                \vec \ell_1,
                \vec \ell_2,
                z,
                \vec u
    )}{
        \partial \vec \ell_1
    }
    \right|_{\substack{z = 0\\ \vec \ell_1 = \vec 0\\ \vec \ell_2 = \vec 0}}
    \quad
    \text{and}
    \quad
    J_2(\vec u) := \left.
    \dfrac{
            \partial \vec{\mathcal K} \uidx \infty(
                \vec \ell_1,
                \vec \ell_2,
                z,
                \vec u
    )}{
        \partial \vec \ell_2
    }
    \right|_{\substack{z = 0\\ \vec \ell_1 = \vec 0\\ \vec \ell_2 = \vec 0}}.
\end{equation}
Then,
\begin{equation}
    J_1(\vec u) + J_2(\vec u)
    =
    \left.
    \dfrac{
        \partial \vec{\mathcal K} \uidx \infty
        (\vec \ell, \vec \ell, z, \vec u)
    }{
        \partial \vec \ell
    }
    \right|_{\substack{z = 0\\ \vec \ell = \vec 0}}.
\end{equation}
Since the limiting system is well-founded
(see~\autoref{definition:nested:systems}) the sum
\({J_1(\vec u) + J_2(\vec u)}\) is nilpotent.
Moreover, since each of the matrices \(J_1(\vec u) \) and \( J_2(\vec u)\)
is non-negative, there exists \( K \) such that all the summands of
the expanded binomial \( (J_1(\vec u) + J_2(\vec u))^K \) are zero.

On the other hand, note that following the definition of \( \vec \Psi(\vec
\lambda, z, \vec u) \) its Jacobian operator \( \dfrac{\partial \vec
\Psi}{\partial \vec{\lambda}} \) evaluated at \( (z, \vec \lambda) = (0, \vec 0)
\) admits the following block structure:
\begin{equation}
    \left.
    \dfrac{\partial \vec \Psi}{\partial \vec{\lambda}}
    \right|_{\substack{z = 0\\ \vec \lambda = \vec 0}}
    =
    \left[
        \begin{array}{c|c|c|c}
            J_1 & J_2 &     & \\
            \hline
                & J_1 & J_2 & \\
            \hline
                &     & J_1 & J_2 \\
            \hline
                &     &     & \ddots \\
        \end{array}
    \right].
\end{equation}
If we take the \( K \)th power of this operator, it will have a block structure
in which each block element consists of a sum of certain summands from the
binomial expansion of \( (J_1 + J_2)^K \). Since the latter is a zero matrix and
the summands corresponding to the blocks are non-negative and dominated by \(
(J_1 + J_2)^K \), all such summands are also zero. This implies that the
Jacobian operator \( \dfrac{\partial \vec \Psi}{\partial \vec{\lambda}} (\vec
\lambda, z, \vec u) \) evaluated at \( (z, \vec \lambda) = (0, \vec 0) \) is
nilpotent with a nilpotence index at most \( K \), i.e.~the nilpotence index of
the Jacobian operator \( \dfrac{\partial \vec{\mathcal K} \uidx \infty(\vec
\ell, \vec \ell, z, \vec u)} {\partial \vec \ell} \) evaluated at \( (z, \vec
\ell) = (0, \vec 0) \).

Now, let us show that the infinitely-dimensional matrix \( \vec{\mathcal J} (z,
\vec u, \vec \lambda, \vec \lambda^+) \) evaluated at \( (z, \vec \lambda, \vec
\lambda^+) = (0, \vec 0, \vec 0) \) is equal to the Jacobian operator \(
\dfrac{\partial \vec \Phi}{\partial \vec{\lambda}} \) evaluated at \( (z,
\vec{\lambda})  = (0, \vec 0) \). Recall that the operator \( \vec{\mathcal J}
\) is determined by the differences of monomials in \( \vec \Phi(\vec \lambda^+,
z, \vec u) - \vec \Phi(\vec \lambda, z, \vec u) \).  Monomials that have degree
zero in \( \vec \lambda \) or \( \vec \lambda^+ \) cancel out because they
depend only on the arguments \( z \) and \(\vec u \).  Likewise, monomials with
degree two or more in \( \vec \lambda \) or \( \vec \lambda^+ \) vanish after
the substitution \( \vec \lambda = \vec \lambda^+ = \vec 0 \).  The only type of
the terms that do not turn to zero upon substitution \( \vec \lambda = \vec
\lambda^+ = \vec 0 \) are terms coming from differences of monomials linear in
\( \vec \lambda \) or \( \vec \lambda^+ \). Note that such terms have the same
contribution to the infinitely-dimensional matrix \( \vec {\mathcal J} \) as the
corresponding terms of the infinitely-dimensional matrix \( \dfrac{\partial \vec
\Phi}{\partial \vec \lambda} \). Hence, \( \vec{\mathcal J} (z,
\vec u, \vec \lambda, \vec \lambda^+) \) evaluated at \( (z, \vec \lambda, \vec
\lambda^+) = (0, \vec 0, \vec 0) \) is indeed equal to the Jacobian operator \(
\dfrac{\partial \vec \Phi}{\partial \vec{\lambda}} \) evaluated at \( (z,
\vec{\lambda})  = (0, \vec 0) \).

The nilpotence of \( \vec{\mathcal J} \) evaluated at \( (z, \vec \lambda, \vec
\lambda^+) = (0, \vec 0, \vec 0) \) follows from the fact that
\( \vec \Psi(\vec \lambda, z, \vec u) \) is dominating \( \vec \Phi(\vec
\lambda, z, \vec u) \), and therefore, the corresponding Jacobian operator
\( \dfrac{\partial \vec \Psi}{\partial \vec \lambda} \)
(respectively its $N$th power) is dominating the operator
\( \dfrac{\partial \vec \Phi}{\partial \vec \lambda} \) (respectively its
$N$th power). For this reason, the latter Jacobian operator
\( \dfrac{\partial \vec \Phi}{\partial \vec \lambda} \), evaluated at \( (z, \vec \lambda, \vec
\lambda^+) = (0, \vec 0, \vec 0) \)
is nilpotent with the nilpotence index at most the corresponding nilpotence
operator of the former Jacobian
operator \( \dfrac{\partial \vec \Psi}{\partial \vec \lambda} \).

And so, we have established that \( \vec{\mathcal J} \) evaluated at \( \vec
\lambda = \vec{\mathcal L}(z, \vec u) \), \( \vec \lambda^+ = \vec{\mathcal
L}^+(z, \vec u) \) and \( z = 0 \) is nilpotent. Equivalently, it meas that
after substituting \( \vec \lambda = \vec{\mathcal L}(z, \vec u) \) and \( \vec
\lambda^+ = \vec{\mathcal L}^+(z, \vec u) \) into \( \vec{\mathcal J} \) there
exists a non-negative integer \( K > 0 \) such that \( \vec{\mathcal J}^K \) is
element-wise divisible by \( z \).  Consequently, each coefficient in \( z \) of
the formal sum \( \sum_{j \geq 0} \vec{\mathcal J}^j \) is finite.  Indeed, for
each integer \( N \geq 0 \), the coefficient at \( z^N \) in this formal series
is a sum of coefficients at \( z^N \) in the finite sum \( \sum_{j = 0}^{K\cdot
N} \vec{\mathcal J}^j \).  Moreover, since \( \vec{\mathcal J} \) is
non-negative, this sum is also non-negative. Finally, this infinite formal
series is equal to \( (\vec I - \vec{\mathcal J})^{-1} \) where \( \vec I \) is
the identity operator of appropriate dimension.
\end{proof}

\begin{remark}
    The condition that
    \( \vec{\mathcal K} \uidx \infty(\vec 0, \vec 0, 0, \vec u) = \vec 0 \)
    can be omitted but we keep it for the simplicity of the proof.
    For the above proof, it is enough to guarantee that each coefficient in \( z
    \) of the infinite formal sum \( \sum_{j \geq 0} \vec{\mathcal J}^j \) is
    finite, which is equivalent to saying that some power of \( \vec{\mathcal J}
    \) is divisible by \( z \). More details on well-founded systems can be
    found in~\cite{pivoteau2012algorithms}.
\end{remark}

\subsection{Coefficient transfer for infinite systems}
Finally, we give our main theorem on the transfer of coefficients for infinitely
nested forward-recursive systems.

\begin{theorem}\label{proposition:infinite:system}
    Let $\mathcal{S}$ be an infinitely nested, forward recursive
    system~\eqref{eq:infinite:system:forward:recursive:system} which tends to an
    irreducible context-free schema at an exponential convergence rate.  Then,
    the respective solutions \( L_j\uidx m (z, \vec u) \) of $\mathcal{S}$ admit
    for each \( m \geq 0 \) an asymptotic expansion of their coefficients
    as \( n \to \infty \) in form of
    \begin{equation}\label{eq:prop:infinite:sys:puiseux}
    [z^n]L_j \uidx m(z, \vec u) \sim [z^n] \sum_{k \geq 0} c_{j,k} \uidx m(\vec u) \left(
        1 - \dfrac{z}{\rho(\vec u)}
    \right)^{k/2}
\end{equation}
where \( \rho(\vec u) \) is the dominant singularity of the corresponding
    limiting system~\eqref{eq:infinite:system:limit:system} and the coefficients
    \( c_{j,k} \uidx m(\vec u) \) are analytic at \( \vec u = \vec 1 \).
    Furthermore, \( \rho(\vec u) \) is analytic near \( \vec u = \vec 1 \).

    The coefficients \( c_{j,k} \uidx m (\vec u) \) can be approximated by
    taking first \( (M-1) \) equations
    from~\eqref{eq:infinite:system:forward:recursive:system} and replacing
    the \( M \)th equation by its following limit variant:
    \begin{equation}
        \vec L \uidx M = \vec{\mathcal K} \uidx \infty
        \left(
            \vec L \uidx M,
            \vec L \uidx M,
            z, \vec u
        \right).
    \end{equation}
    Such a truncated system can be solved recursively. Consequently, the
    coefficients of respective Puiseux
    expansions~\eqref{eq:prop:infinite:sys:puiseux} are estimated with an error
    which is exponentially small in \( M \).
\end{theorem}

\begin{remark}\label{remark:conditions:limit:system}
    ~ 
    \begin{itemize}
        \item The condition that the system \( \mathcal S \) tends to an
            irreducible context-free schema can be replaced by a weaker
            condition asserting that the limiting
            system~\eqref{eq:infinite:system:limit:system} admits a suitable Puiseux
            expansion.
        \item Instead of limiting systems with square-root type singularities,
            it is also possible to consider rational systems or other types of
            systems. The same set of conditions is sufficient to establish the
            transfer of behaviours around the dominant singular point.
    \end{itemize}
\end{remark}

\begin{proof}{(\autoref{proposition:infinite:system})}
\def\vp{\vphantom{hp}}
We divide the proof into three conceptual parts.
\begin{enumerate}
\item First, we show that each component of the difference vector \( \vec L
    \uidx \infty(z, \vec u) - \vec L \uidx m(z, \vec u) \) can be upper bounded
        by a Puiseux series expansion whose coefficients decay exponentially
        fast as
        \( m \) tends to infinity;
    \item Next, we show that if for \( m \geq 1 \) there exist coordinate-wise upper and lower bounds
\begin{equation}
    \underline{\vp \vec L \uidx m}(z, \vec u)
    \preceq
    \vec L \uidx m(z, \vec u)
    \preceq
    \overline{\vp \vec L \uidx m}(z, \vec u),
\end{equation}
then the vector of functions \( \vec L \uidx {m-1}(z, \vec u) \)
obtained from the infinite system \( \mathcal S \) admits upper
        and lower bounds
\( \overline{\vp \vec L \uidx {m-1}} (z, \vec u) \)
and
\( \underline{\vp \vec L \uidx {m-1}} (z, \vec u) \)
satisfying
\begin{equation}
    \overline{\vp \vec L \uidx {m-1}}(z, \vec u)
    -
    \underline{\vp \vec L \uidx {m-1}}(z, \vec u)
    \preceq
    \vec {\mathcal R}(z, \vec u) \left(
    \overline{\vp \vec L \uidx {m}}(z, \vec u)
    -
    \underline{\vp \vec L \uidx {m}}(z, \vec u)
    \right)
\end{equation}
for some matrix \( \vec{\mathcal R} (z, \vec u) \) with spectral radius
satisfying
\( r(\vec {\mathcal R} (z, \vec u)) \leq 1 \) for \( z \in [0, \rho(\vec u)] \)
where \( \rho(\vec u) \) is the dominant singularity of \( \vec L \uidx
\infty(z, \vec u) \);
\item Finally, we combine two previous results and prove that the difference
    between the Puiseux coefficients of upper and lower bounds of
    \( \vec L \uidx m(z, \vec u) \) can be reduced to zero.
\end{enumerate}
\textit{First part.}
According to~\autoref{lemma:inclusion} we have \( \vec L \uidx m \preceq \vec L
\uidx \infty \).
Following the functional definitions of \( \vec L \uidx \infty \) and
\( \vec L \uidx m \) from the infinite system of equations,
their difference can be represented as
\begin{equation}
    \vec L \uidx \infty - \vec L \uidx m
    =
    \vec{\mathcal K} \uidx \infty(
        \vec L \uidx \infty,
        \vec L \uidx \infty, z, \vec u
    )
    -
    \vec{\mathcal K} \uidx m(
        \vec L \uidx m,
        \vec L \uidx{m+1}, z, \vec u
    ).
\end{equation}
For convenience, henceforth we omit the arguments \( z \) and \( \vec u \).
Moreover, \( \vec{\mathcal K} \uidx \infty \) and
\( \vec{\mathcal K} \uidx m \) become functions of two vector arguments
\begin{equation}
    \vec{\mathcal K} \uidx \infty \colon \mathbb C^{d} \times
    \mathbb C^d \to \mathbb C^d
    \quad \text{and} \quad
    \vec{\mathcal K} \uidx m \colon \mathbb C^{d} \times
    \mathbb C^d \to \mathbb C^d.
\end{equation}
In addition, we use the nabla notation to denote the Jacobian operator
\begin{equation}
    \nabla_{\vec x} \vec{\mathcal K}(\vec x, \vec y)
    = \left(
        \dfrac{\partial}{\partial x_1} \vec{\mathcal K}
        ,
        \ldots,
        \dfrac{\partial}{\partial x_d} \vec{\mathcal K}
    \right)
    \quad \text{and} \quad
    \nabla_{\vec y} \vec{\mathcal K}(\vec x, \vec y)
    = \left(
        \dfrac{\partial}{\partial y_1} \vec{\mathcal K}
        ,
        \ldots,
        \dfrac{\partial}{\partial y_d} \vec{\mathcal K}
    \right).
\end{equation}

We start with the following subtraction-addition trick. Each component of the
vector difference is evaluated through~\autoref{lemma:mean:value} (mean value
lemma for formal power series) and then upper-bounded by the value of the
functional \( \vec{\mathcal K} \uidx \infty \) at \( \vec L \uidx \infty \).
Specifically,
\begin{align*}
    \begin{split}
    \vec L \uidx \infty - \vec L \uidx m
    & =
    \vec{\mathcal K} \uidx \infty (
        \vec L \uidx \infty,
        \vec L \uidx \infty
    )
    -
    \vec{\mathcal K} \uidx \infty (\vec L \uidx m, \vec L \uidx \infty)
    +
    \vec{\mathcal K} \uidx \infty (\vec L \uidx m, \vec L \uidx \infty)
    \\&
    -
    \vec{\mathcal K} \uidx \infty (\vec L \uidx m, \vec L \uidx {m+1})
    +
    \vec{\mathcal K} \uidx \infty (\vec L \uidx m, \vec L \uidx {m+1})
    -
    \vec{\mathcal K} \uidx m (\vec L \uidx m, \vec L \uidx {m+1})
    \\
    & \preceq
    \left.
    \nabla_{\vec x} \vec{\mathcal K} \uidx \infty (\vec x, \vec y)
    \right|_
        {
            \substack{
                \vec x = \vec L \uidx \infty \\
                \vec y = \vec L \uidx \infty
            }
        }
        (\vec L \uidx \infty - \vec L \uidx m)
    +
    \left.
    \nabla_{\vec y} \vec{\mathcal K} \uidx \infty (\vec x, \vec y)
    \right|_
        {
            \substack{
                \vec x = \vec L \uidx \infty \\
                \vec y = \vec L \uidx \infty
            }
        }
        (\vec L \uidx \infty - \vec L \uidx {m+1})
    \\ &
        + \Big(
        \vec{\mathcal K} \uidx \infty (\vec L \uidx m, \vec L \uidx {m+1})
        -
        \vec{\mathcal K} \uidx m (\vec L \uidx m, \vec L \uidx {m+1})
        \Big).
    \end{split}
\end{align*}
Since \( \vec{\mathcal K} \uidx m \preceq \vec{\mathcal K} \uidx \infty \), the
final difference in the above sum is a vector of formal power series with non-negative
coefficients. Consequently, the last summand can be bounded by \( \vec{\mathcal
K} \uidx \infty(\vec L \uidx \infty, \vec L \uidx \infty) - \vec{\mathcal K}
\uidx m(\vec L \uidx \infty, \vec L \uidx \infty) \). By the condition of
exponential convergence, this difference can be even further bounded by a vector of
exponentially decaying functions \( \vec A(z, \vec u) (B(z, \vec u))^m \).
For brevity, set
\begin{equation*}
    \partial_1 \vec{\mathcal K}
    :=
    \left.
    \nabla_{\vec x} \vec{\mathcal K} \uidx \infty (\vec x, \vec y)
    \right|_
    {(\vec x, \vec y) = (\vec L \uidx \infty, \vec L \uidx \infty)}
    \quad \text{and} \quad
    \partial_2 \vec{\mathcal K}
    :=
    \left.
    \nabla_{\vec y} \vec{\mathcal K} \uidx \infty (\vec x, \vec y)
    \right|_
    {(\vec x, \vec y) = (\vec L \uidx \infty, \vec L \uidx \infty)}.
\end{equation*}
Note that
\( \partial_1 \vec{\mathcal K}, \partial_2 \vec{\mathcal K} \in \mathbb C^{d
\times d} \).
Now, the upper bound on \( \vec L \uidx \infty -
\vec L \uidx m \) can be stated as
\begin{equation}
    (\vec L \uidx \infty - \vec L \uidx m)
    \preceq
    \vec A(z, \vec u) {B(z, \vec u)}^m
    +
    \partial_1
    \vec{\mathcal K} \cdot
    (\vec L \uidx \infty - \vec L \uidx {m} )
    +
    \partial_2
    \vec{\mathcal K} \cdot
    (\vec L \uidx \infty - \vec L \uidx {m+1} )
\end{equation}
or, equivalently,
\begin{equation}\label{eq:specral:radius:ineq}
    (\vec I_d - \partial_1 \vec{\mathcal K})
    (\vec L \uidx \infty - \vec L \uidx m)
    \preceq
    \vec A(z, \vec u) {B(z, \vec u)}^m
    +
    \partial_2
    \vec{\mathcal K} \cdot
    (\vec L \uidx \infty - \vec L \uidx {m+1} )
\end{equation}
where \( \vec I_d \in \mathbb C^{d \times d} \) is the \( d \times d \) identity
matrix.

Let us show that the inverse matrix \( (\vec I_d - \partial_1 \vec{\mathcal
K})^{-1} \) exists and has non-negative coefficients in the sense of formal
power series. As discussed in the proof of~\autoref{lemma:inclusion}, the matrix
\( \partial_1 \vec{\mathcal{K}} \) is nilpotent at $z = 0$. Equivalently, there
exists a non-negative integer $K$ such that \( \left(\partial_1
\vec{\mathcal{K}}\right)^K \) is divisible by $z$.  It means that the formal
series \( \sum_{j \geq 0}\left(\partial_1 \vec{\mathcal{K}}\right)^j \) is
well-defined and hence so is the formal inverse \( (\vec I_d - \partial_1
\vec{\mathcal K})^{-1} \). Moreover, since \(\partial_1 \vec{\mathcal{K}} \) has
non-negative coefficients, the same holds for the investigated inverse matrix
\( (\vec I_d - \partial_1 \vec{\mathcal K})^{-1} \).

Now, let us focus on the behaviour of \( (\vec I_d - \partial_1 \vec{\mathcal
K})^{-1} \) as a function near $z = \rho( \vec u)$.  As follows
from~\autoref{proposition:dominant:singularity} applied to the system of
equations
\(
    \vec L \uidx \infty =
    \vec{\mathcal K} \uidx \infty(
        \vec L \uidx \infty,
        \vec L \uidx \infty,
        z,
        \vec u
    )
\), for each real \( 0 < z < \rho(\vec u) \) we have
the following inequality:
\begin{equation}
    r(\partial_1 \vec{\mathcal K} + \partial_2 \vec{\mathcal K}) < 1.
\end{equation}
By Perron--Frobenius theorem (see e.g.~\cite[section 2.2.5]{drmota2009random}
and references therein) the spectral radius of a matrix with positive entries is
monotonic in its coefficients. Hence, for all real \( 0 < z < \rho(\vec u)\)
\begin{equation}
    r(\partial_1 \vec{\mathcal K}) < 1
\end{equation}
and so, in the same interval
\begin{equation}
    (\vec I_d - \partial_1 \vec{\mathcal K})^{-1}
    =
    \sum_{j \geq 0} (\partial_1 \vec{\mathcal K})^j.
\end{equation}
Moreover, due to the continuity of the spectral radius, the same identity
can be extended to some complex neighbourhood of \( \vec u = \vec{1}\).

Consequently, we can multiply both sides of~\eqref{eq:specral:radius:ineq} by \(
( \vec I_d - \partial_1 \vec{\mathcal K} )^{-1} \) and obtain
\begin{equation}
    \vec L \uidx \infty - \vec L \uidx m
    \preceq
    (\vec I_d - \partial_1 \vec{\mathcal K})^{-1}
    \vec A(z, u) {B(z, \vec u)}^m
    +
    (\vec I_d - \partial_1 \vec{\mathcal K})^{-1}
    \partial_2
    \vec{\mathcal K}
    (\vec L \uidx \infty - \vec L \uidx {m+1} )
\end{equation}
Let us denote \({
    \vec \delta_m
    :=
    (\vec I_d - \partial_1 \vec{\mathcal K})^{-1}
    \vec A(z, \vec u) B(z, \vec u)^m
}\) and
\({
    \vec {\mathcal R}
    :=
    (\vec I_d - \partial_1 \vec{\mathcal K})^{-1} \partial_2 \vec{\mathcal K}
}\).
Note that the inequality
\(
    \vec L \uidx \infty - \vec L \uidx m
    \preceq
    \vec \delta_m + \vec{\mathcal R}
    ( \vec L \uidx \infty - \vec L \uidx {m+1} )
\)
can be further iterated for increasing values of \( m \).
In doing so, we find that
\begin{equation}
    \vec L \uidx \infty - \vec L \uidx m
    \preceq
    \vec \delta_m +
    \vec{\mathcal R}
    \vec \delta_{m+1} +
    \vec{\mathcal R}^2
    \vec \delta_{m+2} +
    \cdots
\end{equation}
Hence, the difference \( \vec L \uidx \infty - \vec L \uidx m \)
can be bounded by the tail of a geometric progression
which appears as a summation of the formal Neumann series
\begin{equation}
    \label{eq:upper:bound}
    \vec L \uidx \infty - \vec L \uidx m
    \preceq
    B(z, u)^m
    \sum_{k \geq 0}
    (B(z, \vec u)\vec{\mathcal R})^k
    (\vec I_d - \partial_1 \vec{\mathcal K})^{-1}
    \vec A(z, \vec u).
\end{equation}

Let us now focus on the above formal Neumann series.
Note that applying~\autoref{proposition:irreducible:polynomial:systems} (Drmota--Lalley--Woods
theorem) to the limiting system~(\ref{eq:infinite:system:limit:system}), we obtain that the
vector of functions \( \vec L \uidx \infty(z, \vec u) \) admits a
coordinate-wise Puiseux expansion in form of
\begin{equation}
    \vec L \uidx \infty (z, \vec u) \sim
    \vec \ell_0(\vec u) - \vec \ell_1(\vec u) \sqrt{1 - \dfrac{z}{\rho(\vec
    u)}}
\end{equation}
where functions \( \vec\ell_0(\vec u), \vec \ell_1(\vec u), \rho(\vec u) \) are
analytic near \( \vec u = 1 \). Likewise, both matrices \( \partial_1 \vec{\mathcal K} \)
and \( \partial_2 \vec{\mathcal K} \) admit Puiseux expansions of the same kind.

Let us prove that coordinate-wise Puiseux expansions of the matrix \(
\vec{\mathcal R} \) near the singular point \( z = \rho(\vec u) \) have the form
\begin{equation}
    \vec{\mathcal R} =
    (\vec I_d - \partial_1 \vec{\mathcal K})^{-1}
    \partial_2 \vec{\mathcal K}
    \sim \vec{\mathcal R}_0
    -
    \vec{\mathcal R}_1 \sqrt{1 - \dfrac{z}{\rho(\vec u)}}
    , \quad
    z \to \rho(\vec u)
\end{equation}
where the spectral radius of \( \vec{\mathcal R}_0 \) satisfies \(
r(\vec{\mathcal R}_0) = 1 \).  According to Perron--Frobenius theorem, since the
coefficients of \( \vec{\mathcal R} \) are non-negative, the eigenvalue of the
matrix \( \vec{\mathcal R} \) with the largest absolute value, i.e.~the
eigenvalue corresponding to the spectral radius of \( \vec{\mathcal R} \), is
the largest real positive solution \( \lambda \) of the characteristic equation
\begin{equation}
    \label{eq:characteristic:equation}
    \det \left(
        (\vec I_d - \partial_1 \vec{\mathcal K})^{-1}
        \partial_2 \vec{\mathcal K} - \lambda \vec I_d
    \right)
     = 0.
\end{equation}
Since the determinant of a matrix product is equal to the product of respective
determinants and \( \det \left( \vec I_d - \partial_1 \vec{\mathcal{K}} \right)
\neq 0 \), as \(\vec I_d - \partial_1 \vec{\mathcal{K}}\) is invertible, the
above condition is equivalent to
\begin{equation}
    \det\left(
        \partial_2 \vec{\mathcal K}
        + \lambda \partial_1 \vec{\mathcal K}
        - \lambda \vec I_d
    \right) = 0
    \quad \text{and also} \quad
    \det(
        \partial_1 \vec{\mathcal K}
        +
        \lambda^{-1}\partial_2 \vec{\mathcal K}
        - \vec I_d
    ) = 0.
\end{equation}
Let us show that the largest positive real solution (as a function of \( z <
\rho(\vec u)\)) of this equation, does not exceed \( 1 \), with equality when
\( z = \rho(\vec u) \).
Assume, by contrary, that \( \lambda > 1 \).
The matrix \( (\partial_1 \vec{\mathcal K} + \lambda^{-1}
\partial_2 \vec{\mathcal K}) \) is a matrix with non-negative coefficients,
whose coefficients are strictly smaller
than the coefficients of the matrix
\( (\partial_1 \vec{\mathcal K} + \partial_2 \vec{\mathcal K}) \).
By Perron--Frobenius theorem (see e.g.~\cite[section 2.2.5]{drmota2009random}
and references therein), the spectral radius of a matrix with positive
coefficients is monotonic in its coefficients, so for \( \lambda > 1 \)
\begin{equation}
    r(
        \partial_1 \vec{\mathcal K}
        +
        \lambda^{-1}
        \partial_2 \vec{\mathcal K}
    )
    <
    r(
        \partial_1 \vec{\mathcal K}
        +
        \partial_2 \vec{\mathcal K}
    ) = 1.
\end{equation}
Therefore, the characteristic equation cannot have a solution \( \lambda >
1 \). Moreover, the spectral radius of \( \vec{\mathcal R}_0 \) is equal to the
spectral radius of \( \vec{\mathcal R} \) when \( z = \rho(\vec u) \) because in
this case, the two matrices coincide. Therefore, \( r(\vec{\mathcal R}_0) = 1
\).

Moving back to the upper bound~\eqref{eq:upper:bound},
according to the exponential convergence condition
in~\autoref{definition:infinitely:nested:systems},
in a complex vicinity of \( \vec u = \vec 1 \), the absolute value of the
function \( B(z, \vec u) \) is strictly smaller
than \( 1 \), hence the inverse matrix
\(
    (
        \vec I_d
        -
        B(z,u)
        \vec{\mathcal R}
    )^{-1}
\) in~(\ref{eq:upper:bound}) exists.
Moreover, since
\begin{equation}
    A^{-1}= \dfrac{1}{\det\left(A\right)} \cdot \mathbf{adj}(A)
\end{equation}
where \( \mathbf{adj}(A) \) is the adjugate matrix of \(A\),
each element of the inverse matrix
\(
    (
        \vec I_d
        -
        B(z,u)
        \vec{\mathcal R}
    )^{-1}
\)
can be represented as a ratio of a sum of products of functions admiting
Puiseux series expansions in form of
\(
\vec a(\vec u) -
{\vec b(\vec u) \sqrt{1 - z / \rho(\vec u)}}
+ {O(| 1 - z / \rho(\vec u)|)}
\)
and a non-zero determinant of
\(
        {\vec I_d
        -
        B(z,u)
        \vec{\mathcal R}}
\).
It means that each coordinate in the inverse matrix
\(
    (
        \vec I_d
        -
        B(z,u)
        \vec{\mathcal R}
    )^{-1}
\)
also admits a Puiseix series expansion of similar form.  Thus,
the Neumann series in~\eqref{eq:upper:bound} converges and we obtain
\begin{equation}
    \label{eq:upper:bound:next}
    \vec L \uidx \infty - \vec L \uidx m
    \preceq
    B(z, \vec u)^m
    \times
    \left(
        \vec I_d - B(z, \vec u) \vec{\mathcal R}
    \right)^{-1}
    (\vec I_d - \partial_1 \vec{\mathcal K})^{-1}
    \vec A(z, \vec u).
\end{equation}

From~(\ref{eq:upper:bound:next}) we now note that the difference \( \vec L \uidx
\infty - \vec L \uidx m \) can be bounded by a vector of functions having the
same singularity \( \rho(\vec u) \) as the components of the vector \( \vec L
\uidx \infty \). The Puiseux coefficients of this upper bound decay
exponentially fast as \( m \to \infty \).  Moreover, these coefficients are
analytic functions near \( \vec u = \vec 1 \).

\textit{Second part.} Assume that function \( \vec L \uidx m(z, \vec u) \) admits some
upper and lower bounds and denote by \( \Delta_m(z, \vec u) \) the difference
between these bounds:
\def\vp{\vphantom{hp}}
\begin{equation}
    \underline{\vp \vec L \uidx m}(z, \vec u)
    \preceq
    \vec L \uidx m(z, \vec u)
    \preceq
    \overline{\vp \vec L \uidx m}(z, \vec u) ,
    \quad
    \Delta \uidx m(z, \vec u) :=
    \overline{\vp \vec L \uidx m}(z, \vec u)
    -
    \underline{\vp \vec L \uidx m}(z, \vec u)
    .
\end{equation}
Then, another pair of upper and lower bound can be
established for \( \vec L \uidx {m-1} \) from~\eqref{eq:infinite:system:forward:recursive:system}
(infinite system of equations) with the difference \( \Delta \uidx {m-1} \) satisfying
\begin{equation}
    \Delta \uidx {m-1} = \vec{\mathcal K} \uidx {m-1}
    (\overline{\vp \vec L \uidx {m-1}}, \overline{\vp \vec L \uidx m})
    -
    \vec{\mathcal K} \uidx {m-1}
    (\underline{\vp \vec L \uidx {m-1}}, \underline{\vp \vec L \uidx m})
    \preceq
    \partial_1 \vec{\mathcal K}
    \cdot \Delta \uidx {m-1}
    +
    \partial_2 \vec{\mathcal K}
    \cdot \Delta \uidx {m}
    .
\end{equation}
That is, repeating the argument that allows to multiply both sides of the
inequality by the inverse matrix, we obtain
\begin{equation}
\label{eq:chain:of:bounds}
    \Delta \uidx {m-1}
    \preceq
    (\vec I_d - \partial_1 \vec{\mathcal K})^{-1}
    \partial_2 \vec{\mathcal K}
    \cdot
    \Delta \uidx m
    .
\end{equation}
As we discovered in the first part, the matrix
\( \vec{\mathcal R} := (\vec I_d - \partial_1 \vec{\mathcal K})^{-1} \partial_2
\vec{\mathcal K} \) has spectral radius \( 1 \) at the singular point
\( z = \rho(\vec u) \).

\textit{Third part.}
As a result of the first part, we know that \( \vec L \uidx \infty(z, \vec u) -
\vec L \uidx m(z, \vec u) \) can be bounded in the following manner:
\begin{equation}
    \vec 0
    \preceq
    \vec L \uidx \infty(z, \vec u) - \vec L \uidx m(z, \vec u)
    \preceq
    B(z, \vec u)^m
    (\vec I_d - B(z, u) \vec{\mathcal R})^{-1} \vec A (z, \vec u).
\end{equation}
Let us assign
\begin{equation}
    \Delta_0 \uidx m :=
    B(z, \vec u)^m
    (\vec I_d - B(z, u) \vec{\mathcal R})^{-1} \vec A (z, \vec u)
\end{equation}
for the difference between upper and lower bounds for the vector of
functions \( \vec L \uidx m(z, \vec u) \).
Next, using the result of the second part, we construct a family of differences
\( \Delta_k \uidx {m} \) between upper and lower bounds for
\( \vec L \uidx m (z, \vec u) \), so that for every \( k, m \geq 0 \)
it holds
\begin{equation}
    \underline{\vec L_k \uidx m}(z, \vec u)
    \preceq
    \vec L \uidx m(z, \vec u)
    \preceq
    \overline{\vec L_k \uidx m}(z, \vec u)
    \quad \text{and} \quad
    \Delta_k \uidx m :=
    \overline{\vec L_k \uidx m}(z, \vec u)
    -
    \underline{\vec L_k \uidx m}(z, \vec u)
    .
\end{equation}
The family of upper and lower bounds is defined using the procedure described in
the second part. More specifically, for every \( m \geq 1 \) and \( k \geq 0 \)
these bounds satisfy the equations
\begin{equation}
    \begin{split}
        \overline{\vec L_{k+1} \uidx {m-1}}(z, \vec u)
        &:=
        \vec{\mathcal K}\uidx {m-1}
        (
            \overline{\vec L_{k+1} \uidx {m-1}}(z, \vec u)
            ,
            \overline{\vec L_{k} \uidx {m}}(z, \vec u)
            , z, \vec u
        );
        \\
        \underline{\vec L_{k+1} \uidx {m-1}}(z, \vec u)
        &:=
        \vec{\mathcal K}\uidx {m-1}
        (
            \underline{\vec L_{k+1} \uidx {m-1}}(z, \vec u)
            ,
            \underline{\vec L_{k} \uidx {m}}(z, \vec u)
            , z, \vec u
        ).
    \end{split}
\end{equation}

According to the second part, the differences
\( \Delta_k \uidx m \) satisfy formal power series inequalities
\(\Delta_{k+1} \uidx {m-1} \preceq \vec {\mathcal R} \Delta_{k} \uidx{m} \).
By iteration, we thus obtain
\begin{equation}
    \Delta_m \uidx 0
    \preceq
    \vec {\mathcal R}^m \Delta_0 \uidx m.
\end{equation}
Since the spectral radius of \( \vec {\mathcal R} \) is bounded by \( 1 \),
and \( \Delta_0 \uidx m \) is exponentially small in \( m \), the
values of Puiseux coefficients of \( \vec L \uidx 0(z, \vec u) \) can be
approximated within an exponentially small in \( m \) gap, for
arbitrarily large value of \( m \).

Finally, we note that the functions
\( \underline{\vec L_m \uidx 0}(z, \vec u) \) and
\( \overline{\vec L_m \uidx 0}(z, \vec u) \) have Puiseux expansions of type
\begin{equation}
    f_m(z, \vec u) \sim c_m(\vec u) - a_m(\vec u) \sqrt{1 - \dfrac{z}{\rho(\vec
    u)}}
\end{equation}
in a certain delta-domain. According to~\autoref{proposition:transfer:theorem}
(transfer theorem), their coefficients admit the following asymptotic estimate:
\begin{equation}
    f_m(z, \vec u) \sim C_n A_m(\vec u) B(\vec u)^n.
\end{equation}
A final application of~\autoref{lemma:squeeze} (squeeze lemma for formal power
series) combined with~\autoref{remark:squeeze} finishes the proof.
\end{proof}

\begin{remark}
    In~\autoref{proposition:infinite:system} we prove a so-called \emph{weak}
    transfer theorem, i.e.~prove that the asymptotics of the coefficients of
    each \( \vec L \uidx m \) can be obtained by taking the asymptotic expansion
    of the correseponding Puiseux expansion. A stronger version would suggest
    that the generating functions \( \vec L \uidx m(z, \vec u) \) can be
    analytically continued beyond the circle of convergence of corresponding
    formal power series, in a certain delta-domain. However, for our analysis,
    the presented weak variant is enough. The techniques presented above, can be
    further extended to obtain a stronger transfer theorem, by computing the
    Taylor series expansions at points \( z_0 \) inside the circle of
    convergence.
\end{remark}

%% file: 6-advanced-marking.tex
\section{Advanced marking}\label{sec:advanced:marking}
In the following section we investigate more parameters related to plain and
closed \lterms. In particular, we consider:
\begin{itemize}
    \item Several parameters related to closed \lterms, resulting in Gaussian
        limit laws;
    \item Further parameters whose limiting distributions are discrete,
        including the leftmost-outermost redex search time in closed terms, the
        number of free variables in plain terms, the number of head abstractions in
        closed terms, and mean degree profile in closed terms;
    \item Finally, the mean height profile of closed terms for several different
        notions of height.
\end{itemize}

\subsection{$m$-openness and the enumeration of closed terms}\label{subsec:m:openness}
Recall that a term is said to be $m$\nobreakdash-open
(see~\autoref{subsec:lambda:calculus}) if by prepending it with $m$ head
abstractions we obtain a closed \lterm~as a result. Following this
natural, hierarchical notion, the set $\mathcal L_m$ of \( m \)-open \lterms~can
be specified as
\begin{align}\label{eq:advanced:marking:Lm:spec}
    \begin{split}
        \classL[m] &::= `l \classL[m+1]~|~(\classL[m] \classL[m])~|~ \idx{0},
        \idx{1}, \ldots, \idx{m - 1}\\
        \classL[m+1] &::= `l \classL[m+2]~|~(\classL[m+1] \classL[m+1])~|~ \idx{0},
        \idx{1}, \ldots, \idx{m}\\
        \ldots & \qquad \ldots
    \end{split}
\end{align}

Let \( L_m(z) \) denote the generating function associated with the set of $m$\nobreakdash-open
\lterms, i.e.
\(
{L_m(z) = \sum_{n \geq 0} a_{n,m} z^n}
\)
where \( a_{n,m} \) stands for the number of \( m \)\nobreakdash-open lambda
terms of size \( n \). Using~\eqref{eq:advanced:marking:Lm:spec} we obtain a
corresponding infinite system for the functions \( L_m(z) \):
\begin{align}
    \label{eq:infinite:system}
    \begin{split}
        L_0(z) &= z L_1(z) + z {L_0(z)}^2
    , \\
        L_1(z) &= z L_2(z) + z {L_1(z)}^2
    + z
    , \\
    & \cdots
    , \\[-.3cm]
        L_m(z) &= z L_{m + 1}(z) + z {L_m(z)}^2
    + z \dfrac{1 - z^m}{1 - z}
    , \\[-.15cm]
    & \cdots
    \end{split}
\end{align}

In~\cite[Lemma 8]{BodiniGitGol17} the authors prove that for each \( m \geq 0
\) the generating functions for $m$\nobreakdash-open \lterms~$L_m(z)$ admit Puiseux
expansions in form of
\begin{equation}
    L_m(z) \sim a_m - b_m \sqrt{1 - \dfrac{z}{\rho}}.
\end{equation}
Moreover, by the virtue of their proof, we obtain an suitable approximation
procedure for the coefficients \( a_m \) and \( b_m \) by truncating the
system~\eqref{eq:infinite:system} and replacing the function \( L_m(z) \) with
\( L_\infty(z) \). Furthermore, the estimated coefficients \( \widetilde{a}_m \)
and \( \widetilde{b}_m \) tend to their respective limits with an error of order
\( O(\frac{1}{\sqrt{m}}) \).  Using~\autoref{proposition:infinite:system} it is
possible to prove that \( \widetilde{a}_m \) and \( \widetilde{b}_m \) converge
to their respective limits exponentially fast.  Consequently, the approximation
procedure proposed in~\cite{BodiniGitGol17} convergences exponentially fast, as
well.

Immediately, this implies that the probability that a random plain \lterms~is
$m$\nobreakdash-open, but not $(m-1)$\nobreakdash-open is $\dfrac{b_m -
b_{m-1}}{b_\infty}$. Certainly, the limiting distribution associated with
$m$\nobreakdash-openness is discrete.

Note that the probability distribution function of $m$\nobreakdash-openness is
proportional to the coefficient at $z^n$ in the bivariate generating function
\begin{equation}
    L(z, u) = \sum_{k \geq 1} u^k \big(L_k(z) - L_{k-1}(z)\big)
    \sim
    \sum_{k \geq 1} u^k (a_k - a_{k-1}) -
    \sum_{k \geq 1} u^k (b_k - b_{k-1}) \sqrt{1 - \frac{z}{\rho}}.
\end{equation}
The mean value corresponding to $m$\nobreakdash-openness of plain terms can be
calculated as
\begin{equation}
    \dfrac{\left.
        [z^n] \dfrac{\partial}{\partial u} L(z, u)
    \right|_{u=1}}{[z^n] L_\infty(z)}
    \sim \dfrac{\sum_{k \geq 1} k (b_k - b_{k-1})}{\sum_{k \geq 1} (b_k -
    b_{k-1})}
    = \dfrac{\sum_{k \geq 0} (b_\infty - b_k)}{b_\infty}
    = \sum_{k \geq 0} \left(
        1 - \dfrac{b_k}{b_\infty} 
    \right).
\end{equation}
In order to compute this expectation we use the approximation procedure
discussed above. Using the (aptly truncated) recurrence for the coefficients \(
a_m \) and \( b_m \)
\begin{equation}
    a_{m} = \dfrac{1}{2 \rho}\left(
        1 - \sqrt{
            1 - 4 \rho^2 \frac{1 - \rho^m}{1 - \rho} - 4 \rho^2 a_{m+1}
        }
    \right), \quad
    b_m = \dfrac{\rho b_{m+1}}{\sqrt{1 - 4 \rho^2 \frac{1 - \rho^m}{1 - \rho}
    - 4 \rho^2 a_{m+1}}}
\end{equation}
we obtain the numerical approximation for the mean value corresponding to \( m
\)\nobreakdash-openness. Numerical approximation yields an estimate
$2.01922912627$.

\subsection{Variables, abstractions, successors and redexes in closed terms}
In the following section we investigate the joint distribution of several
parameters in closed \lterms, utilising the
novel~\autoref{proposition:infinite:system}.

\begin{prop}\label{proposition:joint:closed}
Let \( \vec X_n = (X_{n (\textsf{var})}, X_{n (\textsf{red})},X_{n
    (\textsf{suc})},X_{n (\textsf{abs})}) \)  denote a vector of random
    variables denoting  the number of variables, redexes, successors and
    abstractions in a random closed \lterm~of size \( n \), respectively.  Then,
    after standardisation, the random vector \( \vec X_n \) converges in law to
    a multivariate Gaussian distribution with identical parameters as plain terms.
\end{prop}

\begin{proof}
    Let us recall that the system of equations from~\autoref{prop:joint:plain}
    associated with the four parameters that we consider is, in the general class of
    plain terms, of the form
\begin{align}
    \begin{split}
        L(z, \vec u) &= u_{(\textsf{abs})} z L(z, \vec u) + A(z, \vec u)
        , \\
        A(z, \vec u) &= \dfrac{u_{(\textsf{var})} z}{1 - u_{(\textsf{suc})} z} +
        u_{(\textsf{red})} u_{(\textsf{abs})}
        z^2 L(z, \vec u)^2 + z A(z, \vec u) L(z, \vec u)
    \end{split}
\end{align}
with \( \vec u
    =
    (
        u_{(\textsf{var})},
        u_{(\textsf{red})},
        u_{(\textsf{suc})},
        u_{(\textsf{abs})}
        ) \) corresponding to respective components of \( \vec X_n \).

In order to compose a similar, infinite system for \( m \)\nobreakdash-open
    terms, we index respective generating functions in accordance with the
    natural combinatorial interpretation of $m$\nobreakdash-openness; if an
    abstraction stands before an occurrence of $L(z, \vec u)$, its respective
    index should be increased by one. This leads us to the following system:
    \begin{align}\label{eq:advanced:marking:joint:distribution:i}
\begin{split}
        L_m(z, \vec u) &= u_{(\textsf{abs})} z L_{m+1}(z, \vec u) + A_m(z, \vec u)
        , \\
        A_m(z, \vec u) &= u_{(\textsf{var})}
        \dfrac{z(1 - (u_{(\textsf{suc})} z)^m)}{1 - u_{(\textsf{suc})} z} +
        u_{(\textsf{red})} u_{(\textsf{abs})} z^2 L_m(z, \vec u) L_{m+1}(z, \vec u) + z A_m(z, \vec u) L_m(z, \vec u)
        .
\end{split}
\end{align}
    Equivalently, we can
    represent~\eqref{eq:advanced:marking:joint:distribution:i} as
\begin{equation}
    \begin{pmatrix}
        L_m(z, \vec u) \\ A_m(z, \vec u)
    \end{pmatrix}
    = \vec{\mathcal K}_m(L_m(z, \vec u), L_{m+1}(z, \vec u), A_m(z, \vec u), z, \vec u)
    , \quad m = 0, 1, 2, \ldots
\end{equation}

It is straightforward to check that all the conditions of
    \autoref{proposition:infinite:system} are satisfied. Consequently, the
    function \( L_0(z, \vec u) \) admits a Puiseux expansion in form of
\begin{equation}
    L_0(z, \vec u) \sim
    a_0(\vec u) - b_0(\vec u) \sqrt{
        1 - \dfrac{z}{\rho(\vec u)}
    }
\end{equation}
with the same \( \rho(\vec u) \) as in~\autoref{prop:joint:plain}.  Therefore,
    the limiting distribution, after standardisation, is Gaussian with the mean
    vector and the covariance matrix completely determined by the behaviour of
    the singularity \( \rho(\vec u) \) near the point \( \vec u = \vec 1 \).
\end{proof}

\subsection{Free variables in plain terms}\label{subsec:free:variables}

\begin{prop}
    Let \( X_n \) be a random variable denoting the number of free variables in
    a random plain lambda term of size \( n \).  Then, \( X_n \) converges in law
    to a computable, discrete limiting distribution.
\end{prop}

\begin{proof}
    Consider the infinite system of functional equations \( \seq{L_m(z,
    u)}_{m=0}^{\infty} \) where $L_m(z,u)$ corresponds to the generating
    function for plain \lterms~in which each de~Bruijn index
    whose value \( k \) exceeds its unary height at least by \( m \),
    is marked.
     For example, \( L_0(z,u) \)
    corresponds to plain \lterms~with marked free variables. Note that
    \begin{align}\label{eq:advanced:marking:free:variables}
    \begin{split}
        L_0(z) &= z L_1(z) + z {L_0(z)}^2
    + uz + uz^2 + \cdots
    , \\
        L_1(z) &= z L_2(z) + z {L_1(z)}^2
    + z + uz^2 + \cdots
    , \\
    & \cdots
    , \\[-0.3cm]
        L_m(z) &= z L_{m + 1}(z) + z {L_m(z)}^2
    + z \dfrac{1 - z^m}{1 - z}
    + u z \dfrac{z^m}{1-z}
    , \\[-0.15cm]
    & \cdots
    \end{split}
\end{align}
Let us apply~\autoref{proposition:infinite:system} to this system taking
    $L_\infty(z)$ as the limiting equation. Certainly, the limiting equation
    does not depend on the marking variable \( u \).  Therefore, the singular
    point \( z^\ast \) also does not depend on \( u \).  An application
    of~\autoref{proposition:discrete:limit:laws} finishes the proof.
\end{proof}

In order to compute the mean value we set
\( L_m^\square(z) := \left.
    \dfrac{\partial}{\partial u} L_m(z, u)
    \right|_{u=1} \) and represent the respective derivative as
\begin{equation}
    L_m^\square(z) \sim c_m - d_m \sqrt{1 - \dfrac{z}{\rho}}.
\end{equation}
Based on~\eqref{eq:advanced:marking:free:variables} we note that
\begin{equation}
    L_m^\square(z) =
    \dfrac{z}{1 - 2z L_\infty(z)}
    \left(L_{m+1}^\square(z) + \dfrac{z^m}{1 - z} \right).
\end{equation}
Since \( \frac{z}{1 - 2z L_\infty(z)} \sim 1 - 2b \sqrt{1 - z/\rho} \) as
\( z \to \rho \) we establish the following recurrence relation for the
coefficients \( c_m \) and \( d_m \):
\begin{equation}
    c_m = c_{m+1} + \dfrac{\rho^m}{1 - \rho}
    \quad \text{and} \quad
    d_m = d_{m+1} + 2 b_\infty c_{m+1} + 2 b_\infty \dfrac{\rho^{m}}{1-\rho}.
\end{equation}
Once solved, this implies
\begin{equation}
    c_m = \dfrac{\rho^m}{(1 - \rho)^2}
        \quad
        \text{and}
        \quad
        d_m = \dfrac{2b_\infty \rho^m}{(1-\rho)^3}.
\end{equation}
Consequently, the mean value corresponding the number of
free variables in a random plain lambda term is equal to
\( \frac{d_0}{b_\infty} = \frac{2}{(1-\rho)^3} \doteq 5.7222625231204 \).

\subsection{Head abstractions in closed terms}\label{subsec:advanced:head:abs}
\begin{prop}
Let \( X_n \) be a random variable denoting the number of head abstractions in a
    closed \lterm~of size \( n \), chosen uniformly at random.  Then, \( X_n \)
    converges in law to a computable, discrete limiting distribution.  The
    corresponding expectation is close to \( 1.447 \).
\end{prop}
\begin{proof}
    Let \( L_m(z, u) \) be the bivariate generating function associated with \(
    m \)\nobreakdash-open lambda where \( u \) marks head abstractions. Then,
    the system \( \seq{L_m(z,u)}_{m=0}^{\infty} \) satisfies
    \begin{align}\label{eq:advanced:marking:head:abstractions}
        \begin{split}
            L_0(z,u) &= zu L_1(z, u) + z {L_0(z,1)}^2
            , \\
            L_1(z,u) &= zu L_2(z, u) + z {L_1(z,1)}^2
            + z
            , \\
            & \ldots
            \\[-0.3cm]
            L_m(z,u) &= zu L_{m+1}(z, u) + z {L_m(z,1)}^2
            + z \dfrac{1 - z^m}{1 - z}
            , \\[-.15cm]
            & \ldots
        \end{split}
    \end{align}
If a \lterm~starts with a head abstraction, then after its removal, the openness
    of the respective subterm increases by one. Consequently, we include the expression
    \( uz L_{m+1} (z, u) \) in the equation for $L_m(z,u)$.  On the other hand,
    if the \lterm~does not start with a head abstraction, i.e. starts with an
    application or is itself a de~Bruijn index, we do not mark remaining
    abstractions as they are no longer head abstractions. Hence, we also include
    expressions \( z L_m^2(z, 1) \) and \( z \frac{1 - z^m}{1 - z} \) in the
    equation corresponding to $L_m(z,u)$.

    Having established the system~\eqref{eq:advanced:marking:head:abstractions} we
    note that \( L_0(z, u) \) can be obtained as a limit of the solutions of
    truncated systems (see the description of the approximation procedure
    in~\autoref{proposition:infinite:system}) and this limit is equal to the sum
    \begin{align}
        \begin{split}
            L_0(z,u) &= z {L_0(z,1)}^2 + z u \left( z u L_2(z,u) + z
            {L_1(z,1)}^2 + z \right)\\
            &= z {L_0(z,1)}^2 + z^2 u {L_1(z,1)}^2 + z^2 u + {(z u)}^2 \left(
            z u L_3(z,u) + z {L_2(z,1)}^2 + z + z^2 \right)\\
            &= \ldots
        \end{split}
    \end{align}
and so
    \begin{align}\label{eq:advanced:marking:head:abstractions:ii}
    \begin{split}
        L_0(z, u) &= z \sum_{m \geq 0} {(u z)}^m {L_m(z, 1)}^2
    + \sum_{m \geq 1} u^m (z^{m+1} + \cdots + z^{2m})
    \\
        &= z \sum_{m \geq 0} {(u z)}^m {L_m(z, 1)}^2
        + \sum_{m \geq 1} {(u z)}^m z \dfrac{1 - z^m}{1 - z}
    \\
        &= z \sum_{m \geq 0} {(u z)}^m {L_m(z, 1)}^2
    + \dfrac{u z^2 }{(1-z)(1-uz)}
    - \dfrac{u z^3}{(1-z)(1-uz^2)}.
    \end{split}
\end{align}
    Denote the final sum in~\eqref{eq:advanced:marking:head:abstractions:ii} as
    $S(z,u)$.
    Since for each \( m \) the function \( L_m(z, 1) \) admits a Puiseux series
    expansion \( L_m(z,1) \sim a_m - b_m \sqrt{1 - z/\rho} \), near $z = \rho$
    it holds
    \begin{equation}
        S(z,u)
        \sim
        c(u) +
        z \sum_{m \geq 0}
            {(u z)}^m \left(
                a_m^2 - 2 a_m b_m \sqrt{1 - \dfrac{z}{\rho}}
            \right)
    \end{equation}
    where \( c(u) \) comes from the last two summands of the previous
    expression. Since \( a_m^2 \leq a_\infty^2 \) and \( a_m b_m \leq a_\infty
    b_\infty \),  $S(z,u)$ is convergent near
    $(z,u) = (\rho,1)$ and the
    function $L_0(z,u)$ admits a Puiseux series expansion in form of
\begin{equation}
    L_0(z, u) \sim a_0(u) - b_0(u) \sqrt{1 - \dfrac{z}{\rho}}.
\end{equation}
Consequently, \( p(u) = b_0(u) / b_0(1) \) is the
limiting probability generating function corresponding to the number of head
abstractions in closed \lterms.  The function $b_0(u)$ satisfies
\begin{equation}
    b_0(u) = 2 \rho \sum_{m \geq 0} {(u \rho)}^m a_m b_m .
\end{equation}
\end{proof}

\subsection{De Bruijn index values in closed lambda terms}\label{subsec:debruijn:closed}
\begin{prop}
    Let \( X_n \) be a random variable denoting the de~Bruijn index value $m$ of
    a random index $\idx{m}$ in a random closed \lterm~of size $n$. Then, $X_n$
    converges in law to a geometric distribution $\Geom(\rho)$ with parameter
    $\rho$. Specifically,
    \begin{equation}\label{eq:closed:terms:head:abstractions:distribution}
        \mathbb P(X_n = h) \xrightarrow[n\to\infty]{} \mathbb P(\Geom(\rho) =
        h) = (1 - \rho) \rho^h.
    \end{equation}
\end{prop}
\begin{proof}
    Let $L_{m,k}(z,u)$ denote the generating function for $m$\nobreakdash-open
    \lterms~with $u$ marking the number of occurrences of de~Bruijn index
    $\idx{k}$. Note that $L_{m,k}(z)$ satisfies a functional equation
    \begin{equation}\label{eq:advanced:marking:debruijn}
        L_{m,k}(z,u) = z L_{m+1,k}(z,u) + z {L_{m,k}(z,u)}^2 + z \dfrac{1 - z^m}{1 - z}
        + (u-1) z^{k+1} \boldsymbol 1_{[k<m]}
    \end{equation}
    where $\boldsymbol 1_{[\cdot]}$ stands for the Iverson bracket notation.

    Taking the partial derivative of~\eqref{eq:advanced:marking:debruijn} with
    respect to $u$ and assigning $u=1$, we obtain the generating function
    corresponding to \lterms~weighted by the number of occurrences of de~Bruijn
    index $\idx{k}$. Denote \({\deriv{u}{L_{m,k}(z,u)}}\at{u=1}\) as \(
    L_{m,k}^\square(z) \). Then, taking into account that $L_{m,k}(z,1) =
    L_m(z)$ we arrive at
    \begin{equation}
        \label{eq:debruijn:closed:lowlevel}
        L_{m,k}^\square(z) = z L_{m+1,k}^\square(z) + 2z L_{m}(z)
        L_{m,k}^\square(z) +
        z^{k+1} \boldsymbol 1_{[k < m]}.
    \end{equation}
    Consider the generating function
    \begin{equation}
        E_m(z, w) = \sum_{k \geq 0} L_{m,k}^\square(z) w^k.
    \end{equation}
    Note that $[z^n]E_m(z, w)$ denotes the probability generating function
    associated with the distribution of variables in $m$\nobreakdash-open
    \lterms~(cf.~\autoref{proposition:dbindices:plain}).  Consequently,
    summing~\eqref{eq:debruijn:closed:lowlevel} over $k$ we obtain
    \begin{equation}
        E_m(z, w) = z E_{m+1}(z, w) + 2z L_m(z) E_m (z, w) + z \dfrac{1 -
        (wz)^m}{1 - wz}.
    \end{equation}
    These equations generate an infinite system for
    which~\autoref{proposition:infinite:system} with a small modification is
    applicable.  Each of the equations of the infinite system is linear, and the
    generating functions \( L_m(z) \) enter the equations as coefficients. This
    yields the desired behaviour of the Puiseux expansions, because the
    non-linearity of the components is used only to provide the Puiseux
    expansion of the limiting equation, which is given in our case by
    construction.

    The condition of exponential convergence holds because the difference
    between the limiting system and the $m$th equation of the system is equal to
    \begin{equation}
        2z E_\infty(z, w) \left(L_\infty(z) - L_m(z)\right) + \dfrac{z}{1 - wz} {(wz)}^m
    \end{equation}
    and decreases at exponential speed. Hence, the limiting distribution of
    de~Bruijn index value is identical to the respective parameter in plain \lterms.
\end{proof}

\subsection{Leftmost-outermost redex search time in closed terms}\label{subsec:advanced:redex:discovery}
\begin{prop}
    Let \( X_n \) denote the number of vertices visited by depth-first traversal
    algorithm searching for the leftmost-outermost $`b$\nobreakdash-redex in a
    random closed \lterm~of size $n$
    (see~\autoref{subsec:basic:redex:discovery}). Then, the random variable \(
    X_n \) converges in law to a computable, discrete limiting distribution.
\end{prop}

\begin{proof}
    Recall that the system~\eqref{eq:plain:redex:spec} defining the generating function \(
    L_\infty(z, u) \) corresponding to plain terms with \( u \) marking visited
    nodes is written as
    \begin{align}\label{eq:advanced:marking:leftmost:redex}
    \begin{split}
        L_\infty(z, u) &= u z L_\infty(z, u) + A(z, u)
        , \\
        A(z, u) &= u \dfrac{z}{1 - z} + z^2 u^2 {L_\infty(z, 1)}^2
        + z u M(z u) L_\infty(z, u) + uz (A(z, u) - M(zu)) L_\infty(z, 1)
    \end{split}
\end{align}
with \( M(z) \) being the generating function associated with so-called neutral
    terms and also Motzkin numbers, see~\autoref{rem:motzkin:numbers:neutral:terms}:
\begin{equation}
    M(z) = \dfrac{
        1 - z - \sqrt{(1+z)(1 - 3z)}
    }{
        2z
    }.
\end{equation}

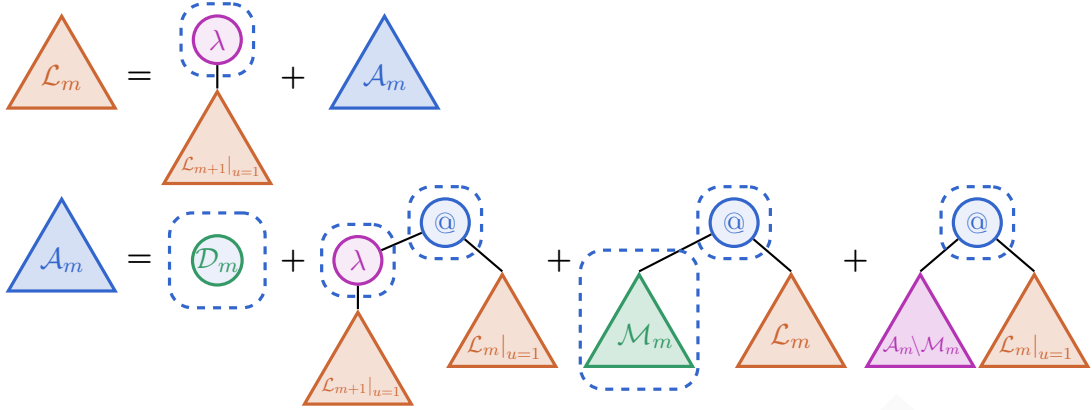
\begin{figure}[hbt]
\centering
\begin{tikzpicture}[>=stealth',level/.style={sibling distance = 1.5cm/#1,
  level distance = 1.5cm, thick}]
\matrix
{
\draw
node[triangle]{\( \mathcal L_m \)};
&
\draw
node{\(\boldsymbol =\)};
&
\draw
--(0,0.5)
node(abstraction1)[arn_r]{\( \lambda\)}
    child{
        node[triangle, text depth = -1.6ex](rt)
        {\scalebox{0.65}{\hspace{-2.7ex}{
            \(
                \left.
                    \mathcal L_{m+1}
                \right|_{u=1}
            \)
        }}}
    }
++(1,-0.5)
node{\( \boldsymbol + \)}
++(1.2, 0)
node[triangle_b]{ \( \mathcal A_m \) }
;
\node[rectangle,dashed,draw,fit=(abstraction1),
rounded corners=3mm,inner sep=4pt, bblue , very thick] {};
\\
\draw
node[triangle_b]{\( \mathcal A_m \)};
&
\draw
node{\(\boldsymbol =\)};
&
\draw
node(var)[arn_g]{\( \mathcal D_m \) }
++(1,0)
node{\( \boldsymbol + \)}
++(2.0,0.5)
node[arn_n](app){\( @ \)}
    child[level distance=0.5cm, sibling distance=2.3cm]{
        node(abstraction2)[arn_r]{ \( \lambda \) }
        child{
            node[triangle, text depth = -1.6ex](rt)
            {\scalebox{0.65}{\hspace{-2.7ex}{
                \(
                    \left.
                        \mathcal L_{m+1}
                    \right|_{u=1}
                \)
            }}}
        }
    }
    child[level distance=1.5cm, child anchor=north]{
        node[triangle, text depth = -1.2ex](rt)
        {\scalebox{0.8}{\hspace{-2.3ex}{
            \(
                \left.
                    \mathcal L_m
                \right|_{u=1}
            \)
        }}}
    }
++(1.5, -0.5)
node{\( \boldsymbol + \)}
++(2.3, +0.5)
node[arn_n](app2){ \( @ \) }
    child[child anchor=north, sibling distance = 2.5cm]{
        node(nf1)[triangle_g]{ \( \mathcal M_m \) }
    }
    child[child anchor=north]{
        node[triangle]{ \( \mathcal L_m \) }
    }
++(1.6, -0.5)
node{\( \boldsymbol + \)}
++(1.6, +0.5)
node[arn_n](app3){ \( @ \) }
    child[child anchor=north]{
        node[triangle_v, text depth = -1.2ex]{
            \scalebox{0.7}{
                \hspace{-.55cm}
            \(
         \mathcal A_m \! \backslash \mathcal M_m
             \)} }
    }
    child[child anchor=north]{
        node[triangle, text depth = -1.2ex](rt)
        {\scalebox{0.8}{\hspace{-2.3ex}{
            \(
                \left.
                    \mathcal L_m
                \right|_{u=1}
            \)
        }}}
    }
;
\node[rectangle,dashed,draw,fit=(app),
rounded corners=3mm,inner sep=4pt, bblue, very thick] {};
\node[rectangle,dashed,draw,fit=(app2),
rounded corners=3mm,inner sep=4pt, bblue, very thick] {};
\node[rectangle,dashed,draw,fit=(app3),
rounded corners=3mm,inner sep=4pt, bblue, very thick] {};
\node[rectangle,dashed,draw,fit=(nf1),
rounded corners=3mm,inner sep=8pt, bblue, very thick] {};
\node[rectangle,dashed,draw,fit=(var),
rounded corners=3mm,inner sep=8pt, bblue, very thick] {};
\node[rectangle,dashed,draw,fit=(abstraction2),
rounded corners=3mm,inner sep=4pt, bblue , very thick] {};
\\
};
\end{tikzpicture}
\caption{\label{fig:search:time:closed}Specification corresponding to redex
search time in closed lambda terms.}
\end{figure}

Note that including indices in~\eqref{eq:advanced:marking:leftmost:redex}
according to \( m \)\nobreakdash-openness we obtain
\begin{align}
    \label{eq:redex:search:closed}
    \begin{split}
        L_m(z, u) &= uz L_{m+1}(z, u) + A_m(z, u)
        , \\
        A_m(z, u) &= uz \dfrac{1 - z^m}{1 - z} + z^2 u^2 L_m(z, 1) L_{m+1}(z, 1)
        \\&
        + z u M_m(zu) L_m(z, u) + uz (A_m(z, u) - M_m(zu)) L_m(z, 1)
    \end{split}
\end{align}
where \( M_m(z) \) is the generating function for \( m \)-open neutral lambda
terms. The sequence of functions \( (M_m(z))_{m=0}^\infty \) can be obtained
from the system of equations
\begin{align}\label{eq:advanced:marking:leftmost:redex:ii}
    \begin{split}
        N_m(z) &= z N_{m+1}(z) + M_m(z) , \\
        M_m(z) &= z M_m(z) N_m(z) + z \dfrac{1 - z^m}{1 - z} .
    \end{split}
\end{align}
Comparing~\eqref{eq:advanced:marking:leftmost:redex:ii} with its limiting
counterpart~\eqref{eq:normal:forms:neutral:terms} we note that the \(
M_\infty(z) - M_m(z) \) decays at exponential speed as \( m \to \infty \) by
virtue of~\autoref{proposition:infinite:system}.  As additionally follows from
the theorem, the functions \( (M_m(z))_{m=0}^\infty \) share the same singularity
\( \rho = 1/3 \).

Next, the system of
equations~\eqref{eq:advanced:marking:leftmost:redex:ii} can be represented in
the form
\begin{equation}
    \label{eq:advanced:marking:leftmost:redex:iii}
\begin{pmatrix}
    L_m(z,u) \\ A_m(z,u)
\end{pmatrix}
=
    \vec{\mathcal K}_m (L_m(z,u), L_{m+1}(z,u), A_m(z,u), z, u).
\end{equation}
In order to apply~\autoref{proposition:infinite:system}
to~\eqref{eq:advanced:marking:leftmost:redex:iii} we need to replace the condition that the
limiting system satisfies the premises of Drmota--Lalley--Woods theorem by an
assumption that the limiting system admits Puiseux expansion.
It was proven in~\autoref{proposition:redex:discovery:plain} that
\( L_\infty(z, u) \) has a fixed singularity \( z^\ast \) which is independent
of \( u \).
Therefore, all the functions \(
L_m(z, u) \) have a fixed singularity \( z^\ast \) and by
applying~\autoref{proposition:discrete:limit:laws}, we obtain that the limiting
distribution of the redex search time is discrete.
\end{proof}

\subsection{Node height profile in closed terms}
Like in~\autoref{subsection:height:profile:plain}, in the current section we
consider unary and natural height profile of variables, abstractions and
applications in closed lambda terms.  For this purpose, we provide a variation
of the semi-large powers theorem
(see~\autoref{proposition:semilarge:power:theorem}).
\begin{theorem}
    \label{theorem:semilarge:powers:variation}
    Let \( (f_k(\rho z))_{k \geq 0} \) be a sequence of functions
    analytic in delta-domain \( \Delta(R) \)
    (see~\autoref{proposition:transfer:theorem}) for some \( R > \rho \)
    admitting Puiseux series expansions in form of
    \begin{equation}
        f_k(z) \sim \sigma_k - a_k \sqrt{1 - \dfrac{z}{\rho}}
    \end{equation}
    as \( z \to \rho \).
    Assume there exist \( \beta \) and \( \widehat \sigma \) such that
    the sequences \( (\sigma_k)_{k \geq 0} \) and \( (a_k)_{k \geq 0} \)
    satisfy
        \begin{equation}
            \sum_{j=0}^k \dfrac{a_j}{\sigma_j} \sim \beta k
            \quad \text{and} \quad
            \lim_{k \to \infty} \prod_{j=0}^k \sigma_k \to \widehat\sigma.
        \end{equation}
        Then, for $x$ in any compact subinterval of \( (0, +\infty) \), as \( n
        \to \infty \), it holds
    \begin{equation}
        [z^n] \prod_{j=0}^k f_j(z) \sim \widehat \sigma
        \dfrac{\rho^{-n}}{n} S(\beta x)
        \quad \text{and} \quad
        x = \dfrac{k}{\sqrt n}
    \end{equation}
    where \( S(x) \) is the Rayleigh function defined
    in~\autoref{proposition:semilarge:power:theorem}.
\end{theorem}
\begin{proof}
    We recall that in the course of the proof of the semi-large power theorem,
    see~\cite[Theorem IX.16]{flajolet09}, the coefficient \( [z^n] {f(z)}^k \) is
    expressed as the following complex contour integral with the help of
    Cauchy's integral theorem:
    \begin{equation}
        [z^n] {f(z)}^k = \dfrac{1}{2 \pi i} \oint {f(z)}^k \dfrac{dz}{z^{n+1}}
    = \dfrac{1}{2 \pi i} \oint e^{h_{n,k}(z)} \dfrac{dz}{z},
    \quad
    h_{n,k}(z) = k \log f(z) - n \log z.
    \end{equation}

    With the change of variables $z = \rho(1 - t/n)$ the coefficient
    $[z^n]{f(z)}^k$ can be accordingly approximated by the following real
    integral:
\begin{equation}
\label{eq:real:integral}
    [z^n] {f(z)}^k
    \sim
    - \dfrac{\rho^n}{n}
    \dfrac{1}{2 \pi i}
    \int_{0}^\infty e^{t - a x \sqrt{t}} dt.
\end{equation}
As proven in the referenced literature, this yields the Rayleigh approximation.
In the statement of the current theorem, the function \( h_{n,k}(z) \), i.e.~the
logarithm of the sub-integral expression, is replaced by
    \begin{equation}
    \widetilde h_{n,k} = \sum_{j=0}^{k} \log f_j(z) - n \log z.
    \end{equation}
Accordingly, with the variable change \( z = \rho(1 - t/n) \) the coefficient
\( [z^n] \prod_{j=0}^k f_j(z) \) becomes
    \begin{equation}
    [z^n] \prod_{j=0}^k f_j(z)
    =
    \prod_{j=0}^k \sigma_k
    \cdot
    [z^n] \prod_{j=0}^k
    \left(
        1 - \dfrac{a_k}{\sigma_k} \sqrt{1 - \dfrac z \rho}
    \right)
    \sim - \dfrac{\rho^n \prod_{j=0}^k \sigma_k}{n}
    \dfrac{1}{2 \pi i} \int_0^\infty e^{t - \beta x \sqrt t} dt
    \end{equation}
which has the same form as~\eqref{eq:real:integral}, finishing the proof.
\end{proof}

\begin{prop}
    Let \( H_n \) be a random variable denoting the unary (respectively natural)
    height of a uniformly random  variable in a random closed lambda term.
    Then, with \( x \) in any compact subinterval of \( (0, +\infty) \), \( H_n
    \) follows the Rayleigh limiting distribution
    \begin{equation}
        \mathbb P(H_n = k) \sim
        \dfrac{C}{\sqrt n} \cdot \dfrac{x}{2} e^{-x^2/4},
        \quad \text{where} \quad
        x = \dfrac{k}{\sqrt n} \cdot C
    \end{equation}
with \( C \doteq 4.30187 \) for unary height and \( C \doteq 1.27162 \) for
    the natural height.
\end{prop}
\begin{proof}
    Let \( C_{m,k}(z, u) \) denote the bivariate generating function
    corresponding to \( m \)\nobreakdash-open \lterms~where variable \( u \)
    marks de~Bruijn indices at unary height \( k-m \). Certainly, $C_{m,k}(z,1)
    = L_m(z)$ for each $m$ and $k$.  Note that, the functions \(
    \seq{C_{m,k}(z,u)}_{m=0}^\infty \) satisfy jointly
    \begin{equation}
    \begin{cases}
        C_{m,k}(z,u) = z\dfrac{1 - z^{m}}{1 - z}
        + z C_{m+1, k}(z, u) + z C_{m,k}(z, u)^2
        & \text{ if } m < k,\\[0.3cm]
        C_{m,k}(z,u) = zu\dfrac{1 - z^{m}}{1 - z}
        + z L_{m+1}(z) + z C_{m,k}(z, u)^2
        & \text{ if } m = k,\\[0.3cm]
        C_{m,k}(z,u) = L_m(z)
        & \text{ if } m > k.
    \end{cases}
    \end{equation}
A straightforward induction yields
\begin{equation}
    \left.
        \dfrac{\partial}{\partial u}
        C_{0,k}(z, u)
    \right|_{u=1}
    =
    \prod_{j=0}^k \dfrac{z}{1 - 2z L_j(z)}
    \cdot
    \dfrac{1 - z^k}{1 - z}.
\end{equation}
This function is amenable to asymptotic analysis of their coefficients
by~\autoref{theorem:semilarge:powers:variation}.
First show that in the respective Puiseux expansions of the functions
\[
    \frac{z}{1 - 2z L_j(z)} \sim \sigma_j - c_j \sqrt{1 - z/\rho}
\]
the sequence \( \sigma_j \) tends to \( 1 \) at exponential speed, and the
sequence \( c_j \) tends to a limit \( 2 b_\infty \doteq 4.30187 \) again at exponential
speed. This holds because in the course of the proof
of~\autoref{proposition:infinite:system} we have shown that the sequences of
coefficients of the Puiseux expansion of \( (L_j(z))_{j=0}^\infty \)
(respectively, the the sequence of first coefficients \( L_j(\rho) \), and the
sequence of the second coefficients) tend to their respective limits, i.e.~to
the coefficients of the Puiseux expansion of \( L_\infty(z) \) exponentially
fast. Comparing with the Puiseux expansion of \( \frac{z}{1 - 2z L_\infty(z)} \)
given in the proof of~\autoref{proposition:hprofile:vars:plain} we obtain the
limiting values of the sequences \( (\sigma_j)_{j=0}^\infty \) and \(
(c_j)_{j=0}^\infty \). Since the speed of convergence is exponential, the
product \( \prod_{j=0}^k \sigma_j \) converges to some \( \widehat \sigma \),
and the sum of the ratios \( c_j / \sigma_j \) tends to a linear function
\( \beta k = 2 b_\infty k \).

Note that up to a normalising constant, the height profile of other
parameters, namely the height profile distribution of abstractions and applications,
remains asymptotically the same because from the generating function viewpoint
only the multiple in front of the product
\( \prod_{j=0}^k f_j(z) \)
changes (see~\autoref{proposition:hprofile:vars:plain}).

In the same manner, there can be obtained Rayleigh distribution for natural
height profile of different parameters. For example, in the case of variable
height profile, we obtain the system of equations for the family of generating
functions \( C_{m,k}(z, u) \) for \( m \)-open lambda terms with variable \( u
\) marking de Bruijn indices at unary height \( k-m\):
\[
    \begin{cases}
        C_{m,k}(z, u) = z \dfrac{1 - z^m}{1 - z}
        + z C_{m+1,k}(z, u) + z C_{m+1,k}(z, u)^2,
        & 0 \leq m < k;\\[0.3cm]
        C_{m,k}(z, u) = u z \dfrac{1 - z^m}{1 - z}
        + z L_{m+1}(z) + z L_m(z)^2,
        & m = k;\\[0.3cm]
        C_{m,k}(z, u) = L_m(z),
        & m > k.
    \end{cases}
\]
This implies
\[
    \left.
        \dfrac{\partial}{\partial u}
        C_{0,k}(z, u)
    \right|_{u=1}
    =
    \prod_{j=1}^k (z + 2z L_j(z)) \cdot z\dfrac{1 - z^k}{1 - z}.
\]
Using the same argument as in the previous case, and taking into account
two first terms of Puiseux expansion of \( (z + 2z L_\infty(z)) \)
(see proof of~\autoref{proposition:hprofile:vars:plain}), we obtain again
Rayleigh distribution, with the same parameter as for plain lambda terms.
\end{proof}

%% file: 7-conclusions.tex
\section{Conclusions}\label{sec:conclusions}
We investigated the statistical properties of \lterms~in the de~Bruijn notation,
providing some insight into their internal, quantitative characteristics. In
essence, our results suggest that random \lterms, both plain and closed,
exhibit typical traits of various tree-like structures. For instance, the distribution
of sub-patterns inside random \lterms~is typically Gaussian whereas their
corresponding finding time tends to a discrete limit distribution. Similarly, the
height profile of random terms follows the Rayleigh distribution.

Nonetheless, some of the investigated parameters do not have analogues in other
tree-like structures such as, for instance, $m$\nobreakdash-openness or the number of
free variables. In both cases we have established a discrete limiting
distribution.  Remarkably, we have not discovered substantially different
statistical traits of plain and closed lambda terms;  however, we found that
among the statistics with discrete limiting distributions, the distribution in
closed terms is often different from the associated distribution in plain terms.

Given the general algorithmic frameworks meant for the construction of effective
exact- and approximate-size combinatorial samplers, such as Boltzmann
samplers~\cite{Duchon04boltzmannsamplers} and the recursive
method~\cite{NijenhuisWilf1978,FLAJOLET19941}, presented parameter
specifications provide a novel source of effective sampling methods for
\lterms~with additional control over their specific combinatorial parameters. In
this context, most parameter specifications associated with plain terms are
finite and hence also readily applicable.  Remaining, infinite specifications
are a bit more involved.  Nonetheless, an appropriate truncation of the
specification followed by a final rejection phase allows to discard inadmissible
terms. The exponential convergence of intermediate, truncated specifications
rationalises such an approach and provides effective samplers for corresponding
\lterms.  Let us also remark that so generated terms do not have to be
restricted to their natural parameter distributions. It is possible to gain an
additional control over the expected parameter values using a dedicated tuning
procedure which distorts the intrinsic parameter distribution and hence allows
for a skewed parameter distribution sampler
construction~\cite{doi:10.1137/1.9781611975062.9}.  Consequently, the presented
analysis provides means for an effective construction of various samplers for
(plain or closed) \lterms~with additional control over their parameter
distribution.

Few more aspects of the parameter analysis of \lterms~remain untouched. For
instance, our empirical data suggests that the distribution of binding
abstractions, both in plain and closed terms, is Gaussian. The same holds for
the number of open subterms. Alas, the theoretical verification of our empirical
findings is left open. Moreover, we have not investigated other, well-known parameters. Let us mention,
for instance, the height distribution of random closed terms, or the
distribution of certain extremal statistics, such as the maximal de~Bruijn
index, longest lambda run, or the maximal number of variables bound to a single
abstraction.  We conjecture that the behaviour of these parameters in closed
terms does not substantially differ from the behaviour of respective parameters
in plain terms.  Finally, the question of generalised $m$\nobreakdash-openness
also has not been settled and the corresponding techniques are still to be
developed.

Arguably, from the viewpoint of analytic combinatorics, our novel result
complements the existing result of Drmota, Gittenberger and
Morgenbesser~\cite{infinitesystems} on infinite systems. In our formulation, the
infinite system is not required to be strongly connected. Consequently, we
conjecture that the properties of the Jacobian operator of the infinite system
are not sufficient to deduce the result, in contrast with the mentioned paper.
In other words, we conjecture that the condition of exponential convergence is
essential. Moreover, we discovered that it is possible to rewrite the infinite
system defining closed \lterms~as a strongly connected system, however the
Jacobian of the resulting system is not compact. Alas, the
framework~\cite{infinitesystems} is not applicable.

Consequently, we finish the paper with an even more general question: what can
be stated about the properties of infinite systems which are either not strongly
connected or have a non-compact Jacobian operator?

%% file: lambda-statistics.bbl
\begin{thebibliography}{10}

\bibitem{banderier2001random}
Cyril Banderier, Philippe Flajolet, Gilles Schaeffer, and Michele Soria.
\newblock Random maps, coalescing saddles, singularity analysis, and airy
  phenomena.
\newblock {\em Random Structures \& Algorithms}, 19(3-4):194--246, 2001.

\bibitem{barendregt1984}
Henk~P. Barendregt.
\newblock {\em The {L}ambda {C}alculus: Its Syntax and Semantics}, volume 103.
\newblock {N}orth {H}olland, revised edition, 1984.

\bibitem{bell2010characteristic}
Jason~P. Bell, Stanley~N. Burris, and Karen~A. Yeats.
\newblock Characteristic points of recursive systems.
\newblock {\em The Electronic Journal of Combinatorics}, 17(1):121, 2010.

\bibitem{BenderCLT}
Edward~A. Bender and L.~Bruce Richmond.
\newblock Central and local limit theorems applied to asymptotic enumeration
  {II: M}ultivariate generating functions.
\newblock {\em Journal of Combinatorial Theory, Series A}, 34(3):255 -- 265,
  1983.

\bibitem{BendkowskiThesis}
Maciej Bendkowski.
\newblock {\em Quantitative aspects and generation of random lambda and
  combinatory logic terms}.
\newblock PhD thesis, Jagiellonian University, Kraków, Poland, 5 2017.

\bibitem{doi:10.1137/1.9781611975062.9}
Maciej Bendkowski, Olivier Bodini, and Sergey Dovgal.
\newblock {\em Polynomial tuning of multiparametric combinatorial samplers},
  pages 92--106.
\newblock 2018.

\bibitem{Bendkowski2016}
Maciej Bendkowski, Katarzyna Grygiel, Pierre Lescanne, and Marek Zaionc.
\newblock A natural counting of lambda terms.
\newblock In Rusins~Martins Freivalds, Gregor Engels, and Barbara Catania,
  editors, {\em {SOFSEM} 2016: Theory and Practice of Computer Science - 42nd
  International Conference on Current Trends in Theory and Practice of Computer
  Science, Harrachov, Czech Republic, January 23-28, 2016, Proceedings}, volume
  9587 of {\em Lecture Notes in Computer Science}, pages 183--194. Springer,
  2016.

\bibitem{BendkowskiGLZ16}
Maciej Bendkowski, Katarzyna Grygiel, Pierre Lescanne, and Marek Zaionc.
\newblock Combinatorics of {$\lambda$}-terms: a natural approach.
\newblock {\em Journal of Logic and Computation}, 27(8):2611--2630, 2017.

\bibitem{Bertot:2010:ITP:1965123}
Yves Bertot and Pierre Castran.
\newblock {\em Interactive Theorem Proving and Program Development: Coq'Art The
  Calculus of Inductive Constructions}.
\newblock Springer Publishing Company, Incorporated, 1st edition, 2010.

\bibitem{BGGJ2013}
Olivier Bodini, Dani\`ele Gardy, Bernhard Gittenberger, and Alice Jacquot.
\newblock Enumeration of generalized bci lambda-terms.
\newblock {\em Electronic Journal of Combinatorics}, 20(4), 2013.

\bibitem{BODINI2013227}
Olivier Bodini, Dani\`ele Gardy, and Alice Jacquot.
\newblock Asymptotics and random sampling for bci and bck lambda terms.
\newblock {\em Theoretical Computer Science}, 502:227 -- 238, 2013.
\newblock Generation of Combinatorial Structures.

\bibitem{doi:10.1137/1.9781611973013.3}
Olivier Bodini, Danièle Gardy, and Bernhard Gittenberger.
\newblock {\em Lambda terms of bounded unary height}, pages 23--32.
\newblock 2011.

\bibitem{BodGenRo2015}
Olivier Bodini, Antoine Genitrini, and Nicolas Rolin.
\newblock Pointed versus singular boltzmann samplers: a comparative analysis.
\newblock {\em Pure Mathematics and Application}, 25(2):115--131, 2015.

\bibitem{doi:10.1137/1.9781611973204.3}
Olivier Bodini and Bernhard Gittenberger.
\newblock {\em On the asymptotic number of BCK(2)-terms}, pages 25--39.
\newblock 2014.

\bibitem{BodiniGitGol17}
Olivier Bodini, Bernhard Gittenberger, and Zbigniew Gołębiewski.
\newblock Enumerating lambda terms by weighted length of their de bruijn
  representation.
\newblock {\em CoRR}, abs/1707.02101, 2017.

\bibitem{bodinitarau2017}
Olivier Bodini and Paul Tarau.
\newblock On uniquely closable and uniquely typable skeletons of lambda terms.
\newblock {\em CoRR}, abs/1709.04302, 2017.

\bibitem{DBLP:series/hhl/CardoneH09}
Felice Cardone and J.~Roger Hindley.
\newblock Lambda-calculus and combinators in the 20th century.
\newblock In {\em Logic from Russell to Church}, volume~5 of {\em Handbook of
  the History of Logic}, pages 723--817. Elsevier, 2009.

\bibitem{Claessen-2000}
Koen Claessen and John Hughes.
\newblock Quickcheck: A lightweight tool for random testing of haskell
  programs.
\newblock In {\em Proceedings of the Fifth ACM SIGPLAN International Conference
  on Functional Programming}, pages 268--279. ACM, 2000.

\bibitem{dgkrtz}
René David, Katarzyna Grygiel, Jakub Kozik, Christophe Raffalli, Guillaume
  Theyssier, and Marek Zaionc.
\newblock Asymptotically almost all $\lambda$-terms are strongly normalizing.
\newblock {\em Logical Methods in Computer Science}, 9:1--30, 2013.

\bibitem{deBruijn1972}
Nicolaas~G. de~Bruijn.
\newblock {L}ambda calculus notation with nameless dummies, a tool for
  automatic formula manipulation, with application to the {C}hurch-{R}osser
  theorem.
\newblock {\em Indagationes Mathematicae (Proceedings)}, 75(5):381--392, 1972.

\bibitem{drmota1997systems}
Michael Drmota.
\newblock Systems of functional equations.
\newblock {\em Random Structures and Algorithms}, 10(1-2):103--124, 1997.

\bibitem{drmota2009random}
Michael Drmota.
\newblock {\em Random Trees: An Interplay between Combinatorics and
  Probability}.
\newblock Springer Science \& Business Media, 2009.

\bibitem{infinitesystems}
Michael Drmota, Bernhard Gittenberger, and Johannes~F. Morgenbesser.
\newblock {Infinite Systems of Functional Equations and Gaussian Limiting
  Distributions}.
\newblock In {\em {23rd International Meeting on Probabilistic, Combinatorial,
  and Asymptotic Methods in the Analysis of Algorithms (AofA'12)}}, volume
  DMTCS Proceedings vol. AQ, 23rd Intern. Meeting on Probabilistic,
  Combinatorial, and Asymptotic Methods for the Analysis of Algorithms
  (AofA'12) of {\em DMTCS Proceedings}, pages 453--478, Montreal, Canada, 2012.
  {Discrete Mathematics and Theoretical Computer Science}.

\bibitem{Duchon04boltzmannsamplers}
Philippe Duchon, Philippe Flajolet, Guy Louchard, and Gilles Schaeffer.
\newblock Boltzmann samplers for the random generation of combinatorial
  structures.
\newblock {\em Combinatorics, {P}robability and {Computing}}, 13:2004, 2004.

\bibitem{flajolet09}
Philippe Flajolet and Robert Sedgewick.
\newblock {\em Analytic Combinatorics}.
\newblock Cambridge University Press, 1 edition, 2009.

\bibitem{FLAJOLET19941}
Philippe Flajolet, Paul Zimmermann, and Bernard~Van Cutsem.
\newblock A calculus for the random generation of labelled combinatorial
  structures.
\newblock {\em Theoretical Computer Science}, 132(1):1--35, 1994.

\bibitem{gittenberger_et_al:LIPIcs:2016:5741}
Bernhard Gittenberger and Zbigniew Gołębiewski.
\newblock {On the Number of Lambda Terms With Prescribed Size of Their De
  Bruijn Representation}.
\newblock In Nicolas Ollinger and Heribert Vollmer, editors, {\em 33rd
  Symposium on Theoretical Aspects of Computer Science (STACS 2016)}, volume~47
  of {\em Leibniz International Proceedings in Informatics (LIPIcs)}, pages
  40:1--40:13. Schloss Dagstuhl--Leibniz-Zentrum fuer Informatik, 2016.

\bibitem{grygiel_lescanne_2013}
Katarzyna Grygiel and Pierre Lescanne.
\newblock Counting and generating lambda terms.
\newblock {\em Journal of Functional Programming}, 23(5):594–628, 2013.

\bibitem{grygiel_lescanne_2015}
Katarzyna Grygiel and Pierre Lescanne.
\newblock Counting and generating terms in the binary lambda calculus.
\newblock {\em Journal of Functional Programming}, 25:e24, 2015.

\bibitem{heuberger2007hwang}
Clemens Heuberger.
\newblock Hwang's quasi-power-theorem in dimension two.
\newblock {\em Quaestiones Mathematicae}, 30(4):507--512, 2007.

\bibitem{higherdimensional2016quasipower}
Clemens Heuberger and Sara Kropf.
\newblock On the higher dimensional quasi-power theorem and a berry-esseen
  inequality.
\newblock In {\em 27th International Conference on Probabilistic, Combinatorial
  and Asymptotic Methods for the Analysis of Algorithms Kraków, Poland, July
  4–8, 2016}, 2016.

\bibitem{hwang1998convergence}
Hsien-Kuei Hwang.
\newblock On convergence rates in the central limit theorems for combinatorial
  structures.
\newblock {\em European Journal of Combinatorics}, 19(3):329--343, 1998.

\bibitem{joyal1981theorie}
Andr{\'e} Joyal.
\newblock Une th{\'e}orie combinatoire des s{\'e}ries formelles.
\newblock {\em Advances in mathematics}, 42(1):1--82, 1981.

\bibitem{lalley1993finite}
Steven~P. Lalley.
\newblock Finite range random walk on free groups and homogeneous trees.
\newblock {\em The Annals of Probability}, pages 2087--2130, 1993.

\bibitem{Lescanne2013}
Pierre Lescanne.
\newblock On counting untyped lambda terms.
\newblock {\em Theoretical Computer Science}, 474:80--97, 2013.

\bibitem{NijenhuisWilf1978}
Albert Nijenhuis and Herbert~S. Wilf.
\newblock {\em Combinatorial Algorithms}.
\newblock Academic Press, 2 edition, 1978.

\bibitem{palka2012}
Michał~H. Pałka.
\newblock {\em Random Structured Test Data Generation for Black-Box Testing}.
\newblock PhD thesis, Chalmers University of Technology, 2012.

\bibitem{peytonJones1987}
Simon~L. Peyton~Jones.
\newblock {\em The Implementation of Functional Programming Languages}.
\newblock Prentice-Hall, Inc., 1987.

\bibitem{pivoteau2012algorithms}
Carine Pivoteau, Bruno Salvy, and Michele Soria.
\newblock Algorithms for combinatorial structures: Well-founded systems and
  newton iterations.
\newblock {\em Journal of Combinatorial Theory, Series A}, 119(8):1711--1773,
  2012.

\bibitem{tromp2006}
John Tromp.
\newblock Binary lambda calculus and combinatory logic.
\newblock In {\em Kol\-mo\-go\-rov Complexity and Applications}, volume 06051.
  Internationales Begegnungs- und Forschungszentrum für Informatik (IBFI),
  Schloss Dagstuhl, Germany, 2006.

\bibitem{Wang05generatingrandom}
Jue Wang.
\newblock Generating random lambda calculus terms.
\newblock Technical report, Boston University, 2005.

\bibitem{Wilf2006}
Herbert~S. Wilf.
\newblock {\em Generatingfunctionology}.
\newblock A. K. Peters, Ltd., 2006.

\bibitem{woods1997coloring}
Alan~R. Woods.
\newblock Coloring rules for finite trees, and probabilities of monadic second
  order sentences.
\newblock {\em Random structures and algorithms}, 10(4):453--485, 1997.

\end{thebibliography}
